\documentclass[11pt]{article}
\usepackage[margin=2cm]{geometry}

\usepackage{graphicx}
\usepackage{amsmath,amssymb,commath,mathtools,mathrsfs}
\usepackage{dsfont}
\usepackage{wasysym}
\usepackage{color}
\usepackage{bm}
\usepackage{algorithm, algpseudocode, algorithmicx} %
\algrenewcommand{\algorithmiccomment}[2][.25\linewidth]{%
  \leavevmode\hfill\makebox[#1][l]{$\triangleright$~#2}} 
\algtext*{EndWhile}
\algtext*{EndIf}
\algtext*{EndFor}
\algtext*{EndFunction}

\usepackage{amsthm,authblk}

\usepackage{paralist}

\usepackage{cases}

\usepackage{booktabs,multirow}

\usepackage[colorlinks=true,breaklinks=true,bookmarks=true,urlcolor=blue,
     citecolor=blue,linkcolor=blue,bookmarksopen=false,draft=false]{hyperref}

\DeclareMathOperator*{\argmax}{arg\,max}
\DeclareMathOperator*{\argmin}{arg\,min}

\renewcommand{\norm}[1]{\left\lVert#1\right\rVert}
\newcommand{\nvert}[1]{\lVert#1\rVert}
\newcommand{\svert}[1]{\lvert#1\rvert}

\newcommand{\bangle}[2]{\langle#1,#2\rangle}
\newcommand{\transpose}{\mathsf{T}}
\renewcommand{\square}{\Square}

\newcommand{\conv}{\mathrm{conv}}

\newcommand{\dom}{\mathrm{dom}\,}

\newcommand{\Reg}{\mathrm{R}}

\newcommand{\Prob}{\mathrm{Prob}}
\renewcommand{\square}{{\scriptsize\Square}}
\makeatletter
\def\ubar#1{\underline{\sbox\tw@{$#1$}\dp\tw@\z@\box\tw@}}
\makeatother

\makeatletter
\def\defcal#1{%
\expandafter\newcommand\csname cal#1\endcsname{\mathcal{#1}}}
\count@=0
\loop
\advance\count@ 1
\edef\y{\@Alph\count@}%
\expandafter\defcal\y
\ifnum\count@<26
\repeat
\def\defscr#1{%
\expandafter\newcommand\csname scr#1\endcsname{\mathscr{#1}}}
\count@=0
\loop
\advance\count@ 1
\edef\y{\@Alph\count@}%
\expandafter\defscr\y
\ifnum\count@<26
\repeat
\def\defbb#1{%
\expandafter\newcommand\csname bb#1\endcsname{\mathbb{#1}}}
\count@=0
\loop
\advance\count@ 1
\edef\y{\@Alph\count@}%
\expandafter\defbb\y
\ifnum\count@<26
\repeat
\def\defrm#1{%
\expandafter\newcommand\csname rm#1\endcsname{\mathrm{#1}}}
\count@=0
\loop
\advance\count@ 1
\edef\y{\@Alph\count@}%
\expandafter\defrm\y
\ifnum\count@<26
\repeat
\makeatother

\renewcommand{\phi}{\varphi}
\renewcommand{\epsilon}{\varepsilon}
\newcommand{\ulQ}{\underline{Q}}

\newcommand{\ulcQ}{\underline{\calQ}}
\newcommand{\olcQ}{\overline{\calQ}}

\newcommand{\UB}{\text{\textsc{UpperBound}}}
\newcommand{\LB}{\text{\textsc{LowerBound}}}
\newcommand{\NumEval}{\textup{\texttt{\#Eval}}}

\newcommand{\Eval}{\mathrm{Eval}}
\newcommand{\Nomin}{\mathrm{Nomin}}
\newcommand{\Robust}{\mathrm{Robust}}

\newcommand{\CVaR}{\mathrm{CVaR}}
\newcommand{\RWass}{\mathrm{RWass}}
\newcommand{\mean}{\mathrm{mean}}
\newcommand{\ext}{\mathrm{ext}\,}

\newtheorem{theorem}{Theorem}
\newtheorem{proposition}[theorem]{Proposition}
\newtheorem{corollary}[theorem]{Corollary}
\newtheorem{lemma}[theorem]{Lemma}
\newtheorem{definition}{Definition}
\newtheorem{assumption}{Assumption}
\newtheorem{example}{Example}
\theoremstyle{remark}
\newtheorem*{remark}{Remark}

\title{On Distributionally Robust Multistage Convex Optimization: Data-driven Models and Performance}
\author[1]{Shixuan Zhang}
\author[2]{Xu Andy Sun}
\date{}
\affil[1]{Wm Michael Barnes '64 Department of Industrial \& Systems Engineering, Texas A\&M University, \quad \url{shixuan.zhang@tamu.edu}}
\affil[2]{Operations Research and Statistics, Sloan School of Management, Massachusetts Institute of Technology, \quad \url{sunx@mit.edu}}

\begin{document}

\maketitle

\begin{abstract}
   This paper presents a novel algorithmic study with extensive numerical experiments of distributionally robust multistage convex optimization (DR-MCO).
   Following the previous work on dual dynamic programming (DDP) algorithmic framework for DR-MCO~\cite{zhang2020distributionally}, we focus on data-driven DR-MCO models with Wasserstein ambiguity sets that allow probability measures with infinite supports.
   These data-driven Wasserstein DR-MCO models have out-of-sample performance guarantees and adjustable in-sample conservatism.
   Then by exploiting additional concavity or convexity in the uncertain cost functions, we design exact single stage subproblem oracle (SSSO) implementations that ensure the convergence of DDP algorithms.
   We test the data-driven Wasserstein DR-MCO models against multistage robust convex optimization (MRCO), risk-neutral and risk-averse multistage stochastic convex optimization (MSCO) models on multi-commodity inventory problems and hydro-thermal power planning problems.
   The results show that our DR-MCO models could outperform MRCO and MSCO models when the data size is small.
\end{abstract}

\section{Introduction}

Multistage convex optimization is a decision-making problem where objective functions and constraints are convex and decisions need to be made each time some of the uncertainty information is revealed.
As it is often challenging to gain precise knowledge of the probability distributions of the uncertainty, distributionally robust multistage convex optimization (DR-MCO) allows ambiguities in probability distribution and aims to find optimal decisions that minimize the overall expected objective cost with respect to the worst-case distribution.
The DR-MCO framework encompasses multistage stochastic convex optimization (MSCO) and multistage robust convex optimization (MRCO) as special cases, and thus have been widely applied in many areas including energy systems, supply chain and inventory planning, portfolio optimization, and finance~\cite{shapiro2021lectures,bental2009robust}.

Distributionally robust optimization (DRO) has received significant research attention in recent years.
Common choices of the ambiguity sets include moment-based ambiguity sets~\cite{scarf1957min,delage2010distributionally,shapiro2004class,wiesemann2014distributionally}, discrepancy or distance-based ambiguity sets~\cite{bental2013robust,bayraksan2015data,rahimian2019identifying,bertsimas2018data}, and many other ones that are constructed from shape requirements or special properties of the distributions (we refer any interested reader to the review papers by~\cite{rahimian2019distributionally} and \cite{lin2022distributionally} for more details).
Among all these choices, data-driven Wasserstein ambiguity sets, i.e., Wasserstein distance balls centered at an empirical probability distribution obtained from given uncertainty data, has gained increasing popularity because of the following reasons:
(1) measure concentration in Wasserstein distance guarantees with high probability that the decisions or policies obtained from solving the in-sample model provides an upper bound on the out-of-sample mean cost~\cite{fournier2015rate,esfahani2018data}; and
(2) even if we consider general probability distributions in the ambiguity sets, tractable finite-dimensional reformulation or approximation can be derived using strong duality for 1-Wasserstein distance of probability measures on Euclidean spaces~\cite{esfahani2018data,zhao2018data,hanasusanto2018conic}, or more generally for any \(p\)-Wasserstein distance of probability measures on Polish spaces~\cite{gao2022distributionally}.
We remark that these reformulations or approximations are derived for single-stage or two-stage settings.
To the best of our knowledge, it remains unclear how to solve DR-MCO problems with Wasserstein ambiguity sets and infinite uncertainty sets.

There are also many works on DRO in the multistage settings.
In particular, if the ambiguity set only includes probability measures that are absolutely continuous with respect to a given reference measure, then such multistage DRO falls into the category of risk-averse multistage stochastic optimization~\cite{shapiro2021lectures,pichler2021mathematical}, which dates back to at least~\cite{eichhorn2005polyhedral}.
When the underlying uncertainty is stagewise independent, a random nested cutting plane algorithm, which is called stochastic dual dynamic programming (DDP) and is very often used for risk-neutral multistage stochastic linear optimization (MSLO) problems~\cite{pereira1991multi}, has been extended to risk-averse MSLO with polyhedral risk measures in~\cite{guigues2012sampling} and \cite{shapiro2013risk}.
DDP algorithms iteratively build under-approximations of the worst-case expected cost-to-go functions in the dynamic programming recursion, which lead to policies that minimize the approximate total cost starting from each stage.
To estimate the quality of the obtained policy, deterministic over-approximation is proposed in addition to the under-approximations for risk-averse MSLO problem in~\cite{philpott2013solving}.
Alternatively, for time-consistent conditional-value-at-risk (CVaR) risk measures, a new sampling scheme based on importance sampling is proposed to achieve tighter estimation of the policy quality~\cite{kozmik2015evaluating}.
Further exploiting the deterministic over-approximation, a deterministic version of the DDP algorithm is proposed for risk-averse MSLO in~\cite{baucke2018deterministic}.
As a variant, robust DDP is proposed for multistage robust linear optimization (MRLO) using similar deterministic over-approximation of cost-to-go functions in~\cite{georghiou2019robust}.
In~\cite{philpott2018distributionally}, stochastic DDP algorithm is used to solve distributionally robust multistage linear optimization (DR-MLO) using ambiguity sets defined by a modified \(\chi^2\)-distance for probability distributions supported on the historical data.
In~\cite{duque2020distributionally}, stochastic DDP algorithm is used to solve DR-MLO with Wasserstein ambiguity sets on finitely supported probability distributions with asymptotic convergence and promising out-of-sample performance.
We comment that the above solution approaches for risk-averse MSLO or DR-MLO rely either on sample average approximations of the risk measures or on the assumption that all uncertainties have finite support.
Recently, \cite{park2022data} propose a data-driven kernel regression-based DR-MLO framework that can handle Markovian uncertainty with asymptotic out-of-sample performance guarantee.
For any fixed number of samples, their out-of-sample suboptimality gap is bounded by $\calO(T^{3/2})$ where $T$ is the number of stages.
Our work, while focused on stagewise independent problems, has a strengthened bound of $\calO(T(\ln{T})^{1/2})$ in the sub-Gaussian case, and a matching bound $\calO(T^{3/2})$ in the finite third moment case.

As DDP-type algorithms are extensively applied in solving risk-neutral and risk-averse MSLO and DR-MLO, their convergence analysis and complexity study become central questions.
The finite time convergence of stochastic DDP algorithms is first proved for MSLO problems using polyhedrality of cost-to-go functions in~\cite{philpott2008convergence} and \cite{shapiro2011analysis}.
Such convergence is similarly proved for deterministic DDP algorithms for MSLO in~\cite{baucke2017deterministic}, for MRLO in~\cite{georghiou2019robust}, and for DR-MLO in~\cite{baucke2018deterministic}.
An asymptotic convergence of stochastic DDP algorithms is proved using monotone convergence argument in the space of convex cost-to-go function approximations for general risk-neutral multistage stochastic convex optimization (MSCO) in~\cite{girardeau2015convergence} and risk-averse MSCO in~\cite{guigues2016convergence}.
Complexity study of DDP algorithms is established using Lipschitz continuity of the under-approximations of the cost-to-go functions for risk-neutral MSCO in~\cite{lan2020complexity} and \cite{lan2022correction}, and for risk-neutral multistage mixed-integer nonlinear optimization in~\cite{zhang2022stochastic}.
Our recent work~\cite{zhang2020distributionally} adopts an abstract definition of single stage subproblem oracles (SSSO) and proves SSSO-based complexity bounds, hence also the convergence, for two DDP algorithms applied to DR-MCO problems with general uncertainty supports and ambiguity sets.

Following~\cite{zhang2020distributionally}, we aim to further study data-driven DR-MCO models and compare their numerical performance to other baselines models.
In particular, our paper makes the following contributions to the literature.
\begin{enumerate}
    \item We prove the out-of-sample performance guarantee using measure concentration results, adjustable in-sample conservatism assuming Lipschitz continuity of the value functions in the uncertainty variables for the data-driven DR-MCO models with Wasserstein ambiguity sets.
    \item We discuss implementations of SSSO in the context of Wasserstein ambiguity sets containing infinitely supported probability distributions.
    Such SSSO allows DDP algorithms to converge with provable complexity bounds to an \(\epsilon\)-global optimal first-stage solution, which could help minimize the influence from the sub-optimality caused by early termination of the algorithms when we compare model performances.
    \item We present extensive numerical experiments using multi-commodity inventory problem with either uncertain demands or uncertain prices, and hydro-thermal power planning problems with real-world data, that compares the out-of-sample performances of our DR-MCO models against risk-neutral and CVaR risk-averse MSCO models, as well as the MRCO model.
    To the best of our knowledge, these are the first numerical experiments in the literature that compare DR-MCO models with MSCO and MRCO models when the uncertainty has infinite support sets  such as the nonnegative orthant $\bbR_{\ge0}^n$.
\end{enumerate}

The rest of the paper is organized as follows.
In Section~\ref{sec:Model}, we define the DR-MCO model considered in this paper and show some of its favorable properties.
In Section~\ref{sec:Algorithm}, we review DDP algorithms for DR-MCO and study the implementation of SSSO for Wasserstein ambiguity sets.
In Section~\ref{sec:Numerical}, we present numerical experiment results comparing DR-MCO models against other baseline models on two application problems.
We provide some concluding remarks in Section~\ref{sec:Conclusion}.

\section{Data-driven Model and Its Properties}
\label{sec:Model}

\subsection{Data-driven Model Formulation}

In this section, we present a data-driven model for DR-MCO and some of its properties.
Let \(\calT:=\{1,\dots,T\}\) denote the set of stage indices.
In each stage \(t\in\calT\), we use \(\calX_t\subset\bbR^{d_t}\) to denote the convex state space and \(x_t\) its elements, which is known as the state vector.
We denote the set of uncertainties before stage \(t\) as \(\Xi_t\subseteq\bbR^{\delta_t}\) and its elements as \(\xi_t\).
For simplicity, we use the notation \(\calX_0=\{x_0\}\) and \(\Xi_1=\{\xi_1\}\) to denote parameter sets of the given initial state.
In each stage \(t\in\calT\setminus\{1\}\), the uncertainty \(\xi_t\) is assumed to be distributed according to an unknown probability measure \(p_t\) taken from a subset of the probability measures \(\calP_t\subset\calM^\Prob(\Xi_t)\) supported on $\Xi_t$.
The cost in each stage \(t\in\calT\) is described through a nonnegative, lower semicontinuous (lsc) local cost function \(f_t(x_{t-1},x_t;\xi_t)\), which is assumed to be convex in \(x_{t-1}\) and \(x_t\) for every \(\xi_t\in\Xi_t\).
We allow \(f_t\) to take \(+\infty\) to model constraints relating the states \(x_{t-1},x_t\) and the uncertainty \(\xi_t\), so that $\calX_t$ can be independent of $x_{t-1}$ and $\xi_t$.
The DR-MCO considered in this paper can be written as follows.
\begin{align}\label{eq:MDRO-ExtensiveForm}
    \inf_{x_1\in\calX_1}\quad f_1(x_0,x_1;\xi_1)\ +\;&\sup_{p_2\in\calP_2}\bbE_{\xi_2\sim p_2} \Bigg[\inf_{x_{2}\in\calX_2}f_{2}(x_1,x_{2};\xi_2) + \\
    &+\sup_{p_{3}\in\calP_3}\bbE_{\xi_3\sim p_3}\bigg[\inf_{x_{3}\in\calX_{3}}f_{3}(x_{2},x_{3};\xi_3)+\cdots\notag\\
    &\;\;\;\;\;\;+\sup_{p_{T}\in\calP_{T}}\bbE_{\xi_{T}\sim p_{T}}\Big[\inf_{x_{T}\in\calX_{T}}f_{T}(x_{T-1},x_{T};\xi_T)\Big]\cdots\bigg]\Bigg].\notag
\end{align}
Here, \(\bbE_{\xi_t\sim p_t}\) is the expectation with respect to variable \(\xi_t\) distributed according to the probability measure \(p_t\).
Note that we only consider uncertainty sets \(\Xi_t\) and ambiguity sets \(\calP_t\) that are independent between stages in this paper.
This is usually referred to as \emph{stagewise independence} and is commonly adopted in MSCO and MRCO literature~\cite{shapiro2011analysis,georghiou2019robust}.

\begin{example}[Constrained Formulation]\label{ex:ConstrainedFormulation}
    In the literature, a stage-$t$ subproblem is often given by
    \begin{equation}\label{eq:MDRO-ConstrainedFormulation}
        \begin{aligned}
            \min_{x_t,y_t}\quad & F_t(x_{t-1},y_t,x_t;\xi_t)\\
            \mathrm{s.t.}\quad & (x_{t-1},y_t,x_t,\xi_t)\in\calF_t,y_t\in\calY_t,
        \end{aligned}
    \end{equation}
    for some \emph{internal} variable $y_t\in\calY_t$, a lsc real-valued function $F_t$ that is convex in $(x_{t-1},y_t,x_t)$, and a closed convex constraint set $\calF_t$~\cite{ding2019python}.
    It can be reduced to single function used in~\eqref{eq:MDRO-ExtensiveForm} if we define
    \begin{equation*}
        f_t(x_{t-1},x_t;\xi_t):=\min_{y_t\in\calY_t}\ F_t(x_{t-1},y_t,x_t;\xi_t)+I_{\calF_t}(x_{t-1},y_t,x_t,\xi_t),
    \end{equation*}
    using the indicator function $I_{\calF_t}$ of the constraint set $\calF_t$, whose value is $0$ for any $(x_{t-1},y_t,x_t,\xi_t)\in\calF_t$, and $+\infty$ everywhere else.
    Such function $f_t$ is convex in $(x_{t-1}, x_t)$ because its epigraph is the coordinate projection of that of $F_t$ for every fixed $\xi_t$.
    Moreover, $f_t$ is lsc when the set of internal variables $\calY_t$ is compact, as is the case for all of our numerical examples in Section~\ref{sec:Numerical}.
\end{example}

Based on stagewise independence in the nested formulation, we can write the following recursion that is equivalent to~\eqref{eq:MDRO-ExtensiveForm} using the (worst-case expected) cost-to-go functions,
\begin{equation}\label{eq:MDRO-RecursiveForm}
    \calQ_{t-1}(x_{t-1}):=\sup_{p_t\in\calP_t}\bbE_{\xi_t\sim p_t}\bigg[\inf_{x_{t}\in\calX_{t}}f_{t}(x_{t-1},x_{t};\xi_{t})+\calQ_{t}(x_{t})\bigg],
\end{equation}
for each \(t\in\calT\) and we set by convention \(\calQ_T(x_T):=0\) for any \(x_T\in\calX_T\).
To simplify the notation, we also define the following value function for each stage \(t\in\calT\):
\begin{equation}\label{eq:MDRO-ValueFunction}
    Q_t(x_{t-1};\xi_{t}):=\inf_{x_{t}\in\calX_{t}}f_{t}(x_{t-1},x_{t};\xi_t)+\calQ_{t}(x_{t}).
\end{equation}
Using these value functions, we may write the optimal value of the DR-MCO~\eqref{eq:MDRO-ExtensiveForm} as \(Q_1(x_0;\xi_1)\) and further simplify the recursion~\eqref{eq:MDRO-RecursiveForm} as
\begin{equation}\label{eq:MDRO-SimplifiedCostToGo}
    \calQ_{t-1}(x_{t-1})=\sup_{p_t\in\calP_t}\int_{\Xi_t}Q_{t}(x_{t-1};\xi_t)\dif p_t(\xi_t).
\end{equation}

While there are many different choices of the ambiguity set \(\calP_t\) for each stage \(t\in\calT\) (see e.g.,~\cite{wiesemann2014distributionally}), we focus on the data-driven ambiguity sets constructed as follows.
Suppose we have the knowledge of \(n_t\) samples \(\hat{\xi}_{t,1},\dots,\hat{\xi}_{t,n_t}\) of the uncertainty \(\xi_t\).
The empirical probability measure is given by  \(\hat{\nu}_t:=\frac{1}{n_t}\sum_{k=1}^{n_t}\Delta_{\hat{\xi}_{t,k}}\), where for each \(k=1,\dots,n_t\), \(\Delta_{\hat{\xi}_{t,k}}\) is the Dirac probability measure supported at the point \(\hat{\xi}_{t,k}\in\Xi_t\), i.e., \(\int_{\Xi_t}f\dif\Delta_{\hat{\xi}_{t,k}}=f(\hat{\xi}_{t,k})\) for any compactly supported function \(f\) on \(\Xi_t\).
Such an empirical probability measure \(\hat{\nu}_t\) captures the information from the sample data and is often used to build the sample average approximation for multistage stochastic optimization~\cite{shapiro2021lectures}.

Fix any distance function \(d_t(\cdot,\cdot)\) on \(\Xi_t\), the Wasserstein distance of order 1 (a.k.a, Kantorovich-Rubinstein distance) is defined as
\begin{equation}\label{eq:WassersteinDistance}
    W_t(\mu,\nu):=\inf_{\pi\in\calM^\Prob(\Xi_t\times\Xi_t)}\left\{\int_{\Xi_t\times\Xi_t} d_t(\xi^1,\xi^2)\dif\pi(\xi^1,\xi^2):P^1_*(\pi)=\mu,\,P^2_*(\pi)=\nu\right\}, 
\end{equation}
for any two probability measures \(\mu,\nu\in\calM^\Prob(\Xi_t)\), where \(P^i_*(\pi)\) is the pushforward measure induced by the projection maps \(P^i:\Xi_t\times\Xi_t\to\Xi_t\) by sending \(P^i(\xi^1,\xi^2)=\xi^i\), for \(i=1\) or \(2\). That is, the joint probability measure \(\pi\) in~\eqref{eq:WassersteinDistance} has marginal probability measures equal to \(\mu\) and \(\nu\).

It can be shown that \(W_t\) is indeed a distance on the space of probability measures \(\calM^\Prob(\Xi_t)\)~\cite[Definition 6.1]{villani2008optimal} except that it may take the value of \(+\infty\).
Thus it is natural to restrict our attention to the convex subset of probability measures with finite distance to a Dirac measure on \(\Xi_t\)
\begin{equation}\label{eq:WassersteinSpace}
    \calW_t:=\left\{\mu\in\calM^\Prob(\Xi_t):\int_{\Xi_t}d_t(\bar{\xi},\xi)\dif\mu(\xi)<+\infty,\text{ for some }\bar{\xi}\in\Xi_t\right\}.
\end{equation}
Note that any continuous function \(g(\xi)\) that satisfies \(\svert{g(\xi)}\le C(1+d_t(\bar{\xi},\xi))\) for some \(C>0\) and \(\bar{\xi}\in\Xi_t\) would be integrable for any probability measure in \(\calW_t\).
Now given any such continuous functions \(g_{t,1},\dots,g_{t,m_t}\) on \(\Xi_t\) and a real vector \(\rho_t:=(\rho_{t,j})_{j=0}^{m_t}\in\bbR^{m_t+1}\), we define the Wasserstein ambiguity set \(\calP_t\) as
\begin{equation}\label{eq:WassersteinAmbiguitySet}
    \calP_t:=\left\{p\in\calW_t:W_t(p,\hat{\nu}_t)\le \rho_{t,0},\;\bangle{g_{t,j}}{p}\le\rho_{t,j},\;j=1,\dots,m_t\right\}.
\end{equation}
The first inequality constraint in the definition~\eqref{eq:WassersteinAmbiguitySet} bounds the Wasserstein distance of the probability measure \(p\in\calW_t\) from the empirical measure \(\hat{\nu}_t\), while the second set of constraints on \(p\) are defined as $\bangle{g_{t,j}}{p}:=\int_{\Xi_t} g_{t,j}(\xi)\dif p(\xi)$, which could be bounds on the moments.

It is well studied that the recursion for data-driven Wasserstein DR-MCO problems~\eqref{eq:MDRO-RecursiveForm} can be reformulated as a finite-dimensional minimum-supremum optimization problem, under the assumption of strict feasibility. 
\begin{assumption}\label{assum:NominalStrictFeasibility}
    The empirical probability measure \(\hat{\nu}_t\) is a strictly feasible solution of \eqref{eq:WassersteinAmbiguitySet}, i.e. it satisfies $\bangle{g_{t,j}}{\hat{\nu}_t}=\frac{1}{n_t}\sum_{k=1}^{n_t}g_{t,j}(\hat{\xi}_{t,k})<\rho_{t,j}$ for all \(j=1,\dots,m_t\).
\end{assumption}
\begin{theorem}\label{thm:FiniteDimensionalDual}
Let $d_{t,k}(\xi):=d_t(\xi,\hat{\xi}_k)$ denote the distance function to each sample $\hat{\xi}_k$ for each $k=1,\dots,n_t$.
Under Assumption~\ref{assum:NominalStrictFeasibility}, in any stage \(t\ge 2\), the expected cost-to-go function~\eqref{eq:MDRO-RecursiveForm} can be equivalently written as
\begin{align}\label{eq:MDRO-FiniteDimensionalRecursion}
    \calQ_{t-1}(x_{t-1})=
    &\min_{\lambda\ge0}\left\{\sum_{j=0}^{m_t}\rho_{t,j}\lambda_j+\frac{1}{n_t}\sum_{k=1}^{n_t}\sup_{\xi_k\in\Xi_t}\left\{Q_{t}(x_{t-1};\xi_k)-\lambda_0 d_{t,k}(\xi_k)-\sum_{j=1}^{m_t}\lambda_j g_{t,j}(\xi_k)\right\}\right\}.
\end{align}
\end{theorem}
We provide the derivation and proof details in Section~\ref{sec:DualRecursion}.
As a corollary, we can prove a special version of the Kantorovich-Rubinstein duality formula~\cite[Remark 6.5]{villani2008optimal}.

\begin{corollary}\label{cor:KRDualityFormula}
    Under Assumption~\ref{assum:NominalStrictFeasibility}, if the value function \(Q_t(x_{t-1};\xi_t)\) is \(l_t\)-Lipschitz continuous in the uncertainty \(\xi_t\in\Xi_t\) for any \(x_{t-1}\in\calX_{t-1}\), then we have
    \[
        \calQ_{t-1}(x_{t-1})\le \rho_{t,0}l_t+\frac{1}{n_t}\sum_{k=1}^{n_t} Q_t(x_{t-1};\hat{\xi}_{t,k}).
    \]
\end{corollary}
\begin{proof}{Proof}
    Take a feasible solution \(\lambda_0=l_t\) and \(\lambda_j=0\) for \(j=1,\dots,m_t\) in~\eqref{eq:MDRO-FiniteDimensionalRecursion} of Theorem~\ref{thm:FiniteDimensionalDual}.
    By the Lipschitz continuity assumption, the supremum is attained at \(\hat{\xi}_{t,k}\) for each \(k=1,\dots,n_t\). 
\end{proof}

\subsection{Out-of-Sample Performance Guarantee}

A major motivation for using Wasserstein DR-MCO models is the out-of-sample performance guarantee, which ensures that the decisions evaluated on the true probability distribution would perform no worse than the in-sample training with high probability.
To begin with, we say that a probability measure \(\mu\in\calW_t\) is sub-Gaussian if \(\sup_{u:\nvert{u}=1}\int_{\Xi_t}\exp(C(u^\transpose\xi)^2)\dif\mu(\xi)<+\infty\) for some constant \(C>0\)~\cite[Chapter 2.5]{vershynin2018high}; or it has finite third moments if \(\int_{\Xi_t}\nvert{\xi}^3\dif\mu(\xi)<+\infty\).
Our discussion is based on the following version of concentration inequality.

\begin{theorem}[{\cite[Theorem 2]{fournier2015rate}}]\label{thm:MeasureConcentrationWasserstein}
    Fix any probability measure \(\nu_t\in\calM^\Prob(\Xi_t)\) in stage \(t\) and let \(\hat{\nu}_t\) denote the empirical measure constructed from \(n_t\) independent and identically distributed (iid) samples of \(\nu_t\).
    Then the probability $\bbP\bigl(W_t(\nu_t,\hat{\nu}_t)>\rho_{t,0}\bigr)$ is bounded from above by
        \begin{numcases}{C_t\exp\Bigl(-C'_t n_t\rho_{t,0}^{\delta_t}\Bigr)\mathbf{1}_{\le1}(\rho_{t,0})+}
        C_t\exp\Bigl(-C'_t n_t\rho_{t,0}^2\Bigr)\mathbf{1}_{>1}(\rho_{t,0}), & if \(\nu_t\) is sub-Gaussian and \(\delta_t\ge 3\), \label{eq:MeasureConcentrationSubGaussian}\\
        C''_t(n_t\rho_{t,0}^2)^{-1}, & if \(\nu_t\) has finite third moments,\label{eq:MeasureConcentrationFinite3rdMoments}
    \end{numcases}
    where $\mathbf{1}_{\le1},\mathbf{1}_{>1}$ are indicator functions for intervals $(0,1]$ and $(1,+\infty)$, and the constants \(C_t,C'_t,C''_t>0\) depend only on \(\nu_t\).
\end{theorem}

The measure concentration bound in~\eqref{eq:MeasureConcentrationSubGaussian} becomes slightly more intricate when the dimension of the uncertainty \(\delta_t\le2\) (see the details in~\cite{fournier2015rate}), so we focus our discussion below on the other cases.

The out-of-sample performance refers to the evaluation of the solutions and policies obtained from solving DR-MCO~\eqref{eq:MDRO-ExtensiveForm} on the true probability measures \(\nu_t\) for each \(t\in\calT\).
To be precise, fix any optimal policy given by the DR-MCO~\eqref{eq:MDRO-ExtensiveForm}, i.e., an optimal initial stage state \(x_1^*\in\argmin_{x_1\in\calX_1}\{f_1(x_0,x_1;\xi_1)+\calQ_1(x_1)\}\), and a collection of mappings \(x_t^*:\calX_{t-1}\times\Xi_t\to\calX_t\) for \(t=2,\dots,T\), such that 
\begin{enumerate}
    \item \(x_t^*(x_{t-1},\cdot)\) is Borel measurable for any \(x_{t-1}\in\calX_{t-1}\);
    \item \(x_t^*(x_{t-1},\xi_t)\in\argmin_{x_t\in\calX_t}\{f_t(x_{t-1},x_t;\xi_t)+\calQ_t(x_t)\}\) for any \(x_{t-1}\in\calX_{t-1}\) and \(\xi_t\in\Xi_t\).
\end{enumerate}
We define recursively the costs in the out-of-sample evaluation cost-to-go functions as
\begin{equation}\label{eq:OutOfSampleEvaluation-Recursion}
    \calQ_{t-1}^\Eval(x_{t-1}):=\int_{\Xi_t}\bigl[f_t(x_{t-1},x_t^*(x_{t-1},\xi_t);\xi_t)+\calQ_t^\Eval(x_t^*(x_{t-1},\xi_t))\bigr]\dif\nu_t(\xi_t),
\end{equation}
with \(\calQ_T^\Eval(x_T)=0\) for any \(x_T\in\calX_T\).
Then the out-of-sample evaluation mean cost associated with this policy is defined as \(v^\Eval:=f_1(x_0,x_1^*;\xi_1)+\calQ_1^\Eval(x_1^*)\).
The next theorem provides a lower bound on the probability that the event \(v^\Eval\le Q_1(x_0;\xi_1)\) happens, which is often known as the out-of-sample performance guarantee.

\begin{theorem}\label{thm:OutOfSamplePerformance}
    Fix any probability measure \(\nu_t\in\calW_t\) and let \(\hat{\nu}_t\) denote the empirical measure from \(n_t\) iid samples of \(\nu_t\) for all stages \(t\in\calT\). 
    Assume that \(\bangle{g_{t,j}}{\nu_t}\le\rho_{t,j}\) for \(j=1,\dots,m_t\) and \(t\in\calT\).
    Then for any \(\alpha\in(0,1)\), we have \(v^\Eval\le Q_1(x_0;\xi_1)\) with probability at least \(\alpha\) if either of the following conditions holds for each \(t\in\calT\):
    \begin{enumerate}
        \item the probability measure \(\nu_t\) is sub-Gaussian, $\delta_t\ge 3$, and 
            \[n_t\cdot\min\{\rho_{t,0}^{\delta_t},\rho_{t,0}^2\}\ge\frac{1}{C'_t}\left[\ln{C_t}-\ln\left(1-\alpha^{1/(T-1)}\right)\right],\]
        \item the probability measure \(\nu_t\) has finite third moments, $\delta_t\ge2$, and 
            \[n_t\cdot\min\{\rho_{t,0}^{\delta_t},\rho_{t,0}^2\}\ge \max\left\{\frac{1}{C'_t}\left[\ln{2C_t}-\ln\left(1-\alpha^{1/(T-1)}\right)\right],\frac{2C''_t}{1-\alpha^{1/(T-1)}}\right\},\]
    \end{enumerate}
    where \(C_t,C'_t,\) and \(C''_t\) are the positive constants in Theorem~\ref{thm:MeasureConcentrationWasserstein}, that depend only on \(\nu_t\).
    For fixed $\alpha$, the right-hand sides of conditions 1 and 2 are $\calO(\ln{T})$ and $\calO(T)$, respectively, as $T\to\infty$.
\end{theorem}
\begin{proof}
    We first verify that if either of the conditions is satisfied, then the probability \(\bbP(W_t(\nu_t,\hat{\nu}_t)>\rho_{t,0})\le 1-\alpha^{1/(T-1)}\).
    For condition 1 (sub-Gaussian case),
    \begin{itemize}
        \item when $\rho_{t,0}\le1$, the condition reduces to 
            \[
                n_t\rho_{t,0}^{\delta_t}\ge\frac{1}{C'_t}\left[\ln{C_t}-\ln\left(1-\alpha^{1/(T-1)}\right)\right],
            \]
            which is equivalent to $C_t\exp(-C'_tn_t\rho_{t,0}^{\delta_t})\le 1-\alpha^{1/(T-1)}$;
        \item when $\rho_{t,0}>1$, the condition reduces to
            \[
                n_t\rho_{t,0}^2\ge\frac{1}{C'_t}\left[\ln{C_t}-\ln\left(1-\alpha^{1/(T-1)}\right)\right],
            \]
            which is equivalent to $C_t\exp(-C'_tn_t\rho_{t,0}^2)\le 1-\alpha^{1/(T-1)}$.
    \end{itemize}
    In both situations, by~\eqref{eq:MeasureConcentrationSubGaussian} in Theorem~\ref{thm:MeasureConcentrationWasserstein}, we know that $\bbP(W_t(\nu_t,\hat{\nu}_t)>\rho_{t,0})\le 1-\alpha^{1/(T-1)}$.\\
    For condition 2 (finite third-moment case),
    \begin{itemize}
        \item when $\rho_{t,0}\le1$, the condition implies
            \[
                n_t\rho_{t,0}^{\delta_t}\ge\frac{1}{C'_t}\left[\ln{2C_t}-\ln\left(1-\alpha^{1/(T-1)}\right)\right],
            \]
            which is equivalent to $C_t\exp(-C'_tn_t\rho_{t,0}^{\delta_t})\le \frac{1-\alpha^{1/(T-1)}}{2}$, and
            \[
                n_t\rho_{t,0}^{2}\ge n_t\rho_{t,0}^{\delta_t}\ge\frac{2C''_t}{1-\alpha^{1/(T-1)}},
            \]
            which is equivalent to $C''(n_t\rho_{t,0}^2)^{-1}\le \frac{1-\alpha^{1/(T-1)}}{2}$, so the sum in~\eqref{eq:MeasureConcentrationFinite3rdMoments} is bounded by $1-\alpha^{1/(T-1)}$;
        \item when $\rho_{t,0}>1$, the condition implies
            \[
                n_t\rho_{t,0}^{2}\ge\frac{2C''_t}{1-\alpha^{1/(T-1)}},
            \]
            which is equivalent to $C''_t(n_t\rho_{t,0}^2)^{-1}\le \frac{1-\alpha^{1/(T-1)}}{2}<1-\alpha^{1/(T-1)}$, so by~\eqref{eq:MeasureConcentrationFinite3rdMoments} we have the desired bound.
    \end{itemize}
    By the assumption on the iid sampling of \(\hat{\nu}_t\) and that \(\bangle{g_{t,j}}{\nu_t}\le\rho_{t,j}\) for \(j=1,\dots,m_t\), the event \(E:=\{\nu_t\in\calP_t\text{ for }t=2,\dots,T\}\) has the probability
    \[
        \bbP(E)=\bbP\{W_t(\nu_t,\hat{\nu}_t)\le\rho_{t,0}\text{ for all }t\in\calT\}= \prod_{t=2}^T\bbP(W_t(\nu_t,\hat{\nu}_t)\le\rho_{t,0})\ge\alpha.
    \]
    To show that $v^{\Eval}\le Q_1(x_0;\xi_1)$ with probability $\alpha$, we claim that \(\calQ_t^\Eval(x_t)\le\calQ_t(x_t)\) for all \(x_t\in\calX_t\) and \(t\in\calT\) everywhere on \(E\).
    Note that \(\calQ_T^\Eval(x_T)=\calQ_T(x_T)=0\) for all \(x_T\in\calX_T\).
    Now if \(\calQ_t^\Eval(x_t)\le\calQ_t(x_t)\) for all \(x_t\in\calX_t\), then
    \begin{align*}
        f_t(x_{t-1},x_t^*(x_{t-1},\xi_t);\xi_t)+\calQ_t^\Eval(x_t^*(x_{t-1},\xi_t))
        &\le f_t(x_{t-1},x_t^*(x_{t-1},\xi_t);\xi_t)+\calQ_t(x_t^*(x_{t-1},\xi_t))\\
        & = \min_{x_t\in\calX_t}\{f_t(x_{t-1},x_t;\xi_t)+\calQ_t(x_t)\}
        = Q_t(x_{t-1};\xi_t).
    \end{align*}
    Therefore, on the event \(E\), we have
    \begin{align*}
        \calQ_{t-1}^\Eval(x_{t-1})
        &=\int_{\Xi_t}\bigl[f_t(x_{t-1},x_t^*(x_{t-1},\xi_t);\xi_t)+\calQ_t^\Eval(x_t^*(x_{t-1},\xi_t))\bigr]\dif\nu_t(\xi_t)\\
        &\le\int_{\Xi_t}Q_t(x_{t-1};\xi_t)\dif\nu_t(\xi_t)
        \le\sup_{p_t\in\calP_t}\int_{\Xi_t}Q_t(x_{t-1};\xi_t)\dif p_t(\xi_t) = \calQ_{t-1}(x_{t-1}).
    \end{align*}
    This recursion shows that \(\calQ_1^\Eval(x_1)\le\calQ_1(x_1)\) for all \(x_1\in\calX_1\), and thus 
    \[
        v^\Eval=f_1(x_0,x_1^*;\xi_1)+\calQ_1^\Eval(x_1^*)\le \min_{x_1\in\calX_1}\{f_1(x_0,x_1;\xi_1)+\calQ_1(x_1)\}=Q_1(x_0;\xi_1).
    \]
    Finally for the claim on the growth rates, it suffices to check that
    \[
        \frac{1}{1-\alpha^{1/(T-1)}}\le \frac{(T-1)}{1-\alpha} \iff (T-1)\bigl(1-\alpha^{1/(T-1)}\bigr)+\alpha\ge1,
    \]
    for any $T\ge2$ and $\alpha\in(0,1)$.
    Make the substitution $S:=T-1$ and $\beta=\alpha^{1/S}$ which leaves us to check
    \[
        H_S(\beta):=S(1-\beta)+\beta^S\ge1,\quad\forall S\ge1,\beta\in(0,1).
    \]
    Note that $H_S$ is continuously differentiable with respect to $\beta$ on $[0,1]$ the derivative $H_S'(\beta)=-S+S\beta^{S-1}<0$ for $\beta\in(0,1)$ and $S\ge 1$.
    Thus it is a non-increasing function on $[0,1]$ and $H_S(1)=1^S=1$, which shows that $H_S(\beta)\ge1$ for all $\beta\in[0,1]$.
\end{proof}

We remark that the minimum of $\rho_{t,0}^{\delta_t}$ and $\rho_{t,0}^{2}$ in condition 1, and the maximum in condition 2, are due to the two regimes $\rho_{t,0}\le 1$ (converging to zero) and $\rho_{t,0}>1$ (bounded away from zero) in Theorem~\ref{thm:MeasureConcentrationWasserstein}.
Theorem~\ref{thm:OutOfSamplePerformance} shows 
the multistage traits of our DR-MCO problem.
For fixed $\alpha$, we need to increase $n_t$ or $\rho_{t,0}$ to obtain the guarantee as $T$ grows.
In particular, when the true probability measures are sub-Gaussian, condition 1's right-hand side (RHS) is $\calO(\ln(T))$ as $T$ grows.
Then, one may want to increase $\rho_{t,0}$ as $\sqrt{\ln(T)}$ or to increase $n_t$ as $\ln(T)$. When the true probability measures are not sub-Gaussian but have finite third moments (e.g., the lognormal probability measures in Section~\ref{sec:Numerical:Hydrothermal}), condition 2's RHS is $\calO(T)$. 
Hence, one may want to increase $\rho_{t,0}$ with $\sqrt{T}$, as it may be difficult to increase $n_t$ linearly with $T$.
It is worth noting that, assuming $n_t$ is fixed, the required growth of $\rho_{t,0}$ (after it exceeds $1$) does not depend on the uncertainty dimensions $\delta_t$ in both cases.  
Finally, since the true probability measures are unknown, so are the constants $C_t,C'_t$, and $C''_t$ in Theorem~\ref{thm:OutOfSamplePerformance} and thus we determine $\rho_{t,0}$ heuristically in the numerical experiments in Section~\ref{sec:Numerical}.

\subsection{Adjustable In-Sample Policy Conservatism}

Another well studied approach to guarantee out-of-sample performance is the multistage robust convex optimization (MRCO) model~\cite{georghiou2019robust,zhang2020distributionally}.
However, MRCO considers only the worst-case outcomes of the uncertainties and thus can be overly conservative.
To be precise, we define the nominal MSCO from the empirical measures \(\hat{\nu}_t\) without risk aversion by the following recursion
\begin{equation}\label{eq:MNSO-RecursiveForm}
    \calQ_{t-1}^\Nomin(x_{t-1}):=\int_{\Xi_t}Q_t^\Nomin(x_{t-1};\xi_t)\dif\hat{\nu}_t(\xi_t)=\frac{1}{n_t}\sum_{k=1}^{n_t}Q_t^\Nomin(x_{t-1};\hat{\xi}_{t,k}),
\end{equation}
where
\begin{equation}\label{eq:MNSO-ValueFunction}
    Q_t^\Nomin(x_{t-1};\xi_{t}):=\inf_{x_{t}\in\calX_{t}}f_{t}(x_{t-1},x_{t};\xi_t)+\calQ_{t}^\Nomin(x_{t}),
\end{equation}
and \(\calQ_T^\Nomin(x_T)=0\) for any \(x_T\in\calX_T\).
Any MRCO that is built directly from data could have a much larger optimal cost than the nominal MSCO with a probability growing with the numbers of samples \(n_t\) and stages \(T\), as illustrated by the following example.

\begin{example}
    Consider an MSCO with local cost functions \(f_t(x_{t-1},x_t;\xi_t):=x_t+\xi_t\) and state spaces \(\calX_t:=[0,1]\subseteq\bbR\) for all \(t\in\calT\).
    For each \(t\ge 2\), the uncertainties are described by a probability measure \(\nu_t\) on the set \(\Xi_t:=\bbR_{\ge0}\) such that for any \(C>0\), we have \(\nu_t((C,+\infty))>0\).
    Now to approximate this MSCO, suppose we are given iid samples \(\hat{\xi}_{t,1},\dots,\hat{\xi}_{t,n_t}\) from the probability measure \(\nu_t\) in each stage \(t\ge 2\).
    Then any MRCO defined by the following recursion of expected cost-to-go functions
    \[
        \calQ_{t-1}^\Robust(x_{t-1}):=\sup_{\xi_t\in\hat{\Xi}_t}\min_{x_t\in\calX_t}\{f_t(x_{t-1},x_t;\xi_t)+\calQ_t^\Robust(x_t)\},\quad \forall\,t\ge2,
    \]
    where \(\hat{\Xi}_t\supseteq\{\hat{\xi}_{t,1},\dots,\hat{\xi}_{t,n_t}\}\) is an uncertainty subset constructed from the data, would have its optimal value \(Q_1^\Robust(x_0;\xi_1)\ge x_0+\xi_1+\sum_{t=2}^T\hat{\xi}_t^{\max}\), where \(\hat{\xi}_t^{\max}:=\max\{\hat{\xi}_{t,1},\dots,\hat{\xi}_{t,n_t}\}\).
    The optimal value of the corresponding nominal MSCO is \(Q_1^\Nomin(x_0;\xi_1)=x_0+\xi_1+\sum_{t=2}^T\hat{\xi}_t^{\mean}\), where \(\hat{\xi}_t^{\mean}:=\frac{1}{n_t}\sum_{k=1}^{n_t}\hat{\xi}_{t,k}\).
    Therefore, for any constant \(C>\max_{t=2,\dots,T}\bbE[\xi_t]\), we can show that
    \begin{align*}
        \bbP\Big\{Q_1^\Robust(x_0;\xi_1)-Q_1^\Nomin(x_0;\xi_1)>C\Big\}
        &= 1-\prod_{t=2}^T\Big((\nu_t[0,2C])^{n_t}+\bbP\big\{\hat{\xi}_t^{\mean}> C\big\}\Big),
    \end{align*}
    which goes to 1 as \(n_t\to\infty\).
\end{example}

The example suggests for MRCO models to be moderately conservative, it is sometimes essential to disregard a subset of the data samples, which, however, causes difficulty in ensuring the out-of-sample performance guarantee.
In contrast, the difference of the optimal values between the DR-MCO and the nominal MSCO \(Q_1(x_0;\xi_1)-Q_1^\Nomin(x_0;\xi_1)\), and consequently those between the DR-MCO and the true MSCO $Q_1(x_0;\xi_1)-v^\Eval$ can be bounded by the Wasserstein distances \(\rho_{t,0}\), as shown below. 

\begin{theorem}\label{thm:InSampleConservatism}
    Suppose the value function \(Q_t(x_{t-1};\xi_t)\) is \(l_t\)-Lipschitz continuous in \(\xi_t\) for any \(x_{t-1}\in\calX_{t-1}\), for each stage \(t\in\calT\).
    Under Assumption~\ref{assum:NominalStrictFeasibility}, the difference of optimal values between the DR-MCO and the nominal MSCO satisfies
    \[
        Q_1(x_0;\xi_1)-Q_1^\Nomin(x_0;\xi_1)\le\sum_{t=2}^{T}l_t\rho_{t,0}.
    \]
    Consequently, if $Q_t^\Nomin(x_{t-1};\xi_t)$ is also $l_t$-Lipschitz continuous in $\xi_t$ and $\rho_{t,0}$ satisfies either condition in Theorem~\ref{thm:OutOfSamplePerformance}, then with probability $\alpha\in(0,1)$, the difference of optimal values between the DR-MCO and the true MSCO satisfies
    \[
        Q_1(x_0;\xi_1)-v^\Eval\le 2\sum_{t=2}^{T}l_t\rho_{t,0}.
    \]
\end{theorem}
\begin{proof}
    We first observe that if there exists \(\epsilon\ge 0\) such that \(\sup_{x_t\in\calX_t}\svert{\calQ_t(x_t)-\calQ_t^{\Nomin}(x_t)}\le \epsilon\), then the value functions satisfy \(Q_t(x_{t-1};\xi_t)-Q_t^{\Nomin}(x_{t-1};\xi_t)\le\epsilon\) by the definitions~\eqref{eq:MDRO-ValueFunction} and~\eqref{eq:MNSO-ValueFunction}.
    Now we prove by recursion that \(\calQ_t(x_t)-\calQ_t^\Nomin(x_t)\le\sum_{s>t}l_s\rho_{s,0}\) for any \(x_t\in\calX_t\), which holds trivially for \(t=T\).
    For any \(t\in\calT\), we have
    \begin{align*}
        \calQ_{t-1}(x_{t-1})-\calQ_{t-1}^{\Nomin}(x_{t-1})
        =&
        \left(\sup_{p_t\in\calP_t}\int_{\Xi_t}Q_t(x_{t-1};\xi_t)\dif p_t(\xi_t)-\int_{\Xi_t}Q_t(x_{t-1};\xi_t)\dif\hat{\nu}_t(\xi_t)\right)\\
        +&\int_{\Xi_t}\left(Q_t(x_{t-1};\xi_t)-Q_t^{\Nomin}(x_{t-1};\xi_t)\right)\dif\hat{\nu}_t(\xi_t)\\
        \
        \le & l_t\rho_{t,0} +\sum_{s>t}l_s\rho_{s,0}
        =\sum_{s>t-1}l_s\rho_{s,0},
    \end{align*}
    where the first part before the inequality is bounded by \(l_t\rho_{t,0}\) using Corollary~\ref{cor:KRDualityFormula}, and the second part is bounded by our observation above.
    This recursion shows that \(\calQ_1(x_1)-\calQ_1^\Nomin(x_1)\le\sum_{t=2}^T l_t\rho_{t,0}\), hence the first part.
    For the second part, we do the same recursion by assuming $\svert{\calQ_t^\Eval(x_t)-\calQ_t^\Nomin(x_t)}\le\sum_{s>t}l_s\rho_{s,0}$ for any $x_t\in\calX_t$, which is true for $t=T$ due to Theorem~\ref{thm:OutOfSamplePerformance}.
    Then for any $t\in\calT$, we have
    \begin{align*}
        \svert{\calQ_{t-1}^{\Eval}(x_{t-1})-\calQ_{t-1}^{\Nomin}(x_{t-1})}
        \le&
        \left\vert\int_{\Xi_t}Q_t^\Nomin(x_{t-1};\xi_t)\dif\nu_t(\xi_t)-\int_{\Xi_t}Q_t^\Nomin(x_{t-1};\xi_t)\dif\hat{\nu}_t(\xi_t)\right\vert\\
        +&\int_{\Xi_t}\left\vert Q_t^\Eval(x_{t-1};\xi_t)-Q_t^{\Nomin}(x_{t-1};\xi_t)\right\vert\dif\nu_t(\xi_t)\\
        \
        \le & l_t\rho_{t,0} +\sum_{s>t}l_s\rho_{s,0}
        =\sum_{s>t-1}l_s\rho_{s,0},
    \end{align*}
    where $\left\vert\int_{\Xi_t}Q_t^\Nomin(x_{t-1};\xi_t)\dif\nu_t(\xi_t)-\int_{\Xi_t}Q_t^\Nomin(x_{t-1};\xi_t)\dif\hat{\nu}_t(\xi_t)\right\vert\le l_t\rho_{t,0}$ is a consequence of $W_t(\hat{\nu}_t,\nu_t)\le\rho_{t,0}$ by the Kantorovich-Rubinstein duality~\cite[Remark 6.5]{villani2008optimal}.
    Finally, we have 
    \[
    \calQ_1(x_1)-\calQ_1^\Eval(x_1)\le \calQ_1(x_1)-\calQ_1^\Nomin(x_1)+\svert{\calQ_1^\Eval(x_1)-\calQ_1^\Nomin(x_1)}\le 2\sum_{t=2}^{T}l_t\rho_{t,0},
    \]
    which completes the proof by the same observation at the beginning of the proof.
\end{proof}

Theorem~\ref{thm:InSampleConservatism} shows that the conservatism of the DR-MCO can be adjusted linearly with the Wasserstein distance bound \(\rho_{t,0}\), assuming the Lipschitz continuity of the value functions in the uncertainty.
Assuming that $Q_t(x_{t-1};\xi_t)$ are uniformly Lipschitz continuous with respect to $\xi_t$ for all $t\in\calT$, Theorems~\ref{thm:OutOfSamplePerformance} and~\ref{thm:InSampleConservatism} imply that an optimal choice of the Wasserstein radius, i.e., the smallest radius $\rho_{t,0}$ satisfying the conditions in Theorem~\ref{thm:OutOfSamplePerformance}, will lead to a difference of $\calO(T(\ln{T})^{1/2})$ (the sub-Gaussian case) or $\calO(T^{3/2})$ (the finite third moment case) between the optimal value of our DR-MCO model and that of the true MSCO model.

\section{Dual Dynamic Programming Algorithm}
\label{sec:Algorithm}

In this section, we focus on different realizations of the single stage subproblem oracles (SSSO) for DR-MCO with Wasserstein ambiguity sets that would guarantee the convergence of the dual dynamic programming (DDP) algorithms.

\subsection{Review of Dual Dynamic Programming Algorithms}
\label{sec:Algorithm:DDPReview}

Recall that for any convex function \(\calQ:\calX\to\bbR\cup\{+\infty\}\), an affine function \(\calV:\calX\to\bbR\) is called a (valid) linear cut if \(\calQ(x)\ge\calV(x)\) for all \(x\in\calX\).
A collection of such valid linear cuts \(\{\calV^j\}_{1\le j\le i}\) defines a valid under-approximation \(\ulcQ^i(x):=\max_{1\le j\le i}\calV^j(x)\) of \(\calQ(x)\).
Similarly by convexity, given a collection of overestimate values \(v^j\ge\calQ(x^j)\) for \(j=1,\dots,i\), we can define a valid over-approximation by the convex envelope \(\olcQ^i(x):=\conv_{1\le j\le i}(v^j+\iota_{x^j}(x))\), where \(\iota_{x^j}(x)=0\) when \(x=x^j\) and \(+\infty\) otherwise, is the convex indicator function centered at \(x^j\).
The validness of these approximations \(\ulcQ(x)\le\calQ(x)\le\olcQ(x)\) for all \(x\in\calX\) suggests that we may use them in the place of \(\calQ\) for recursive updates during a stagewise decomposition algorithm.

To see how the recursive updates may work, let us assume temporarily in this subsection that the worse-case probability measure \(p_t^*\) in~\eqref{eq:MDRO-SimplifiedCostToGo} exists and can be found.
Given any under-approximation \(\ulcQ_t\) for \(\calQ_t\), for any \(\hat{\xi}_t\in\Xi_t\), we can generate a linear cut \(V_t(x_{t-1};\xi_t):=\ubar{v}(\hat{\xi}_t)+u(\hat{\xi}_t)^\transpose(x_{t-1}-\tilde{x}_{t-1})\) for the value function \(Q_t(x_{t-1};\hat{\xi}_t)\) with some \(\tilde{x}_{t-1}\in\calX_{t-1}\), where \(\ubar{v}(\hat{\xi}_t)\) and \(u(\hat{\xi}_t)\) are the optimal value and an optimal dual solution to the following Lagrangian dual of the under-approximation problem that is parametrized by \(\hat{\xi}_t\):
\begin{equation}\label{eq:RecursiveUnderApproximation}
    \ubar{v}(\hat{\xi}_t):=\sup_{u\in\bbR^{d_{t-1}}}\min\left\{f_t(z_t,x_t;\hat{\xi}_t)+\ulcQ_t(x_t)+u^\transpose(\tilde{x}_{t-1}-z_t):x_t\in\calX_t,z_t\in\bbR^{d_{t-1}}\right\},
\end{equation}
assuming that \(u(\hat{\xi}_t)\) exists.
Then we can aggregate these linear cuts \(V_t(x_{t-1};\xi_t)\) into a linear cut \(\calV_{t-1}(\cdot):=\bbE_{\xi_t\sim p_t^*}V_t(\cdot;\xi_t)\) for \(\calQ_{t-1}\), where the expectation is taken componentwise with respect to the probability measure \(p_t^*\).
In this way, we can use the under-approximation \(\ulcQ_t\) for \(\calQ_t\) to update the under-approximation \(\ulcQ_{t-1}\) for \(\calQ_{t-1}\) by the aggregated linear cut \(\calV_{t-1}\).

Likewise, given any over-approximation \(\olcQ_t\) for \(\calQ_t\), for any \(\hat{\xi}_t\in\Xi_t\) and \(\tilde{x}_{t-1}\in\calX_{t-1}\), we can solve the following over-estimation problem
\begin{equation}\label{eq:RecursiveOverApproximation}
    \bar{v}(\hat{\xi}_t):=\min\left\{f_t(\tilde{x}_{t-1},x_t;\hat{\xi}_t)+\olcQ_t(x_t):x_t\in\calX_t\right\},
\end{equation}
which gives an overestimate value \(\bar{v}(\hat{\xi}_t)\ge Q_t(\tilde{x}_{t-1};\hat{\xi}_t)\).
Again we can aggregate the overestimate value by setting \(v_{t-1}:=\bbE_{\xi_t\sim p_t^*}\bar{v}(\xi_t)\), which by definition satisfies \(v_{t-1}\ge\calQ_{t-1}(\tilde{x}_{t-1})\) and thus can be used to update the over-approximation \(\olcQ_{t-1}\) for \(\calQ_{t-1}\).

There are, however, some potential issues with this recursive approximation method.
First, the supremum in the Lagrangian dual problem~\eqref{eq:RecursiveUnderApproximation} may not be attained, which could happen if \(\ubar{v}(\hat{\xi}_t)=+\infty\) (i.e., \(\tilde{x}_{t-1}\) is an infeasible state for \(\hat{\xi}_t\)) or if any neighborhood of \(\tilde{x}_{t-1}\) contains such an infeasible state.
In this case, we may fail to generate a linear cut \(V_t(\cdot;\hat{\xi}_t)\).
Second, Lipschitz constants of the linear cuts \(\calV_{t-1}\) may be affected by the under-approximation \(\ulcQ_t\), causing a worse approximation quality.
This may happens when \(\tilde{x}_{t-1}\) is an extreme point of \(\calX_{t-1}\).
In fact, it is shown in~\cite{zhang2020distributionally} that the Lipschitz constants of \(\ulcQ_t\) could exceed those of \(\calQ_t\) and grow with the total number of stages \(T\).
Third, the over-approximation function evaluates to \(+\infty\) at any point that is not in the convex hull of previously visited points, which makes the gap \(\olcQ_t(x)-\ulcQ_t(x)\) less useful as an estimate of the quality of the current solutions or policies.

To remedy these issues, we consider a technique called \emph{Lipschitzian regularization}.
Given regularization factors \(M_t>0\), we define the regularized local cost function as
\begin{equation}\label{eq:RegularizedCostFunction}
    f_t^\Reg(x_{t-1},x_t;\xi_t):=\inf_{z_t\in\bbR^{d_{t-1}}} f_t(z_t,x_t;\xi_t)+M_t\nvert{x_{t-1}-z_t},
\end{equation}
and the regularized value function
\begin{equation}\label{eq:RegularizedValueFunction}
    Q_t^\Reg(x_{t-1};\xi_t):=\min_{x_t\in\calX_t}f_t^\Reg(x_{t-1},x_t;\xi_t)+\calQ_t^\Reg(x_t),
\end{equation}
recursively for \(t=T,T-1,\dots,2\), where \(\calQ_t^\Reg\) is the regularized expected cost-to-go function defined as
\begin{equation}\label{eq:RegularizedCostToGoFunction}
    \calQ_t^\Reg(x_t):=\sup_{p_{t+1}\in\calP_{t+1}}\bbE_{\xi_{t+1}\sim p_{t+1}} Q_{t+1}^\Reg(x_t;\xi_{t+1}),
\end{equation}
for \(t\le T-1\), and \(\calQ_T^\Reg(x_T)\equiv 0\) for any \(x_T\in\calX_T\).
It is then straightforward to check that \(Q_t^\Reg(x_{t-1};\xi_t)\) is uniformly \(M_t\)-Lipschitz continuous in \(x_{t-1}\) for each \(\xi_t\in\Xi_t\), and consequently \(\calQ_{t-1}^\Reg\) is also \(M_t\)-Lipschitz continuous.
In the definitions~\eqref{eq:RecursiveUnderApproximation} and~\eqref{eq:RecursiveOverApproximation}, we can replace accordingly the original cost functions \(f_t\) with the regularized cost function \(f_t^\Reg\), which would guarantee the \(M_t\)-Lipschitz continuity of the generated cuts and allow us to enhance the over-approximation to be \(M_t\)-Lipschitz continuous.

Lipschitzian regularization in general only gives under-approximations of the true value and expected cost-to-go functions.
We need the following assumption to preserve the optimality and feasibility of the solutions.
\begin{assumption}\label{assum:ExactRegularization}
    For the given regularization factors \(M_t>0\), \(t\in\calT\), the optimal value of the regularized DR-MCO satisfies 
    \[
        \min_{x_1\in\calX_1}f_1(x_0,x_1;\xi_1)+\calQ_1(x_1)=\min_{x_1\in\calX_1}f_1(x_0,x_1;\xi_1)+\calQ_1^\Reg(x_1)
    \] 
    and the sets of optimal first-stage solutions are the same, i.e.,
    \(
        \argmin\{f_1(x_0,x_1;\xi_1)+\calQ_1^\Reg(x_1):x_1\in\calX_1\}=\argmin\{f_1(x_0,x_1;\xi_1)+\calQ_1(x_1):x_1\in\calX_1\}.
    \)
\end{assumption}

We remark by the following proposition that Assumption~\ref{assum:ExactRegularization} can be satisfied in any problem that already have uniformly Lipschitz continuous value function \(Q_t(\cdot;\xi_t)\) for all \(\xi_t\in\Xi_t\).
\begin{proposition}[{\cite[Proposition~4]{zhang2020distributionally}}]\label{prop:LipschitzRegularization}
    Suppose each state space \(\calX_t\subseteq\bbR^{d_t}\) is full dimensional, i.e., \(\mathrm{int}\calX_t\neq\varnothing\).
    Then Assumption~\ref{assum:ExactRegularization} holds if for each stage \(t\ge 2\), the value function \(Q_t(\cdot;\xi_t)\) is \(M_t\)-Lipschitz continuous for any \(\xi_t\in\Xi_t\).
\end{proposition}

Now we can present the (consecutive) DDP algorithm based on the SSSO. 
In Algorithm~\ref{alg:ConsecutiveDualDP}, each iteration \(i\in\bbN\) consists of two steps: the noninitial stage step and the initial stage step.
In the noninitial stage step, we evaluate the SSSO from stage \(t=2\) to \(t=T\) to collect feasible states \(x_t^i\), overestimate values \(v_{t-1}^i\), and a valid linear cut \(\calV_{t-1}^i\) for updating the approximations \(\ulcQ_{t-1}^i\) and \(\olcQ_{t-1}^i\).
Then we evaluate the initial stage SSSO to get an optimal solution \(x_1^{i+1}\) and its optimality gap \(\gamma_1^{i+1}\).
We terminate the algorithm when the gap is sufficiently small.

\begin{algorithm}[ht]
    \caption{Dual Dynamic Programming Algorithm}
    \label{alg:ConsecutiveDualDP}
    \begin{algorithmic}[1]
        \Require{subproblem oracles \(\scrO_t\) for \(t\in\calT\), optimality gap \(\epsilon>0\)} 
        \Ensure{an \(\epsilon\)-optimal first stage solution \(x_1^*\) to the regularization~\eqref{eq:RegularizedCostToGoFunction}}
        \State{initialize: \(\ulcQ_t^0\leftarrow 0,\olcQ_t^0\leftarrow +\infty,t\in\calT\backslash\{T\}\); \(\ulcQ_T^j,\olcQ_T^j\leftarrow0,j\in\bbN\); \(i\leftarrow 1\)}
        \State{evaluate \((x_1^1;\gamma_1^1)\leftarrow\scrO_1(\ulcQ_1^0,\olcQ_1^0)\)}
        \State{set \(\LB\leftarrow f_1(x_0,x_1^1;\xi_1),\ \UB\leftarrow +\infty\)}
        \While{\(\UB-\LB >\epsilon\)}
        \For{\(t=2,\dots,T\)}
        \State{evaluate \((\calV_{t-1}^{i},v_{t-1}^{i},x_t^i;\gamma_t^i)=\scrO_t(x_{t-1}^i,\ulcQ_t^{i-1},\olcQ_t^{i-1})\)}
        \Comment{noninitial stage step}
        \State{update \(\ulcQ_{t-1}^i(x)\leftarrow \max\{\ulcQ_{t-1}^{i-1}(x),\calV_{t-1}^i(x)\}\)}
        \State{update \(\olcQ_{t-1}^i(x)\leftarrow \conv\{\olcQ_{t-1}^{i-1}(x),v_{t-1}^i+M_{t-1}\nvert{x-x_{t-1}^i}\}\)}
        \EndFor
        \State{evaluate \((x_1^{i+1};\gamma_1^{i+1})\leftarrow\scrO_1(\ulcQ_1^i,\olcQ_1^i)\)}
        \Comment{initial stage step}
        \State{update \(\LB\leftarrow f_1(x_0,x_1^{i+1};\xi_1)+\ulcQ_1^i(x_1^{i+1})\)}
        \State{update \(\UB'\leftarrow f_1(x_0,x_1^{i+1};\xi_1)+\olcQ_1^i(x_1^{i+1})\)}
        \If{\(\UB'<\UB\)}
        \State{set \(x_1^*\leftarrow x_1^{i+1}\),  \(\UB\leftarrow\UB'\)}
        \EndIf
        \State{update \(i\leftarrow i+1\)}
        \EndWhile
    \end{algorithmic}
\end{algorithm}

Algorithm~\ref{alg:ConsecutiveDualDP} is proved to terminate with an \(\epsilon\)-optimal solution in finitely many iterations~\cite{zhang2020distributionally}.
We include the theorem here for completeness.
\begin{theorem}\label{thm:CDDPComplexityBound}
    Suppose that all the state spaces \(\calX_t\) have the dimensions bounded by \(d\) and diameters by \(D\), and let \(M\coloneqq\max\{M_t:t=1,\dots,T-1\}\).
    If for each stage \(t\in\calT\), the local cost functions are strictly positive \(f_t(x_{t-1},x_t;\xi_t)\ge C\) for all feasible solutions \(x_t\in\calX_t\), uncertainties \(\xi_t\in\Xi_t\) and some constant \(C>0\), 
    then the total number of noninitial stage subproblem oracle evaluations before achieving an \(\alpha\)-relative optimal solution \(x_1^*\) for Algorithm~\ref{alg:ConsecutiveDualDP} is bounded by
    \begin{equation*}
        \NumEval\le 1+T(T-1)\left(1+\frac{2MD}{\alpha C}\right)^d.
    \end{equation*}
\end{theorem}

We remark that the DDP algorithm can also be executed in a nonconsecutive way with a similar complexity bound.
However, since we need to run the algorithm for a fixed number of iterations to compare different models, we restrict our attention to the consecutive version here.
Moreover, we can execute the DDP algorithm even with \(M_t=+\infty\), with the risk that it may not converge in finite time.
Any interested reader is referred to~\cite{zhang2020distributionally} for more detailed discussion.

\subsection{Single Stage Subproblem Oracles}
\label{sec:Algorithm:SSSO}

DDP algorithms generally refer to the recursive cutting plane algorithms that exploit the stagewise independence structure.
We first review the definitions of single stage subproblem oracles (SSSO), which symbolize the subroutines of subproblem solving in each stage with under- and over-approximations of the expected cost-to-go functions~\cite{zhang2020distributionally}. 

\begin{definition}[Initial stage subproblem oracle]
\label{def:InitialStageOracle}
    Let \(\ulcQ_1,\olcQ_1:\calX_1\to\bar\bbR\) denote two lsc convex functions, representing an under-approximation and an over-approximation of the cost-to-go function \(\calQ_1\) (or its regularized surrogate defined in~\eqref{eq:RegularizedCostToGoFunction}), respectively.
    Consider the following subproblem for the first stage \(t=1\),
    \begin{equation}\label{eq:InitialStageOracle}
        \min_{x_1\in\calX_1}f_1(x_0,x_1;\xi_1)+\ulcQ_1(x_1).
    \end{equation}
    The initial stage subproblem oracle provides an optimal solution \(x_1\) to \eqref{eq:InitialStageOracle} and calculates the approximation gap \(\gamma_1\coloneqq\olcQ_1(x_1)-\ulcQ_1(x_1)\) at the solution.
    We thus define the subproblem oracle formally as a map \(\scrO_1:(\ulcQ_1,\olcQ_1)\mapsto(x_1;\gamma_1)\).
\end{definition}

\begin{definition}[Noninitial stage subproblem oracle]
\label{def:NoninitialStageOracle}
    Let \(\ulcQ_t,\olcQ_t:\calX_t\to\bar\bbR\) denote two lsc convex functions, representing an under-approximation and an over-approximation of the cost-to-go function \(\calQ_t\) (or its regularized surrogate defined in~\eqref{eq:RegularizedCostToGoFunction}), respectively, for some stage \(t>1\).
    Then given a feasible state \(x_{t-1}\in\calX_{t-1}\), the noninitial stage subproblem oracle provides a feasible state \(x_t\in\calX_t\), an \(M_t\)-Lipschitz continuous linear cut \(\calV_{t-1}(\cdot)\), and an over-estimate value \(v_{t-1}\) such that
    \begin{compactitem}
        \item they are valid, i.e., \(\calV_{t-1}(x)\le\calQ_{t-1}(x)\) for any \(x\in\calX_{t-1}\) and \(v_{t-1}\ge\calQ_{t-1}(x_{t-1})\);
        \item the gap is controlled, i.e., \(v_{t-1}-\calV_{t-1}(x_{t-1})\le\gamma_t\coloneqq\olcQ_t(x_t)-\ulcQ_t(x_t)\).
    \end{compactitem}
    We thus define the subproblem oracle formally as a map \(\scrO_t:(x_{t-1},\ulcQ_t,\olcQ_t)\mapsto(\calV_{t-1},v_{t-1},x_t;\gamma_t)\).
\end{definition}

Given these SSSO, the DDP algorithm is proved to converge under Lipschitzian assumptions on the expected cost-to-go functions.
The precise statement can be found in Theorem~\ref{thm:CDDPComplexityBound}.
In~\cite{zhang2020distributionally}, we showed how such an SSSO can be implemented through a standard enumerative method when the uncertainty is supported on a finite set $\Xi_t$.
However, for more general cases with continuous and possibly unbounded support, we face the following two challenges.
\begin{itemize}
    \item The worst-case probability measure may not exist when the support set $\Xi_t$ is unbounded (see e.g.,~\cite[Example 2]{esfahani2018data}).
    \item Even when the worst-case probability measure exists, its support is often different from that of the empirical measure $\hat{\nu}_{t}$, which makes the standard enumerative method inapplicable. 
\end{itemize}
To circumvent these challenges, we introduce two SSSO implementations which are based directly on the recursion~\eqref{eq:MDRO-FiniteDimensionalRecursion} without the need to find worst-case probability measures.
First in Section~\ref{sec:Algorithm:ConcaveSingleStageSubproblem}, we focus on the concave uncertain cost functions $f_t$, where an efficient conic reformulation of the SSSO is presented for a practically important class of the cost functions.
Second in Section~\ref{sec:Algorithm:ConvexSingleStageSubproblem}, we consider convex uncertain cost functions $f_t$ with polyhedral uncertainty sets $\Xi_t$ and distance function $d_{t,k}$, and develop an exact enumerative SSSO implementation in a lifted space, which ensures the convergence of DDP algorithms.

\subsection{Subproblem Oracles: Concave Uncertain Cost Functions}
\label{sec:Algorithm:ConcaveSingleStageSubproblem}

We begin with the easier case where $f_t$ is concave and upper semicontinuous in the uncertainty $\xi_t$.
This is a generalization of the single- or two-stage Wasserstein DRO studied in~\cite{gao2022distributionally} and in~\cite{esfahani2018data}.
We will see in~\eqref{eq:LinearCostFunction} that such DR-MCO model occurs when the uncertainty only affects the objective function in a multistage linear optimization.
\begin{assumption}\label{assum:ConcaveUncertaintyCost}
    The local cost function \(f_t(x_{t-1},x_t;\xi_t)\) is concave and upper semicontinuous in the uncertainty \(\xi_t\) for any \(x_{t-1}\in\calX_{t-1}\) and \(x_t\in\calX_t\).
\end{assumption}

A direct consequence of Assumption~\ref{assum:ConcaveUncertaintyCost} is that the effective domain of the state \(x_t\) does not depend on the uncertainty \(\xi_t\), as shown in the following lemma.

\begin{lemma}\label{lemma:FixedEffectiveDomain}
    Under Assumption~\ref{assum:ConcaveUncertaintyCost}, we have \(\dom{f_t(x_{t-1},\cdot;\xi_t)}=\dom{f_t(x_{t-1},\cdot;\xi'_t)}\) for any \(x_{t-1}\in\calX_{t-1}\) and \(\xi_t,\xi'_t\in\Xi_t\).
\end{lemma}
\begin{proof}{Proof}
    Assume for contradiction that there exists some \(x_{t-1}\in\calX_{t-1}\), \(x_t\in\calX_t\), and \(\xi_t,\xi'_t\in\Xi_t\) such that \(f_t(x_{t-1},x_t;\xi_t)<+\infty\) but \(f_t(x_{t-1},x_t;\xi'_t)=+\infty\).
    Then for \(c\in(0,1)\), we have \(f_t(x_{t-1},x_t;(1-c)\xi_t+c\xi'_t)=+\infty\) by the concavity and nonnegativity of \(f_t\).
    It follows from upper semicontinuity that \(f_t(x_{t-1},x_t;\xi_t)\ge\limsup_{c\to0+}f_t(x_{t-1},x_t;(1-c)\xi_t+c\xi'_t)=+\infty\), which is a contradiction.
\end{proof}

We are now ready to prove the alternative formulation of the recursion~\eqref{eq:MDRO-FiniteDimensionalRecursion}.
\begin{theorem}\label{thm:ConcaveUncertaintyCostRecursion}
    Under Assumption~\ref{assum:ConcaveUncertaintyCost}, if we further assume that the continuous functions \(g_{t,j}\) are convex for \(j=1,\dots,m\) and \(d_{t,k}(\xi_{t,k})=\nvert{\xi_{t,k}-\hat{\xi}_{t,k}}\) for \(k=1,\dots,n_t\), then we have
\begin{equation}\label{eq:MDRO-ConcaveCostRecursion}
    \begin{aligned}
        \calQ_{t-1}(x_{t-1})=
        \min\quad&\sum_{j=0}^{m_t}\rho_{t,j}\lambda_{t,j}+\frac{1}{n_t}\sum_{k=1}^{n_t}\left[h_{t,k}(z_t,x_{t,k},\zeta_{t,k},\lambda_t)+\calQ_t(x_{t,k})\right]\\
        \mathrm{s.t.}\quad&\norm{\zeta_{t,k}}_*\le\lambda_{t,0},\\
                          &z_t=x_{t-1},\\
                          &\lambda_t\in\bbR^{m_t+1}_{\ge0},x_{t,k}\in\calX_t,
    \end{aligned}
\end{equation}
where for each \(k=1,\dots,n_t\), \(h_{t,k}\) is defined as
\[
    h_{t,k}(x_{t-1},x_{t,k},\zeta_{t,k},\lambda_t):=\sup_{\xi_{t,k}\in\Xi_t}f_{t}(x_{t-1},x_{t,k};\xi_{t,k})-\sum_{j=1}^{m_t}\lambda_{t,j} g_{t,j}(\xi_{t,k})+\zeta_{t,k}^\transpose(\xi_{t,k}-\hat{\xi}_{t,k}).
\] 
\end{theorem}
\begin{proof}{Proof}
    By Lemma~\ref{lemma:FixedEffectiveDomain}, for any \(x_{t-1}\in\calX_{t-1}\), we can define a set \(\calX_t(x_{t-1}):=\dom f_t(x_{t-1},\cdot;\xi_t)\subseteq\calX_t\) which is independent of \(\xi_t\in\Xi_t\) and closed by the lower semicontinuity of \(f_t\).
    Note that the norm function has the dual representation \(d_{t,k}(\xi_{t,k})=\nvert{\xi_{t,k}-\hat{\xi}_{t,k}}=\max_{\nvert{\zeta}_*\le1}\zeta^\transpose(\xi_{t,k}-\hat{\xi}_{t,k})\).
    Thus by the recursion~\eqref{eq:MDRO-FiniteDimensionalRecursion}, we can write
    \begin{align*}
        \calQ_{t-1}(x_{t-1})=
        \min_{\lambda_t\in\bbR^{m_t+1}_{\ge0}}\sum_{j=0}^{m_t}\rho_{t,j}\lambda_{t,j}
        &+\frac{1}{n_t}\sum_{k=1}^{n_t}
        \sup_{\xi_{t,k}\in\Xi_t}\bigg[-\sum_{j=1}^{m_t}\lambda_{t,j} g_{t,j}(\xi_{t,k})\\
        +\min_{x_{t,k},\zeta_{t,k}}\quad&f_t(x_{t-1},x_{t,k};\xi_{t,k})+\calQ_t(x_{t,k})+\zeta_{t,k}^\transpose(\xi_{t,k}-\hat{\xi}_{t,k})\bigg]\\
        \mathrm{s.t.}\quad&\norm{\zeta_{t,k}}_*\le\lambda_{t,0},\\
                          &x_{t,k}\in\calX_t(x_{t-1}).
    \end{align*}
    Now for any fixed \(x_{t-1}\) and \(\lambda_t\), we see that the sets \(\{\zeta_{t,k}:\nvert{\zeta_{t,k}}_*\le\lambda_{t,0}\}\) and \(\calX_t(x_{t-1})\) are compact.
    Moreover, the function inside the supremum of \(\xi_{t,k}\) is concave and upper semicontinuous in \(\xi_{t,k}\), while convex and lower semicontinuous in \(\xi_{t,k}\) and \(\zeta_{t,k}\).
    Thus the result follows by applying Sion's minimax theorem~\cite{komiya1988elementary}.
\end{proof}
\begin{remark}
    The proof remains valid if we replace simultaneously \(\calQ_t\), \(\calQ_{t-1}\), and \(f_t\) with their regularized surrogates \(\calQ_t^\Reg\), \(\calQ_{t-1}^\Reg\), and \(f_t^\Reg\) (see definitions in~\ref{sec:Algorithm:DDPReview}) in the theorem.
    In this case we use \(h_{t,k}^\Reg\) to denote the convex conjugate functions.
\end{remark}

We provide a possible implementation for noninitial stage SSSO in Algorithm~\ref{alg:concaveSSSO} based on Theorem~\ref{thm:ConcaveUncertaintyCostRecursion}.
Its correctness is verified by the following corollary.

\begin{corollary}\label{cor:ConcaveUncertaintyCostRecursion}
    Under the same assumptions of Theorem~\ref{thm:ConcaveUncertaintyCostRecursion}, the outputs \((\calV_{t-1},v_{t-1},x_t;\gamma_t)\) of Algorithm~\ref{alg:concaveSSSO} satisfy the conditions in Definition~\ref{def:NoninitialStageOracle}.
\end{corollary}
\begin{proof}
    To check the validness of \(v_{t-1}\), let \((z_t^*,x_{t,k}^*,\zeta_{t,k}^*,\lambda_t^*)\) denote an optimal solution in the minimization~\eqref{eq:MDRO-ConcaveCostRecursion} with \(\calQ_t\) replaced by \(\ulcQ_t^i\).
    Then we have
    \begin{align*}
        v_{t-1}&=v_{t-1}^*+\frac{1}{n_t}\sum_{k=1}^{n_t}\gamma_{t,k}\\
               &=\sum_{j=0}^{m_t}\rho_{t,j}\lambda_{t,j}^*+\frac{1}{n_t}\sum_{k=1}^{n_t}[h_{t,k}(z_t^*,x_{t,k}^*,\zeta_{t,k}^*,\lambda_t^*)+\ulcQ_t^i(x_{t,k}^*)+\gamma_{t,k}]\\
               &\ge\sum_{j=0}^{m_t}\rho_{t,j}\lambda_{t,j}^*+\frac{1}{n_t}\sum_{k=1}^{n_t}[h_{t,k}(z_t^*,x_{t,k}^*,\zeta_{t,k}^*,\lambda_t^*)+\calQ_t(x_{t,k}^*)]\ge\calQ_{t-1}(x_{t-1}),
    \end{align*}
    the last inequality is due to the feasibility of \((\calV_{t-1},v_{t-1},x_t;\gamma_t)\) in the minimization~\eqref{eq:MDRO-ConcaveCostRecursion}.
    For the validness of \(\calV_{t-1}(\cdot)\), note that the value \(v_{t-1}^*\) and the dual solution \(u_t\) define a valid linear under-approximation for the function \(\calQ'_{t-1}(\cdot)\) defined by replacing \(\calQ_t\) with \(\ulcQ_t^i\) in the minimization~\eqref{eq:MDRO-ConcaveCostRecursion}.
    Since clearly \(\calQ'_{t-1}(x_{t-1})\le\calQ_{t-1}(x_{t-1})\) for all \(x_{t-1}\in\calX_{t-1}\), we see that \(\calV_{t-1}(\cdot)\) is a valid under-approximation for \(\calQ_{t-1}(\cdot)\).
    Finally the gap \(v_{t-1}-\calV_{t-1}(x_{t-1})=\frac{1}{n_t}\sum_{k=1}^{n_t}\gamma_{t,k}\le\gamma_t\) is controlled.
\end{proof}

Theorem~\ref{thm:ConcaveUncertaintyCostRecursion} and Algorithm~\ref{alg:concaveSSSO} would be most useful when the functions \(h_{t,k}\) can be written explicitly as minimization problems.
We thus spend the rest of this section to derive the form of \(h_{t,k}\) in a special yet practically important case, where the local cost function \(f_t\) can be written as
\begin{equation}\label{eq:LinearCostFunction}
    \begin{aligned}
        f_t(x_{t-1},x_t;\xi_t)=
        \min_{y_t}\quad& (A_t\xi_t+a_t)^\transpose y_t\\
        \mathrm{s.t.}\quad& (x_{t-1},y_t,x_t)\in\calF_t,
    \end{aligned}
\end{equation}
for some compact convex set \(\calF_t\subseteq\calX_{t-1}\times\bbR^{d'_t}\times\calX_t\) in each stage \(t\in\calT\).
It is straightforward to check that \(f_t\) in~\eqref{eq:LinearCostFunction} is lower semicontinuous and convex in \((x_{t-1},x_t)\) for any \(\xi_t\in\Xi_t\).
To simplify our discussion, we assume that \(f_t(\cdot,x_t;\xi_t)\) is \(M_t\)-Lipschitz continuous, so \(\calQ_t=\calQ_t^\Reg\) for all \(t\in\calT\) as discussed in~\cite{zhang2020distributionally}.
The problem~\eqref{eq:LinearCostFunction} is a common formulation in the usual MSCO literature, such as~\cite{shapiro2011analysis}, where \(\calF_t\) is supposed to be a polytope.

\begin{algorithm}[htb]
    \caption{Single Stage Subproblem Oracle Implementation Under Assumption~\ref{assum:ConcaveUncertaintyCost}}
    \label{alg:concaveSSSO}
    \begin{algorithmic}[1]
        \Require{function \(h_t\), over- and under-approximations \(\olcQ_t^i\) and \(\ulcQ_t^i\), and a state \(x_{t-1}\in\calX_{t-1}\)} 
        \Ensure{a linear cut \(\calV_{t-1}\), an overestimate \(v_{t-1}\), a state \(x_t\), and a gap value \(\gamma_t\)}
        \State{Solve the minimization~\eqref{eq:MDRO-ConcaveCostRecursion} with \(\calQ_t\) replaced by \(\ulcQ_t^i\) and store the optimal value \(v_{t-1}^*\), optimal solutions \(\lambda_t^*\) and \((x_{t,k}^*,\zeta_{t,k}^*)_{k=1}^{n_t}\) and the dual solutions \(u_t\) associated with the constraints \(z_t=x_{t-1}\)}
        \For{\(k=1,\dots,n_t\)}
        \State{Compute the gap value \(\gamma_{t,k}:=\olcQ_t^i(x_{t,k}^*)-\ulcQ_t^i(x_{t,k}^*)\)}
        \EndFor
        \State{Set \(\calV_{t-1}(\cdot)\leftarrow v_{t-1}^*+u_t^\transpose(\cdot)\)}
        \State{Set \(v_{t-1}\leftarrow v_{t-1}^*+\frac{1}{n_t}\sum_{k=1}^{n_t}\gamma_{t,k}\)}
        \State{Take any \(k^*\in\argmax\{\gamma_{t,k}:k=1,\dots,n_t\}\) and set \(x_t\leftarrow x_{t,k^*}^*\), \(\gamma_t\leftarrow \gamma_{t,k^*}\)}
    \end{algorithmic}
\end{algorithm}

\begin{proposition}\label{prop:LinearCostConvexConjugate}
    Suppose the local cost function \(f_t(x_{t-1},x_t;\xi_t)\) is given in the form~\eqref{eq:LinearCostFunction} and uniformly \(M_t\)-Lipschitz continuous in the variable \(x_{t-1}\).
    Fix any point \(\bar{\xi}_t\in\mathrm{int}(\Xi_t)\) and let \(\sigma_t(\zeta):=\sup_{\xi\in\Xi_t}\zeta^\transpose(\xi-\bar{\xi}_t)\) denote the support function of the set \(\Xi_t-\bar{\xi}_t\).
    If the functions \(g_{t,j}(\xi_t)=\xi_t^\transpose B_{t,j}\xi_t+b_{t,j}^\transpose\xi_t\) are quadratic with coefficients \(B_{t,j}\in\calS^{\delta_t}_{\succeq 0}\) and \(b_{t,j}\in\bbR^{\delta_t}\) for \(j=1,\dots,m_t\), then we can write
\begin{align}\label{eq:LinearCostConvexConjugate}
        h_{t,k}(x_{t-1},x_{t,k},\zeta_{t,k},\lambda_t)=
    \min_{y_{t,k},w_{t,j},w'_{t,j},\kappa_j}\quad &(A_t\bar{\xi}_t+a_t)^\transpose y_{t,k}-\sum_{j=1}^{m_t}\lambda_{t,j}[\bar{\xi}_t^\transpose B_{t,j}\bar{\xi}_t+b_{t,j}^\transpose\bar{\xi}_t]\\
    &+\zeta_{t,k}^\transpose(\bar{\xi}_t-\hat{\xi}_{t,k})+\sum_{j=1}^{m_t}\kappa_{t,j}+\sigma_t(w_{t,0})\notag\\
\mathrm{s.t.}\quad & \sum_{j=0}^{m_t}w_{t,j}=\zeta_{t,k}+A_t^\transpose y_{t,k}-\sum_{j=1}^{m_t}\lambda_{t,j}\big[2B_{t,j}\bar{\xi}_t+b_{t,j}\big],\notag\\
   & \kappa_{t,j}\ge 0,\quad j=1,\dots,m,\notag\\
   & w_{t,j}\in\bbR^{\delta_t},\quad j=0,\dots,m,\notag\\
   & \kappa_{t,j}+\lambda_{t,j}\ge\nvert{(\kappa_{t,j}-\lambda_{t,j},U_{t,j}w_{t,j})}_2,\quad j=1,\dots,m_t,\notag\\
   & w_{t,j}=B_{t,j}w'_{t,j},\ w'_{t,j}\in\bbR^{\delta_t},\quad j=1,\dots,m,\notag\\
   & (x_{t-1},y_{t,k},x_{t,k})\in\calF_t.\notag
\end{align}
    Here, \(U_{t,j}\) is a \(\delta_t\times\delta_t\) real matrix such that \(U_{t,j}^\transpose U_{t,j}\) is the pseudoinverse of \(B_{t,j}\).
\end{proposition}

\begin{proof}
    Under the assumptions, we can write the function \(h_{t,k}\) as
    \[
        \begin{aligned}
        h_{t,k}(x_{t-1},x_{t,k},\zeta_{t,k},\lambda_t)=\sup_{\xi_{t}\in\Xi_t}
        \min_{y_{t,k}}\ &(A_t\xi_t+a_t)^\transpose y_{t,k}+\zeta_{t,k}^\transpose(\xi_{t}-\hat{\xi}_{t,k})\\
        &-\sum_{j=1}^{m_t}\lambda_{t,j}(\xi_t^\transpose B_{t,j}\xi_t+b_{t,j}^\transpose\xi_{t})\\
        \mathrm{s.t.}\ &(x_{t-1},y_{t,k},x_{t,k})\in\calF_t.
        \end{aligned}
    \]
    Note that the objective function in~\eqref{eq:LinearCostConvexConjugate} is continuous in both \(y_t\) and \(\xi_{t,k}\), and the projection of \(\calF_t\) onto the variables \(y_{t,k}\) is compact.
    Thus by the minimax theorem~\cite{komiya1988elementary}, we can exchange the supremum and minimum operations
    \begin{equation}\label{eq:LinearCostConvexConjugate-Supremum}
        \begin{aligned}
            h_{t,k}(x_{t-1},x_{t,k},\zeta_{t,k},\lambda_t)= 
        \min_{y_{t,k}}\ &
        (A_t\bar{\xi}_t+a_t)^\transpose y_{t,k}-\sum_{j=1}^{m_t}\lambda_{t,j}[\bar{\xi}_t^\transpose B_{t,j}\bar{\xi}_t+b_{t,j}^\transpose\bar{\xi}_t]+\zeta_{t,k}^\transpose(\bar{\xi}_t-\hat{\xi}_{t,k})\\
        +\sup_{\xi_{t}\in\bbR^{\delta_t}}\bigg\{&\zeta_{t,k}^\transpose\xi_t-\iota_t(\xi_t)+ 
y_{t,k}^\transpose A_t\xi_t-\sum_{j=1}^{m_t}\lambda_{t,j}(\xi_t^\transpose B_{t,j}\xi_t+2\bar{\xi}_tB_{t,j}\xi_t+b_{t,j}^\transpose\xi_{t})\bigg\}\\
        \mathrm{s.t.}\ &(x_{t-1},y_{t,k},x_{t,k})\in\calF_t.
        \end{aligned}
    \end{equation}
    Here, \(\iota_t\) is the convex indicator function of the set \(\Xi_t-\bar{\xi}_t\), the convex conjugate of which is the support function \(\sigma_t\) by definition.
    If we further denote \(\phi_{t,j}(\xi_t;\lambda_{t,j}):=\lambda_{t,j}(\xi_t^\transpose B_{t,j}\xi_t)\), the supremum can be written using convex conjugacy as
    \[
        \left(\iota_t+\sum_{j=1}^{m_t}\phi_{t,j}(\cdot;\lambda_{t,j})\right)^*\bigg(\zeta_{t,k}+A_t^\transpose y_{t,k}-\sum_{j=1}^{m_t}\lambda_{t,j}\big[2B_{t,j}\bar{\xi}_t+b_{t,j}\big]\bigg).
    \]
    Note that for each \(j=1,\dots,m_t\), the parametrized conjugate function \(\phi_{t,j}^*(\cdot;\lambda_{t,j})\) can be written as \cite[Example 11.10]{rockafellar2009variational}
    \[
    \begin{aligned}
        \phi_{t,j}^*(w;\lambda_{t,j})
        &=\begin{cases}
            &\displaystyle\frac{w^\transpose B_{t,j}^\dagger w}{4\lambda_{t,j}},\text{ if }w\in\mathrm{range}B_{t,j},\\
            &+\infty,\quad\text{otherwise},
        \end{cases}\\
        &=\min\left\{\kappa_{t,j}\ge 0:4\kappa_{t,j}\lambda_{t,j}\ge(U_{t,j}w)^\transpose(U_{t,j}w),w=B_{t,j}w'\right\}\\
        &=\min\left\{\kappa_{t,j}\ge 0:\kappa_{t,j}+\lambda_{t,j}\ge\nvert{(\kappa_{t,j}-\lambda_{t,j},U_{t,j}w)}_2,w=B_{t,j}w'\right\},
    \end{aligned}
    \]
    which is nonnegative and second-order conic representable.
    Here the convention for \(\lambda_{t,j}=0\) is consistent: 
    we have \(\phi_{t,j}^*(0;0)=0\) and \(\phi_{t,j}^*(w;0)=+\infty\) for any \(w\neq 0\) because \((U_{t,j}B_{t,j})^\transpose(U_{t,j}B_{t,j})=B_{t,j}B_{t,j}^\dagger B_{t,j}=B_{t,j}\), which implies that \(U_{t,j}w=U_{t,j}B_{t,j}w'\neq 0\).
    Now using the formula for convex conjugate of sum of convex functions, we have
    \begin{equation}\label{eq:LinearCostConvexConjugate-InfConvolution}
        \left(\iota_t+\sum_{j=1}^{m_t}\phi_{t,j}(\cdot;\lambda_{t,j})\right)^*=\mathrm{cl}\left(\sigma_t\square\phi_{t,1}^*(\cdot;\lambda_{t,1})\square\cdots\square\phi_{t,m_t}(\cdot;\lambda_{t,m_t})\right),
    \end{equation}
    where \(\square\) denotes the infimal convolution (a.k.a.\ epi-addition) of two convex functions and \(\mathrm{cl}\) denotes the lower semicontinuous hull of a proper function.
    Since \(\bar{\xi}_t\in\mathrm{int}\Xi_t\), the support function is coercive, i.e.,  \(\lim_{\nvert{w}\to\infty}\sigma_t(w)=+\infty\).
    Moreover, each \(\phi_{t,j}^*\) is bounded below as it is nonnegative.
    Therefore, the closure operation is superficial and the convex conjugate of the sum is indeed lower semicontinuous \cite[Proposition 12.14]{bauschke2011convex}.
    The rest of the proof follows from substitution of this convex conjugate expression~\eqref{eq:LinearCostConvexConjugate-InfConvolution} into the supremum in~\eqref{eq:LinearCostConvexConjugate-Supremum}.
\end{proof}

\subsection{Subproblem Oracles: Convex Uncertain Cost Functions}
\label{sec:Algorithm:ConvexSingleStageSubproblem}

We provide another useful reformulation of the recursion~\eqref{eq:MDRO-FiniteDimensionalRecursion} based on the following assumption.

\begin{assumption}\label{assum:ConvexUncertaintyCost}
    The local cost function \(f_t(x_{t-1},x_t;\xi_t)\) is jointly convex in the state variable \(x_t\) and the uncertainty \(\xi_t\), for any \(x_{t-1}\in\calX_{t-1}\).
    Moreover, the uncertainty set \(\Xi_t\) is a polyhedron and the distance function \(d_{t,k}(\cdot)\) is polyhedrally representable.
\end{assumption}

From Assumption~\ref{assum:ConvexUncertaintyCost}, the value function \(Q_t(x_{t-1};\xi_t)\) would be a convex function in the uncertainty \(\xi_t\) for each state \(x_{t-1}\in\calX_{t-1}\), although it may not be a jointly convex function.
Moreover, we may define a lifted uncertainty set as \(\tilde{\Xi}_{t,k}:=\{(\zeta,\xi):\xi\in\Xi_t,\,\zeta\ge d_{t,k}(\xi)\}\), which is also a polyhedron.
We denote its finite set of extreme points as \(\ext\tilde{\Xi}_{t,k}=\{(\tilde{\zeta}_{l},\tilde{\xi}_{l})\}_{l\in E_{t,k}}\) where \(E_{t,k}\) is the set of indices, which is nonempty since $(\zeta,\xi)=(0,\hat{\xi}_{t,k})$ is an extreme point of $\tilde{\Xi}_{t,k}$.
We next show that the maximization in~\eqref{eq:MDRO-FiniteDimensionalRecursion} can be taken over the finite set \(\{(\tilde{\zeta}_{l},\tilde{\xi}_{l})\}_{l\in E_{t,k}}\) in two important cases.
The first case is when we have bounded uncertainty sets \(\Xi_t\).

\begin{proposition}\label{prop:ConvexCostRecursionBounded}
    Under Assumption~\ref{assum:ConvexUncertaintyCost}, if we further assume that \(\Xi_t\) is bounded and all functions \(g_{t,j}\) are concave for \(j=1,\dots,m_t\), then the problem~\eqref{eq:MDRO-FiniteDimensionalRecursion} can be equivalently reformulated as
    \begin{align}\label{eq:ConvexCostRecursionBounded}
        \calQ_{t-1}(x_{t-1})= \min_{\lambda_t,\tau_t}\quad&\sum_{j=0}^{m_t}\rho_{t,j}\lambda_{t,j}+\frac{1}{n_t}\sum_{k=1}^{n_t}\tau_{t,k}\\
        \mathrm{s.t.}\quad& \lambda_t\ge0,\notag\\
                          &\tau_{t,k}\ge Q_{t}(x_{t-1};\tilde{\xi}_l)-\lambda_{t,0} \tilde{\zeta}_l-\sum_{j=1}^{m_t}\lambda_{t,j} g_{t,j}(\tilde{\xi}_l),
                          \forall\, l\in E_{t,k} \textnormal{ and } k=1,\dots,n_t.\notag
    \end{align}
\end{proposition}

\begin{proof}
    From the definition of lifted uncertainty set \(\tilde{\Xi}_t\), we have
\begin{align*}
    &\sup_{\xi_k\in\Xi_t}\left\{Q_{t}(x_{t-1};\xi_k)-\lambda_{t,0} d_{t,k}(\xi_k)-\sum_{j=1}^{m_t}\lambda_{t,j} g_{t,j}(\xi_k)\right\}\\
    &=\max_{\xi_k\in\Xi_t,\zeta_k\in\bbR}\left\{Q_{t}(x_{t-1};\xi_k)-\lambda_{t,0} \zeta_k-\sum_{j=1}^{m_t}\lambda_{t,j} g_{t,j}(\xi_k):\zeta_k\ge d_{t,k}(\xi_k)\right\}\notag\\
    &=\max_{(\zeta_k,\xi_k)\in \tilde{\Xi}_{t,k}}\left\{Q_{t}(x_{t-1};\xi_k)-\lambda_{t,0} \zeta_k-\sum_{j=1}^{m_t}\lambda_{t,j} g_{t,j}(\xi_k)\right\}\notag\\
    &=\max_{l\in E_{t,k}}\left\{Q_{t}(x_{t-1};\tilde{\xi}_l)-\lambda_{t,0} \tilde{\zeta}_l-\sum_{j=1}^{m_t}\lambda_{t,j} g_{t,j}(\tilde{\xi}_l)\right\}.\notag
\end{align*}
    To see the last equality, note that if \(\Xi_t\) is bounded, then the only recession direction of the lifted uncertainty set \(\tilde{\Xi}_{t,k}\) is \((1,0)\).
    Since \(\lambda_{t,0}\ge0\), any maximum solution \((\zeta_k^*,\xi_k^*)\) lies in the convex hull of \(\ext\tilde{\Xi}_{t,k}\).
    Now the last equality follows from the convexity of the function \(Q_{t}(x_{t-1};\xi_k)-\lambda_{t,0}\zeta_k-\sum_{j=1}^{m_t}\lambda_{t,j} g_{t,j}(\xi_k)\) in terms of \(\xi_k\) and \(\zeta_k\).
    Finally, the reformulation is done by replacing the maximum of finitely many functions by its epigraphical representation \(\tau_{t,k}\ge Q_{t}(x_{t-1};\xi_l)-\lambda_{t,0} \zeta_l-\sum_{j=1}^{m_t}\lambda_{t,j} g_{t,j}(\xi_l)\) for all \(l\in E_{t,k}\) and \(k=1,\dots,n_t\).
\end{proof}

If the uncertainty sets \(\Xi_t\) are unbounded, then in general the supremum in~\eqref{eq:MDRO-FiniteDimensionalRecursion} can take \(+\infty\) in some unbounded directions of \(\Xi_t\), even when the value function \(Q_{t}(x_{t-1};\cdot)\) has finite values everywhere.
To avoid such situation, we consider the growth rate of the value function \(Q_{t}(x_{t-1};\cdot)\) defined as
\begin{equation}\label{eq:ValueFunctionGrowthRate}
    r_{t}(x_{t-1}):=\limsup_{\substack{d_{t,k}(\xi_t)\to\infty,\\\xi_t\in\Xi_t}}\frac{Q_{t}(x_{t-1};\xi_t)-Q_{t}(x_{t-1};\hat{\xi}_{t,k})}{d_{t,k}(\xi_t)}\ge0,
\end{equation}
for any real-valued \(Q_t(x_{t-1},\cdot)\), where the limit superior is in fact independent of the choice of \(k=1,\dots,n_t\), and the inequality is due to that \(Q_t(x_{t-1};\cdot)\) is assumed to be lower bounded by 0.
Our convention is to set \(r_t(x_{t-1})\equiv 0\) when \(\Xi_t\) is bounded.
We now consider problems with unbounded uncertainty sets \(\Xi_t\).

\begin{proposition}\label{prop:ConvexCostRecursionUnbounded}
    Under Assumption~\ref{assum:ConvexUncertaintyCost}, if \(Q_{t}(x_{t-1};\cdot)\) has finite growth rate \(r_t(x_{t-1})\) and all functions \(g_{t,j}\) are bounded and concave for \(j=1,\dots,m_t\), then the problem~\eqref{eq:MDRO-FiniteDimensionalRecursion} with any \(x_{t-1}\in\calX_{t-1}\) such that \(\calQ_{t-1}(x_{t-1})<+\infty\) can be equivalently reformulated as
    \begin{align}\label{eq:ConvexCostRecursionUnbounded}
        \calQ_{t-1}(x_{t-1})= \min_{\lambda_t,\tau_t}\quad&\sum_{j=0}^{m_t}\rho_{t,j}\lambda_{t,j}+\frac{1}{n_t}\sum_{k=1}^{n_t}\tau_{t,k}\\
        \mathrm{s.t.}\quad& \lambda_t\ge0,\notag\\
                          & \lambda_{t,0}\ge r_t(x_{t-1}),\notag\\
                          &\tau_{t,k}\ge Q_{t}(x_{t-1};\tilde{\xi}_l)-\lambda_{t,0} \tilde{\zeta}_l-\sum_{j=1}^{m_t}\lambda_{t,j} g_{t,j}(\tilde{\xi}_l),
                          \forall\, l\in E_{t,k} \textnormal{ and } k=1,\dots,n_t.\notag
    \end{align}
\end{proposition}

\begin{proof}
    We claim that the supremum 
    \[
        \sup_{\xi_k\in\Xi_t}\biggl\{Q_{t}(x_{t-1};\xi_k)-\lambda_{t,0} d_{t,k}(\xi_k)-\sum_{j=1}^{m_t}\lambda_{t,j} g_{t,j}(\xi_k)\biggr\}<+\infty
    \] 
    if and only if \(\lambda_{t,0}\ge r_t(x_{t-1})\), for each \(k=1,\dots,n_t\).
    Suppose \(\lambda_{t,0}<r_t(x_{t-1})\). 
    By definition~\eqref{eq:ValueFunctionGrowthRate}, there exists a sequence \(\{\xi_k^{(i)}\}_{i\in\bbN}\subseteq\Xi_t\) and a constant \(\epsilon>0\) such that \(d_{t,k}(\xi_k^{(i)})\to\infty\) as \(i\to\infty\) and \(Q_t(x_{t-1};\xi_k^{(i)})\ge Q_t(x_{t-1};\hat{\xi}_{t,k})+(\lambda_{t,0}+\epsilon)d_{t,k}(\xi_k^{(i)})\).
    Thus \(\sup_{i\in\bbN}\{Q_{t}(x_{t-1};\xi_k^{(i)})-\lambda_{t,0} d_{t,k}(\xi_k^{(i)})-\sum_{j=1}^{m_t}\lambda_{t,j} g_{t,j}(\xi_k^{(i)})\}\ge\sup_{i\in\bbN}\{\epsilon d_{t,k}(\xi_k^{(i)})-\sum_{j=1}^{m_t}\lambda_{t,j} g_{t,j}(\xi_k^{(i)})\}=+\infty\) as \(g_{t,j}(\xi_k^{(i)})\) for \(j=1,\dots,m_t\) are bounded.

    Conversely, by definition~\eqref{eq:ValueFunctionGrowthRate}, there exists a constant \(\bar{d}>0\) such that \(Q_t(x_{t-1};\xi_k)\le Q_t(x_{t-1};\hat{\xi}_{t,k})+\lambda_{t,0}d_{t,k}(\xi_k)\) for all \(\xi_k\in\Xi_t\) with \(d_{t,k}(\xi_k)\ge\bar{d}\). 
    Thus we have 
    \[
    \begin{aligned}
        &\sup_{\xi_k\in\Xi_t}\left\{Q_{t}(x_{t-1};\xi_k)-\lambda_{t,0} d_{t,k}(\xi_k)-\sum_{j=1}^{m_t}\lambda_{t,j} g_{t,j}(\xi_k)\right\}\\
        &\le\sup_{d_{t,k}(\xi_k)\le\bar{d}}\left\{Q_{t}(x_{t-1};\xi_k)-\lambda_{t,0} d_{t,k}(\xi_k)\right\}+\sup_{\xi_k\in\Xi_t}\sum_{j=1}^{m_t}\left(-\lambda_{t,j}g_{t,j}(\xi_k)\right)\\
        &=\max_{(\zeta_k,\xi_k)\in\tilde{\Xi}_t(\bar{d})}\left\{Q_{t}(x_{t-1};\xi_k)-\lambda_{t,0} \zeta_k\right\}+\sup_{\xi_k\in\Xi_t}\sum_{j=1}^{m_t}\left(-\lambda_{t,j}g_{t,j}(\xi_k)\right)<+\infty,
    \end{aligned}
    \]
    where \(\tilde{\Xi}_t(\bar{d}):=\{(\zeta,\xi):\xi\in\Xi_t,d_{t,k}(\xi)\le\bar{d},\zeta\ge d_{t,k}(\xi)\}\), and the maximum is finite because it is attained on some extreme point \((\bar{\zeta}_k,\bar{\xi}_k)\in\tilde{\Xi}_t(\bar{d})\) by convexity, so \(Q_t(x_{t-1};\bar{\xi}_k)-\lambda_{t,0}\bar{\zeta}_k<+\infty\).

    Now from this claim, we see that for any \(x_{t-1}\in\calX_{t-1}\) such that \(\calQ_{t-1}(x_{t-1})<+\infty\), the problem~\eqref{eq:MDRO-FiniteDimensionalRecursion} can be formulated equivalently as
    \[
        \begin{aligned}
            \calQ_{t-1}(x_{t-1})=
            \min_{\lambda_t\ge0}\quad&
            \frac{1}{n_t}\sum_{k=1}^{n_t}\sup_{(\zeta_k,\xi_k)\in\Xi_t}\left\{Q_{t}(x_{t-1};\xi_k)-\lambda_{t,0} \zeta_k-\sum_{j=1}^{m_t}\lambda_{t,j} g_{t,j}(\xi_k)\right\}
            +\sum_{j=0}^{m_t}\rho_{t,j}\lambda_{t,j}\\
            \mathrm{s.t.}\quad& \lambda_{t,0}\ge r_t(x_{t-1}).
        \end{aligned}
    \]
    The supremum can be attained in \(\tilde{\Xi}'_{t,k}:=\conv(\ext\tilde{\Xi}_{t,k})\): otherwise there exists a point \((\check{\zeta}_k,\check{\xi}_k)\in\tilde{\Xi}_{t,k}\setminus\tilde{\Xi}'_{t,k}\) and \((\bar{\zeta}_k,\bar{\xi}_k)\in\tilde{\Xi}'_{t,k}\) such that
    \[
        Q_{t}(x_{t-1};\check{\xi}_k)-\lambda_{t,0} \check{\zeta}_k-\sum_{j=1}^{m_t}\lambda_{t,j} g_{t,j}(\check{\xi}_k)
        >Q_{t}(x_{t-1};\bar{\xi}_k)-\lambda_{t,0} \bar{\zeta}_k-\sum_{j=1}^{m_t}\lambda_{t,j} g_{t,j}(\bar{\xi}_k).
    \]
    In other words, \((\check{\zeta}_k,\check{\xi}_k)-(\bar{\zeta}_k,\bar{\xi}_k)\) defines a strictly increasing ray of \(\tilde{\Xi}_t\), which by convexity implies that the supremum is \(+\infty\), a contradiction.
    Using the convexity again as in the proof of Proposition~\ref{prop:ConvexCostRecursionBounded}, we conclude that the supremum is indeed attained in \(\ext\tilde{\Xi}_{t,k}\), and this completes the proof.
\end{proof}

Proposition~\ref{prop:ConvexCostRecursionUnbounded} reduces to Proposition~\ref{prop:ConvexCostRecursionBounded} since the growth rate \(r_t(x_{t-1})=0\) and any continuous function \(g_{t,j}\) over a bounded polyhedron is bounded.
It is possible that the unbounded case is easier to handle than the bounded case.
For instance, if $\Xi_t=\bbR^{\delta_t}$ is the entire Euclidean space, then $(0,\hat{\xi}_{t,k})$ is the only extreme point of the lifted uncertainty set $\tilde{\Xi}_{t,k}$.
In this case, if $m_t=0$, then Proposition~\ref{prop:ConvexCostRecursionUnbounded} shows that $\calQ_{t-1}(x_{t-1})=\rho_{t,0}r_t(x_{t-1})+\frac{1}{n_t}\sum_{k=1}^{n_t}Q_t(x_{t-1};\hat{\xi}_{t,k})$, which is the sample average plus a regularization term on the growth rate $r_t(x_{t-1})$.

The finite growth rate condition is often satisfied, especially when the value function \(Q_t(x_{t-1};\cdot)\) is Lipschitz continuous.
However, it is in general difficult to estimate the growth rate~\eqref{eq:ValueFunctionGrowthRate}.
Fortunately, the growth rate can be calculated for a class of problems of the following form:
\begin{equation}\label{eq:RHSCostFunction}
    \begin{aligned}
        f_t(x_{t-1},x_t;\xi_t)=\min_{y_t}\quad &\max_{s=1,\dots,S}\bigl\{c_{t,s}^\transpose x_t+(c'_{t,s})^\transpose y_t+(c''_{t,s})^\transpose \xi_t\bigr\}\\
        \mathrm{s.t.}\quad & A_tx_t+B_ty_t\le A'_tx_{t-1}+B'_t\xi_t+b_t,\ y_t\in\calY_t,\\
    \end{aligned}
\end{equation}
for vectors $c_{t,s},c'_{t,s},c''_{t,s},b_t$ and matrices $A_t,B_t,A'_t,B'_t$ of appropriate dimensions.
Here, $f_t$ is lsc when the set of internal variables $\calY_t$ is compact; it  satisfies Assumption~\ref{assum:ConvexUncertaintyCost} as $f_t$ is a partial minimization of a maximum of convex function in $x_{t-1}$, $x_t$, $\xi_t$, and $y_t$.
Now assuming $\dom{f_t(x_{t-1},\cdot;\xi_t)}=\calX_t$ for any $x_{t-1}\in\calX_{t-1}$ and $\xi_t\in\Xi_t$, known as \emph{complete recourse}, then the growth rate $r_t(x_{t-1})$ is independent of $x_{t-1}$ and is determined by the supremum of $\max_{s}(u^\transpose c''_{t,s})$ over all unit vectors $u$ in the recession cone of $\Xi_t$. 
This can be done by enumerating over the finitely many extreme rays of $\Xi_t$ when it is pointed, e.g., the standard unit vectors when $\Xi_t=\bbR_{\ge0}^{\delta_t}$ is the nonnegative orthant (see Section~\ref{sec:Numerical:Hydrothermal}).

Note that the problems~\eqref{eq:ConvexCostRecursionBounded} and~\eqref{eq:ConvexCostRecursionUnbounded} are standard linear optimization problems in the variables \(\lambda_t\) and \(\tau_t\).
Thus by strong duality, we can write the dual problem as
\begin{equation}\label{eq:ConvexCostRecursionDualProblem}
    \begin{aligned}
            \calQ_{t-1}(x_{t-1})= \max_{\theta_t,\kappa_{t,k,l}\ge 0}\quad& \theta_t r_t(x_{t-1})+\sum_{k=1}^{n_t}\sum_{l\in E_{t,k}}\kappa_{t,k,l}Q_t(x_{t-1};\tilde{\xi}_l)\\
            \mathrm{s.t.}\quad&\sum_{l\in E_{t,k}}\kappa_{t,k,l}=\frac{1}{n_t},\quad k=1,\dots,n_t,\\
                              &\theta_t+\sum_{k=1}^{n_t}\sum_{l\in E_{t,k}}\tilde{\zeta}_l\kappa_{t,k,l}\le \rho_{t,0},\\
                              &\sum_{k=1}^{n_t}\sum_{l\in E_{t,k}}g_{t,j}(\tilde{\xi}_l)\kappa_{t,k,l}\le\rho_{t,j},\quad j=1,\dots,m_t.
    \end{aligned}
\end{equation}
Consequently, any feasible dual solutions \(\theta_t\) and \(\kappa_{t,k,l}\) to the dual~\eqref{eq:ConvexCostRecursionDualProblem} define a valid under-approximation 
\begin{equation}\label{eq:ConvexCostRecursionDualInequality}
    \calQ_{t-1}(x_{t-1})\ge\theta_t r_t(x_{t-1})+\sum_{k=1}^{n_t}\sum_{l\in E_{t,k}}\kappa_{t,k,l}Q_t(x_{t-1};\tilde{\xi}_l),\quad\forall\,x_{t-1}\in\calX_{t-1}.
\end{equation}

\begin{algorithm}[ht]
    \caption{Single Stage Subproblem Oracle Implementation Under Assumption~\ref{assum:ConvexUncertaintyCost}}
    \label{alg:convexSSSO}
    \begin{algorithmic}[1]
        \Require{ over- and under-approximations \(\olcQ_t^i\) and \(\ulcQ_t^i\), a state \(x_{t-1}\in\calX_{t-1}\), growth rate \(r_t(x_{t-1})\), and extreme point sets \(\ext\tilde{\Xi}_{t,k}\) for \(k=1,\dots,n_t\)} 
        \Ensure{a linear cut \(\calV_{t-1}\), an overestimate \(v_{t-1}\), a state \(x_t\), and a gap value \(\gamma_t\)}
        \For{\(k=1,\dots,n_t\)}
        \For{\(l\in E_{t,k}\)}\label{alg:convexSSSO:enumeration}
        \State{Evaluate the approximate value function 
            \[
                \ulQ_t(x_{t-1};\tilde{\xi}_l):=\min_{x_t\in\calX_t}f_t(x_{t-1},x_t;\tilde{\xi}_l)+\ulcQ_t^i(x_t)
            \]
        with a minimizer stored as \(x_{t,k,l}\) and a subgradient vector as \(u_{t,k,l}\in\partial\ulQ_t(\cdot;\tilde{\xi}_l)\) at \(x_{t-1}\)}
        \State{Calculate \(\gamma_{t,k,l}:=\olcQ_t^i(x_{t,k,l})-\ulcQ_t^i(x_{t,k,l})\)}
        \EndFor
        \EndFor
        \State{Solve the problem~\eqref{eq:ConvexCostRecursionUnbounded} (or~\eqref{eq:ConvexCostRecursionBounded} if \(\Xi_t\) is bounded) with \(Q_t(x_{t-1};\tilde{\xi}_l)\) replaced by \(\ulQ_t(x_{t-1};\tilde{\xi}_l)\) and store the optimal value \(v_{t-1}^*\) and dual solutions \(\theta_t^*\), \(\kappa_{t,k,l}^*\) to~\eqref{eq:ConvexCostRecursionDualProblem}}
        \For{\(k=1,\dots,n_t\)}
        \State{Take any \(l^*\in\argmax\big\{\olcQ_t^i(x_{t,k,l})-\ulcQ_t^i(x_{t,k,l}):l\in E_{t,k}\big\}\)}
        \State{Set \(\gamma_{t,k}\leftarrow\gamma_{t,k,l^*}\) and \(x_{t,k}\leftarrow x_{t,k,l^*}\)}
        \EndFor
        \State{Take a subgradient \(w_t\in\partial r_t(\cdot)\) at \(x_{t-1}\)}
        \State{Set \(\calV_{t-1}(\cdot)\leftarrow v_{t-1}^*+\theta^*_t w_t^\transpose(\cdot-x_{t-1})+\sum_{k=1}^{n_t}\sum_{l\in E_t}\kappa_{t,k,l}^*u_{t,k,l}^\transpose(\cdot-x_{t-1})\)}
        \State{Set \(v_{t-1}\leftarrow v_{t-1}^*+\frac{1}{n_t}\sum_{k=1}^{n_t}\gamma_{t,k}\)}
        \State{Take any \(k^*\in\argmax\{\gamma_{t,k}:k=1,\dots,n_t\}\) and set \(x_t\leftarrow x_{t,k^*}\), \(\gamma_t\leftarrow\gamma_{t,k^*}\)}
    \end{algorithmic}
\end{algorithm}

We now describe an SSSO implementation in Algorithm~\ref{alg:convexSSSO}.
Its correctness is verified in the following corollary.

\begin{corollary}\label{cor:ConvexCostRecursionUnbounded}
    Suppose that the growth rate function \(r_t(\cdot)\) is convex.
    Under the assumptions of of Proposition~\ref{prop:ConvexCostRecursionUnbounded}, the outputs \((\calV_{t-1},v_{t-1},x_t;\gamma_t)\) of Algorithm~\ref{alg:convexSSSO} satisfy the conditions in Definition~\ref{def:NoninitialStageOracle}.
\end{corollary}

\begin{proof}
    The validness of \(\calV_{t-1}(\cdot)\) follows directly from the inequality~\eqref{eq:ConvexCostRecursionDualInequality} and the fact that \(\ulQ_t(x_{t-1};\xi_k)\le Q_t(x_{t-1};\xi_k)\) for any \(x_{t-1}\in\calX_{t-1}\) and \(\xi_k\in\Xi_t\) by definition.
    To see the validness of \(v_{t-1}\), note that for any \(k=1,\dots,n_t\) and \(l\in E_{t,k}\), we have
    \[
    \begin{aligned}
        Q_t(x_{t-1};\tilde{\xi}_l)
        &\le\min_{x_t\in\calX_t}\left[f_t(x_{t-1},x_t;\tilde{\xi}_l)+\olcQ_t^i(x_t)\right]\\
        &\le f_t(x_{t-1},x_{t,k,l};\tilde{\xi}_l)+\olcQ_t^i(x_{t,k,l})\\
        &\le f_t(x_{t-1},x_{t,k,l};\tilde{\xi}_l)+\ulcQ_t^i(x_{t,k,l})+\gamma_{t,k,l}\\
        &\le \ulQ_t(x_{t-1};\tilde{\xi}_l)+\gamma_{t,k}
    \end{aligned}
    \]
    by the definition of \(\gamma_{t,k}\) in Algorithm~\ref{alg:convexSSSO}.
    Thus for any optimal solution \(\lambda_t^*\) to the problem~\eqref{eq:ConvexCostRecursionUnbounded} with \(Q_t(x_{t-1};\tilde{\xi}_l)\) replaced by \(\ulQ_t(x_{t-1};\tilde{\xi}_l)\), we have
    \begin{align*}
        &\max_{l\in E_{t,k}}Q_t(x_{t-1};\tilde{\xi}_l)-\lambda_{t,0}^*\tilde{\zeta}_l-\sum_{j=1}^{m_t}\lambda_{t,j}^*g_{t,j}(\tilde{\xi}_l)\\
        &\le \max_{l\in E_{t,k}}\ulQ_t(x_{t-1};\tilde{\xi}_l)-\lambda_{t,0}^*\tilde{\zeta}_l-\sum_{j=1}^{m_t}\lambda_{t,j}^*g_{t,j}(\tilde{\xi}_l)+\gamma_{t,k},
    \end{align*}
    and consequently \(v_{t-1}\ge\calQ_{t-1}(x_{t-1})\) since \(\lambda_t^*\) is also a feasible solution to the minimization in~\eqref{eq:ConvexCostRecursionUnbounded}.
    Finally, the gap is controlled since \(v_{t-1}-\calV_{t-1}(x_{t-1})=\frac{1}{n_t}\sum_{k=1}^{n_t}\gamma_{t,k}\le\gamma_{t,k^*}=\olcQ_t^i(x_t)-\ulcQ_t^i(x_t)\).
\end{proof}

In general, the size of the extreme point set $E_{t,k}$ can grow exponentially with respect to the uncertainty dimension $\delta_t$.
There are two potential remedies.
(1) The for-loop on line~\ref{alg:convexSSSO:enumeration} in Algorithm~\ref{alg:convexSSSO} can be fully parallelized, improving the efficiency of the enumeration step.
We adopt this strategy in Section~\ref{sec:Numerical}.
(2) Alternatively, one may want to reformulate the recursion~\eqref{eq:MDRO-FiniteDimensionalRecursion} as a mixed-integer linear optimization following~\cite[Section 3.2]{georghiou2019robust} instead of calling Algorithm~\ref{alg:convexSSSO}, which could lead to better efficiency particularly on larger dimensions $\delta_t$.

\section{Numerical Experiments}
\label{sec:Numerical}

In this section, we first introduce baseline models used for comparison against the DR-MCO model~\eqref{eq:MDRO-ExtensiveForm}.
Then we present comprehensive numerical studies of two application problems: the multi-commodity inventory problem with either uncertain demands or uncertain prices, and the hydro-thermal power system planning problem with uncertain water inflows.

\subsection{Baseline Models and Experiment Settings}
\label{sec:Numerical:Settings}

For performance comparison, we introduce three types of baseline models in addition to the DR-MCO with Wasserstein ambiguity sets~\eqref{eq:WassersteinAmbiguitySet}.
The first baseline model is the simple multistage robust convex optimization (MRCO) model, where we simply consider the worst-case outcome out of the uncertainty set \(\Xi_t\) in each stage \(t\).
Namely, the cost-to-go functions of the MRCO can be defined recursively as
\begin{equation}\label{eq:Baseline-MRCO}
    \calQ_{t-1}^\Robust(x_{t-1}):=\sup_{\xi_t\in\Xi_t}\min_{x_t\in\calX_t}f_t(x_{t-1},x_t;\xi_t)+\calQ_t^\Robust(x_t),\quad t=T,T-1,\dots,2.
\end{equation}
When the sum \(f_t(x_{t-1},x_t;\xi_t)+\calQ_t^\Robust(x_t)\) is jointly convex in the state \(x_t\) and the uncertainty \(\xi_t\) for any given \(x_{t-1}\), then the supremum can be attained at some extreme point of the convex hull of \(\Xi_t\) if it is finite.
In particular, if we have relatively complete recourse, (i.e., the sum is always finite for any given \(x_{t-1}\)), and if \(\Xi_t\) is a polytope, (i.e., it is a convex hull of finitely many points), then we can enumerate over the extreme points of \(\Xi_t\) to find the supremum, which allows us to solve the simple MRCO by Algorithm~\ref{alg:ConsecutiveDualDP}.
In general, if the uncertainty set \(\Xi_t\) is unbounded, then the cost-to-go functions of the MRCO model can take \(+\infty\) everywhere, so we will only use the baseline MRCO model when we have polytope uncertainty sets \(\Xi_t\).

The second type of baseline models consist of risk-neutral and risk-averse multistage stochastic convex optimization (MSCO) models.
The nominal probability measures in the MSCO models can be either the empirical measure \(\hat{\nu}_t\), or a probability measure associated with the sample average approximation (SAA) of a fitted probability measure, which we denote as \(\tilde{\nu}_t=\frac{1}{n'_t}\sum_{k=1}^{n'_t}\Delta_{\tilde{\xi}_{t,k}}\).
The main difference here is that the outcomes in an SAA probability measure \(\tilde{\xi}_{t,1},\dots,\tilde{\xi}_{t,n'_t}\) can be different from those given by the empirical measure \(\hat{\xi}_{t,1},\dots,\hat{\xi}_{t,n_t}\).
Moreover, we are able to take \(n'_t>n_t\) for a potentially better training effect.
To ease the notation, we also allow \(\tilde{\nu}_t=\hat{\nu}_t\) to happen when we describe the risk measures in the rest of this section.

For the risk-averse MSCO models, we use the risk measure that is called conditional value-at-risk (CVaR, a.k.a.\ average value-at-risk or expected shortfall).
Its coherence leads to a dual representation~\cite{philpott2013solving}, that allows the risk-averse MSCO models solved by Algorithm~\ref{alg:ConsecutiveDualDP} with a straightforward implementation of SSSO.
For simplicity, we only introduce the CVaR risk-averse MSCO based on this dual representation, and any interested reader is referred to~\cite{shapiro2021lectures} for the primal definition and the proof of duality.

Given parameters \(\alpha\in(0,1)\) and \(\beta\in[0,1]\), we define the cost-to-go functions associated with the \((\alpha,\beta)\)-CVaR risk measures recursively for \(t=T,T-1,\dots,2\) as
\begin{equation}\label{eq:Baseline-MSCO-CVaR}
    \calQ_{t-1}^\CVaR(x_{t-1}):=\max_{p_t\in\calP_t^\CVaR}\sum_{k=1}^{n'_t}p_{t,k}\biggl\{\min_{x_t\in\calX_t}f_t(x_{t-1},x_t;\tilde{\xi}_{t,k})+\calQ_t^\CVaR(x_t)\biggr\},
\end{equation}
where the ambiguity set is defined as
\begin{equation}\label{eq:Baseline-MSCO-CVaR-Ambiguity}
    \calP_t^\CVaR:=\left\{
    p_t=(p_{t,1},\dots,p_{t,n'_t})\in\bbR^{n'_t}:
     0\le p_{t,k}\le\frac{\beta}{n_t}+\frac{1-\beta}{\alpha n_t},\ 
     k=1,\dots,n'_t;\ 
     \sum_{k=1}^{n'_t}p_{t,k}=1\right\}.
\end{equation}
Note that when \(\beta=1\), the ambiguity set \(\calP_t^\CVaR=\{(\frac{1}{n'_t},\dots,\frac{1}{n'_t})\}\) has only one element corresponding to the SAA probability measure \(\tilde{\nu}_t\).
Thus the CVaR risk-averse MSCO model~\eqref{eq:Baseline-MSCO-CVaR} reduces to the risk-neutral nominal MSCO in this case.
Alternatively, if \(\beta=0\) and \(\alpha\le \frac{1}{n_t}\), then the CVaR risk-averse MSCO model~\eqref{eq:Baseline-MSCO-CVaR} considers only the worst outcome of the SAA probability measure in each stage.

The third type of baseline models are special versions of our DR-MCO model~\eqref{eq:WassersteinAmbiguitySet}: the probability measures in the Wasserstein ambiguity sets are restricted to those with the same support as the nominal probability measures \(\hat{\nu}_t\).
We refer to such a DR-MCO model as a DR-MCO with \emph{restricted Wasserstein ambiguity sets}.
To be more precise, the cost-to-go functions associated with such restricted Wasserstein ambiguity sets are defined as
\begin{equation}\label{eq:Baseline-MDRO-RWass}
    \calQ_{t-1}^\RWass(x_{t-1}):=\max_{p_t\in\calP_t^\RWass}\sum_{k=1}^{n_t}p_{t,k}\biggl\{\min_{x_t\in\calX_t}f_t(x_{t-1},x_t;\tilde{\xi}_{t,k})+\calQ_t^\RWass(x_t)\biggr\},
\end{equation}
where the ambiguity set is defined as
\begin{equation}\label{eq:Baseline-MDRO-RWass-Ambiguity}
    \calP_t^\RWass:=\left\{p_t=(p_{t,1},\dots,p_{t,n_t})\in\bbR^{n_t}:
    \begin{aligned}
        &\exists\,\pi\in\bbR^{n_t\times n_t},\ \pi_{k,k'}\ge0,\ \forall\,k,k'=1,\dots,n_t,\\
        &\sum_{k=1}^{n_t}\sum_{k'=1}^{n_t}d_t(\hat{\xi}_{t,k},\hat{\xi}_{t,k'})\pi_{k,k'}\le\rho_{t,0},\ \sum_{k=1}^{n_t}p_{t,k}=1,\\
        & \sum_{k=1}^{n_t}\pi_{k,k'}=\frac{1}{n_t},\ \forall\,k'=1,\dots,n_t,\\
        & \sum_{k'=1}^{n_t}\pi_{k,k'}=p_{t,k},\ \forall\,k=1,\dots,n_t,\\
        & \sum_{k=1}^{n_t}p_{t,k}g_{t,j}(\hat{\xi}_{t,k})\le\rho_{t,j},\ \forall\,j=1,\dots,m_t.
    \end{aligned}
    \right\}.
\end{equation}
The restricted Wasserstein ambiguity set~\eqref{eq:Baseline-MDRO-RWass-Ambiguity} is indeed a polyhedral set in the probability mass vector \(p_t\), and has been considered by~\cite{duque2020distributionally}.
We remark that all of the above baseline models, the simple MRCO model~\eqref{eq:Baseline-MRCO}, the CVaR risk-averse MSCO model~\eqref{eq:Baseline-MSCO-CVaR}, and the DR-MCO with restricted Wasserstein ambiguity sets~\eqref{eq:Baseline-MDRO-RWass} can be solved by Algorithm~\ref{alg:ConsecutiveDualDP} since only finitely many outcomes need to be considered in each stage \(t\).
More details on the SSSO for these baseline models can be found in our previous work~\cite{zhang2020distributionally}.

Our numerical experiments aim to demonstrate two attractive aspects of the DR-MCO models on some application problems:
better out-of-sample performance compared to the baseline models, and
ability to achieve out-of-sample performance guarantee with reasonable conservatism.
For ease of evaluation, we assume that we have the knowledge of the true underlying probability measure \(\nu\in\calM(\Xi_2\times\cdots\times\Xi_T)\), and thus the marginal probability measures \(\nu_t:=P^t_*(\nu)\), where \(P^t_*\) is the pushforward of the canonical projection map \(P^t:\Xi_2\times\cdots\Xi_T\to\Xi_t\), for \(t=2,\dots,T\).
Here, we do not restrict our attention to the case that \(\nu\) is a product of \(\nu_2,\dots,\nu_T\), i.e., the true probability measure is stagewise independent, so our modeling~\eqref{eq:MDRO-RecursiveForm} can be used as approximation for problems under general stochastic processes.
The experiments are then carried out in the following procedures with a uniform number of data points \(n_t=n\) in each stage.
\begin{compactenum}
    \item Draw \(n\) iid samples from \(\nu\) to form the empirical probability measures \(\hat{\nu}_t\);
    \item Construct the baseline models and DR-MCO models using \(\hat{\nu}_t\);
    \item Solve these models using our DDP algorithm (Algorithm~\ref{alg:ConsecutiveDualDP}) to a desired accuracy or within the maximum number of iterations or computation time;
    \item Draw \(N\) iid sample paths from \(\nu\) and evaluate the performance profiles (mean, variance, and quantiles) of the models on these sample paths.
\end{compactenum}
In particular, we focus on limited or moderate training sample sizes \(n\in\{5,10,20,40\}\), while keeping our sizes of evaluation sample paths to be relatively large (\(N=100,000\)).
In each independent test run of our numerical experiment, the training samples used in a smaller-sized test are kept in larger-sized tests, and the evaluation sample paths are held unchanged for all models and sample sizes.
We remark that our assumption of the knowledge on the true probability measure $\nu_t$ is simply to facilitate a more accurate evaluation of the obtained policy through a large $N$.
In practice, cross validation procedure is often used to calibrate the model and select the Wasserstein radius~\cite{duque2020distributionally,park2022data}.

Our algorithms and numerical examples are implemented using Julia 1.6~\cite{bezanson2017julia}, with Gurobi 9.0~\cite{gurobi} interfaced through the JuMP package (version 0.23)~\cite{dunning2017}.
We use 25 CPUs (24 for the worker processes and 1 for the manager process) with 50 GByte of RAM to allow parallelization of the SSSO (Algorithm~\ref{alg:convexSSSO}).

\subsection{Multi-commodity Inventory Problems}
\label{sec:Numerical:Inventory}

We consider a multi-commodity inventory problem which is adapted from the ones studied in~\cite{georghiou2019robust} and~\cite{zhang2020distributionally}. 
Let \(\calJ\coloneqq\{1,2,\dots,J\}\) denote the set of product indices.
We first describe the variables in each stage \(t\in\calT\). 
We use \(x^l_{t,j}\) to denote the variable of inventory level, \(y^a_{t,j}\) (resp. \(x^b_{t,j}\)) to denote the amount of express (resp.\ standard) order fulfilled in the current (resp.\ subsequent) stage, and \(y^r_{t,j}\) to denote the amount of rejected order of each product \(j\in\calJ\).
Let \(x_t:=(x^l_{t,1},\dots,x^l_{t,K},x^b_{t,1},\dots,x^b_{t,K})\) be the state variable and \(y_t:=(y^a_{t,1},\dots,y^a_{t,K},y^r_{t,1},\dots,y^r_{t,K})\) be the internal variable for each stage \(t\in\calT\).
The stage-\(t\) subproblem can be defined through the local cost functions \(f_t\) as
\begin{align}\label{eq:InventoryProblem}
    f_t(x_{t-1},x_t;\xi_t)\coloneqq
    \min_{y_t}\quad&
    C^F + \sum_{j\in\calJ}\left(C^a_{t,j} y^a_{t,j} + C^b_{t,j} x^b_{t,j} + C^r_j y^r_{t,j} + C^H_j[x^l_{t,j}]_+ + C^B_j[x^l_{t,j}]_-\right)\\
    \mathrm{s.t.}\quad 
    & \sum_{j\in\calJ}y^a_{t,j}\le B^c,\notag\\
    & x^l_{t,j}\le x^l_{t-1,j}+y^a_{t,j}+x^b_{t-1,j}+y^r_{t,j}-D_{t,j}, \forall\, j\in\calJ,\notag\\
    & y^a_{t,j}\in[0,B^a_j],\ y^r_{t,j}\in[0,D_{t,j}],\  \forall\,j\in\calJ,\notag\\
    & x^b_{t,j}\in[0,B^b_j],\ x^l_{t,j}\in[B^{l,-}_j,B^{l,+}_j],\  \forall\,j\in\calJ.\notag
\end{align}
In the definition~\eqref{eq:InventoryProblem}, we use \(C^a_{t,j}=C^a_{t,j}(\xi_t)\) (resp.\ \(C^b_{t,j}=C^b_{t,j}(\xi_t)\)) to denote the uncertain express (resp.\ standard) order unit cost,
\(C^H_j\) (resp.\ \(C^B_j\)) the inventory holding (resp.\ backlogging) unit cost, \(C^r_j\) the penalty on order rejections, \(C^F\equiv 1\) a positive fixed cost,
\(B^a_j\) (resp.\ \(B^b_j\)) the bound for the express (resp.\ standard) order, 
and \(B^{l,-}_j,B^{l,+}_j\) the bounds on the backlogging and inventory levels,
\(D_{t,j}=D_{t,j}(\xi_t)\) the uncertain demand for the product \(j\), respectively.
The first constraint in~\eqref{eq:InventoryProblem} is a cumulative bound \(B^c\) on the express orders, the second constraint characterizes the change in the inventory level, and the rest are bounds on the decision variables with respect to each product.
The notations \([x]_+\coloneqq\max\{x,0\}\) and \([x]_-\coloneqq-\min\{0,x\}\) are used to denote the positive and negative parts of a real number \(x\).
The initial state is given by \(x_{0,j}^b=x_{0,j}^l=0\) for all \(j\in\calJ\).
Before we discuss the uncertain parameters \(C^a_{t,j},C^b_{t,j}\) or \(D_{t,j}\), we make the following remarks on the definition~\eqref{eq:InventoryProblem}.

First, as long as the demand $D_{t,1},\dots,D_{t,J}$ are bounded from above, the internal variables $y_t$ is constrained in a compact set $\calY_t$, so the defined $f_t$ fits into our formulation~\eqref{eq:MDRO-ExtensiveForm} as discussed in Example~\ref{ex:ConstrainedFormulation}.
Second, it is easy to check that if \(C_{t,j}^a,C_{t,j}^b\) (resp.\ \(D_{t,j}\)) are deterministic, then Assumption~\ref{assum:ConvexUncertaintyCost} (resp.\ Assumption~\ref{assum:ConcaveUncertaintyCost}) is satisfied so we are able to apply the SSSO implementations discussed in Sections~\ref{sec:Algorithm:ConvexSingleStageSubproblem} and~\ref{sec:Algorithm:ConcaveSingleStageSubproblem}.
Third, as the bounds \(B_j^{l,-},B_j^{l,+}\) do not change with \(t\) and all orders can be rejected (i.e., \(y_{t,j}^r=D_{t,j}\) is feasible for all \(j\in\calJ\)), we see that the problem~\eqref{eq:InventoryProblem} has relatively complete recourse.
Fourth, the Lipschitz constant of the value functions \(Q_t(\cdot;\xi_t)\) is uniformly bounded by \(\sum_{j\in\calJ}C_j^r\), so Proposition~\ref{prop:LipschitzRegularization} can be applied here if we set the regularization factors to be sufficiently large \(M_t\ge\sum_{j\in\calJ}C_j^r\).
Besides, the Lipschitz continuity guarantees the in-sample adjustable conservatism by Theorem~\ref{thm:InSampleConservatism}.  
Last, since all state variables are bounded, together with the above observation, we know by Theorem~\ref{thm:CDDPComplexityBound} that Algorithm~\ref{alg:ConsecutiveDualDP} would always converge on our inventory problem~\eqref{eq:InventoryProblem}.

\subsubsection{Inventory Problems with Uncertain Demands}
\label{sec:InventoryUncertainDemand}

First, we consider the inventory problems with uncertain demands, where the goal is to seek a policy with minimum mean inventory cost plus the penalty on order rejections.
The uncertain demands are modeled by the following expression:
\begin{equation}
    D_{t,j}(\xi_{t}):=D_0\left[1+\cos\biggl(\frac{2\pi(t+j)}{\tau}\biggr)\right]+\bar{D}\cdot\xi_{t,j},\quad j\in\calJ,t\in\calT.
\end{equation}
Here, \(D_0\) is a factor and \(\tau\) is the period for the base demands, and \(\bar{D}\) is the bound on the uncertain demands.
The uncertainty set $\Xi_t:=[0,1]^J$ is a $J$-dimensional box, and the true distribution of the uncertainty vector $\xi_t$ is 
described as follows: 
\(\xi_{t,1}\sim\mathrm{Uniform}(0,1)\), and for \(j=2,\dots,J\), we have
\begin{equation}
    \xi_{t,j}\mid\xi_{t,j-1}\sim\left\{\begin{aligned}
        &\mathrm{Uniform}(0,(1+\xi_{t,j-1})/2),\quad &&\text{if }\xi_{t,j-1}\le\frac{1}{2},\\
        &\mathrm{Uniform}(\xi_{t,j-1}/2,1),\quad &&\text{otherwise}.
    \end{aligned}\right.
\end{equation}
For the experiments, we consider \(J=3\) products and \(T=\tau=5\) stages.
The unit prices of each product are deterministically set to \(C^a_{t,j}=5\) and \(C^b_{t,j}=1\) for all \(t\in\calT\);
the inventory and holding costs are \(C_j^H=2\) and \(C_j^B=10\), and the rejection costs are \(C_j^r=100\), for each \(j\in\calJ\).
The bounds are set to \(B^c=15\), \(B_j^a=10\), \(B_j^b=20\), \(B_j^{l,-}=10\), and \(B_j^{l,+}=100\) for each \(j\in\calJ\).
We pick the uncertainty parameters \(D_0=5\) and \(\bar{D}=50\).
We terminate the DDP algorithm if it reaches 1\% relative optimality or 2000 iterations.
For Wasserstein ambiguity sets, we only consider the radius constraint (i.e., \(m_t=0\)) with radius set to be relative to the following estimation of the distance among data points:
\begin{equation}\label{eq:RelativeWassersteinRadiusEstimation}
    \hat{d}_t:=\max_{k=1,\dots,n_t}W_t(\hat{\nu}_t,\Delta_{\hat{\xi}_{t,k}})=\max_{k=1,\dots,n_t}\frac{1}{n_t}\sum_{k'\neq k}\nvert{\hat{\xi}_{t,k}-\hat{\xi}_{t,k'}}.
\end{equation}
For the MSCO models, we directly use the empirical probability measures \(\hat{\nu}_t\) for each \(t\in\calT\).
Further, we consider parameters \(\alpha\in\{0.01,0.05,0.10\}\) and \(\beta\in\{0.0,0.25,0.50,0.75\}\) for the CVaR risk-averse MSCO models.

Using the experiment procedure described in Section~\ref{sec:Numerical:Settings}, we present the results of our data-driven DR-MCO model with Wasserstein ambiguity sets and the baseline models.

\begin{figure}[htbp]
    \centering
    \includegraphics[width=1.0\textwidth]{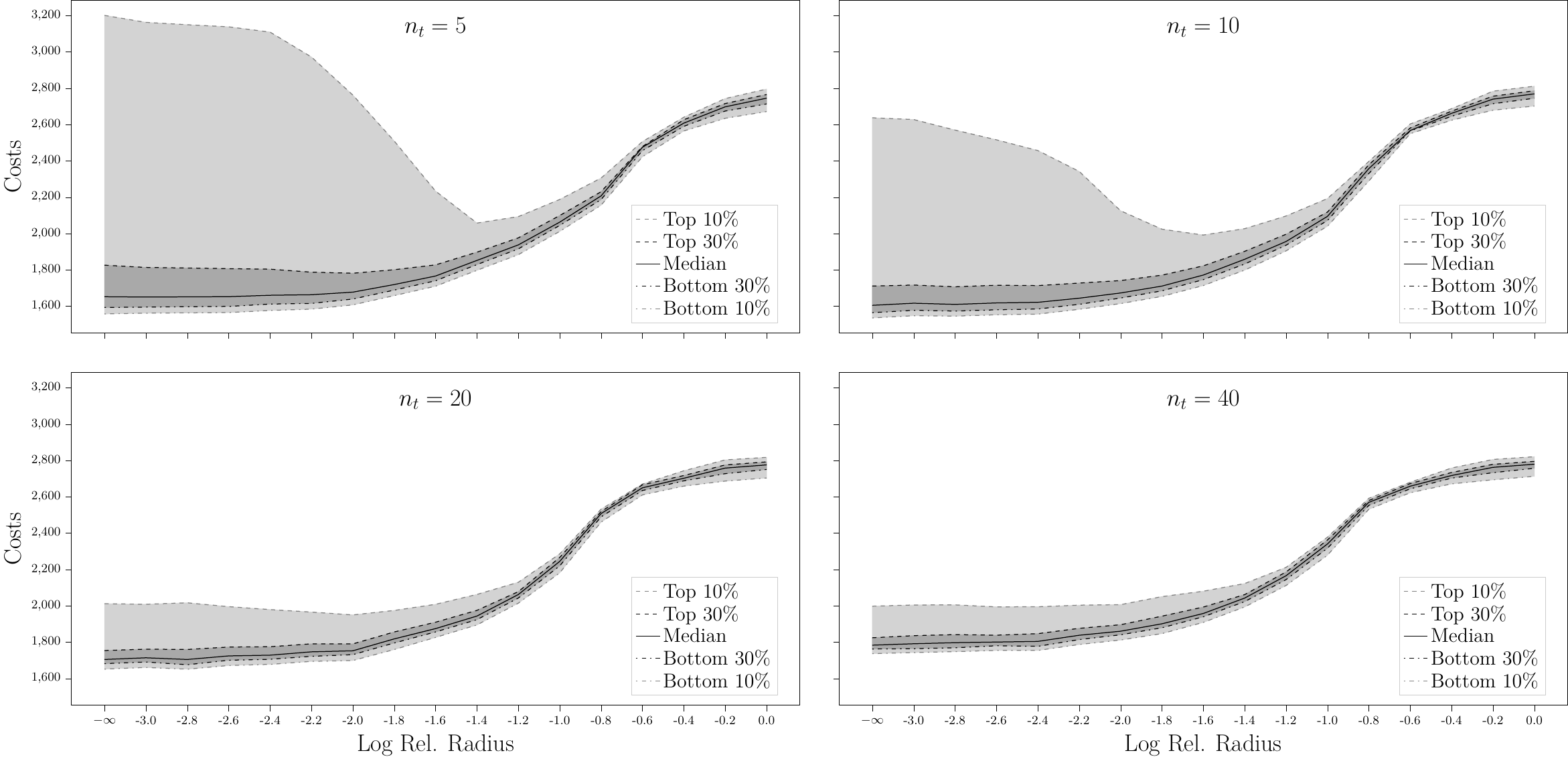}
    \caption{Out-of-sample Cost Quantiles for Different Radii on Multi-commodity Inventory with Uncertain Demands}
    \label{fig:DemandLinePlot}
\end{figure}

Figure~\ref{fig:DemandLinePlot} (and Figures~\ref{fig:DemandLinePlot1} and~\ref{fig:DemandLinePlot2} in Section~\ref{sec:SupplementalResults}) displays the out-of-sample cost quantiles of the nominal stochastic model and the DR-MCO models with different Wasserstein radii, constructed from the empirical probability measures \(\hat{\nu}_t\).
Here we use the log radius \(-\infty\) to denote the nominal stochastic model, i.e., \(\rho_{t,0}=0\).
From the plot, we see that in small-sample case (\(n_t=5\)), the Wasserstein DR-MCO model significantly reduces the top 10\% out-of-sample evaluation costs when the radius is set to be $10^{-1.6}$-$10^{-1.2}$ of the estimation \(\hat{d}_t\), and as we will see in Figure~\ref{fig:DemandScatterPlot} below, it reduces the out-of-sample mean consequently.
Moreover, the difference between top 10\% and bottom 10\% of the out-of-sample evaluation costs becomes smaller around \(10^{-0.8}\cdot\hat{d}_t\) even for larger sample sizes.
However, the median out-of-sample cost increases with the Wasserstein radius, suggesting that larger Wasserstein radii may lead to overly conservative policies.

\begin{figure}[htbp]
    \centering
    \includegraphics[width=1.0\textwidth]{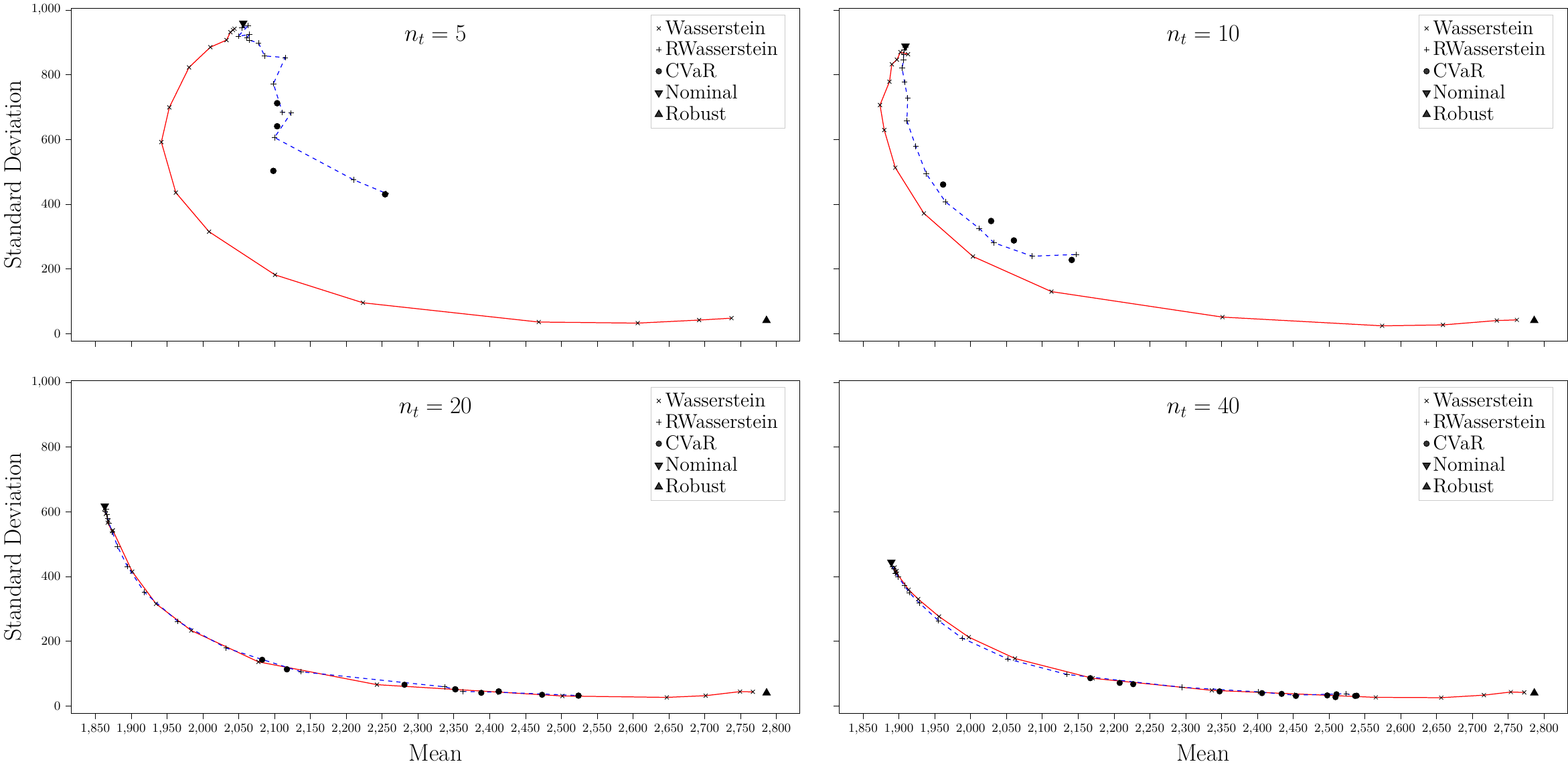}
    \caption{Comparison against Baseline Models on Multi-commodity Inventory with Uncertain Demands}
    \label{fig:DemandScatterPlot}
\end{figure}

To better quantify the trade-off between mean and variance of the out-of-sample evaluation costs, we present Figure~\ref{fig:DemandScatterPlot} (and Figures~\ref{fig:DemandScatterPlot1} and~\ref{fig:DemandScatterPlot2} in Section~\ref{sec:SupplementalResults}).
Here, the lines connect the points representing our Wasserstein DR-MCO models~\eqref{eq:WassersteinAmbiguitySet} from the smallest radius to the largest one.
The policies obtained from DR-MCO with restricted Wasserstein ambiguity sets~\eqref{eq:Baseline-MDRO-RWass} are labeled by \(\mathrm{RWasserstein}\).
We say one policy dominates another policy if the former has smaller mean and standard deviation than the latter does, and have the following observations.
First, in all cases, the policy from the MRCO model is dominated by some policy from the Wasserstein DR-MCO model, and also by some CVaR MSCO model when \(n_t\) is large.
Second, our DR-MCO model with the Wasserstein ambiguity sets~\eqref{eq:WassersteinAmbiguitySet} leads to policies that dominate those from the CVaR MSCO models, as well as those from the DR-MCO with restricted Wasserstein ambiguity sets, when the sample size is relatively small, \(n_t=5\) or \(10\).
This reconfirms the value of considering outcomes of the uncertainty \(\xi_t\) other than the data \(\hat{\xi}_{t,1},\dots,\hat{\xi}_{t,n_t}\) when the data size is limited.
Last, policies from the CVaR MSCO model or the DR-MCO model with restricted Wasserstein ambiguity sets are usually quite close to, and could sometimes penetrate the frontier formed by policies from the DR-MCO model with Wasserstein ambiguity sets, when the sample sizes increase to \(n_t=20\) or \(40\).
This phenomenon contrasts with the limited data size cases, and supports the common practice of using CVaR MSCO or DR-MCO with restricted Wasserstein ambiguity sets.

\subsubsection{Inventory Problems with Uncertain Prices}
\label{sec:InventoryUncertainPrice}

Now we discuss the inventory problems with uncertain prices and fixed demands.
Such problem can be viewed as a simplified model for supply contract problems~\cite{li1999flexible}, where the goal is to estimate the total cost of such supply contract and under-estimation is undesirable.
The uncertain prices are modeled by the following expression:
\begin{equation}
    C^b_{t,j}(\xi_{t}):=\xi_{t,j},\quad C^a_{t,j}(\xi_t):=C_1\cdot \xi_{t,j},\quad j\in\calJ,t\in\calT.
\end{equation}
Here, \(C_1\) is a factor for express orders.
The uncertain vector \(\xi_t\) follows a lower-truncated multivariate normal distribution supported on $\Xi_t:=[\ubar{C},+\infty)^J\subseteq\bbR_{\ge0}^J$:
\begin{equation}
    \xi_{t}:=\max\left\{\mathrm{Normal}(\mu_t,\bar{C}\cdot\Sigma_t),\ubar{C}\right\},\quad
    \mu_t:=C_0\left[1+\sin\bigl(2\pi(t+j)/\tau\bigr)\right],
\end{equation}
where the maximum is taken componentwise, \(C_0\) is a factor for base prices, \(\tau\) is the period, \(\bar{C}\) is the magnitude on the price variation, \(\ubar{C}\ge0\) is the lower bound on the prices, and the covariance matrix \(\Sigma_t\) is randomly generated (by multiplying a uniformly distributed random matrix with its transpose) and normalized to have its maximum eigenvalue equal to 1.
The demands are deterministically given by
\begin{equation}
    D_{t,j}:=D_0\left[1+\cos\bigl(2\pi(t+j)/\tau\bigr)\right]+\bar{D},\quad j\in\calJ,t\in\calT.
\end{equation}
For the experiments, we consider \(J=5\) products, \(T=10\) stages, and the period \(\tau=5\).
The price uncertainty has parameters \(C_0=1\), \(C_1=5\), \(\bar{C}=0.1\), and \(\ubar{C}=0.001\).
We choose the demand parameters \(D_0=5\) and \(\bar{D}=10\).
The inventory and holding costs are \(C_j^H=1\) and \(C_j^B=10\), and the rejection costs are \(C_j^r=100\), for each \(j\in\calJ\).
The bounds are set to \(B^c=15\), \(B_j^a=10\), \(B_j^b=20\), \(B_j^{l,-}=20\), and \(B_j^{l,+}=20\) for each \(j\in\calJ\).
The Wasserstein radii in the DR-MCO models are set relatively with respect to \(\hat{d}_t\) defined in~\eqref{eq:RelativeWassersteinRadiusEstimation}.
The baseline MSCO models are constructed in the same way as described in Section~\ref{sec:InventoryUncertainDemand}.
We do not consider the baseline MRCO model here as the uncertainty set \(\Xi_t\) is unbounded.

\begin{figure}[htbp]
    \centering
    \includegraphics[width=1.0\textwidth]{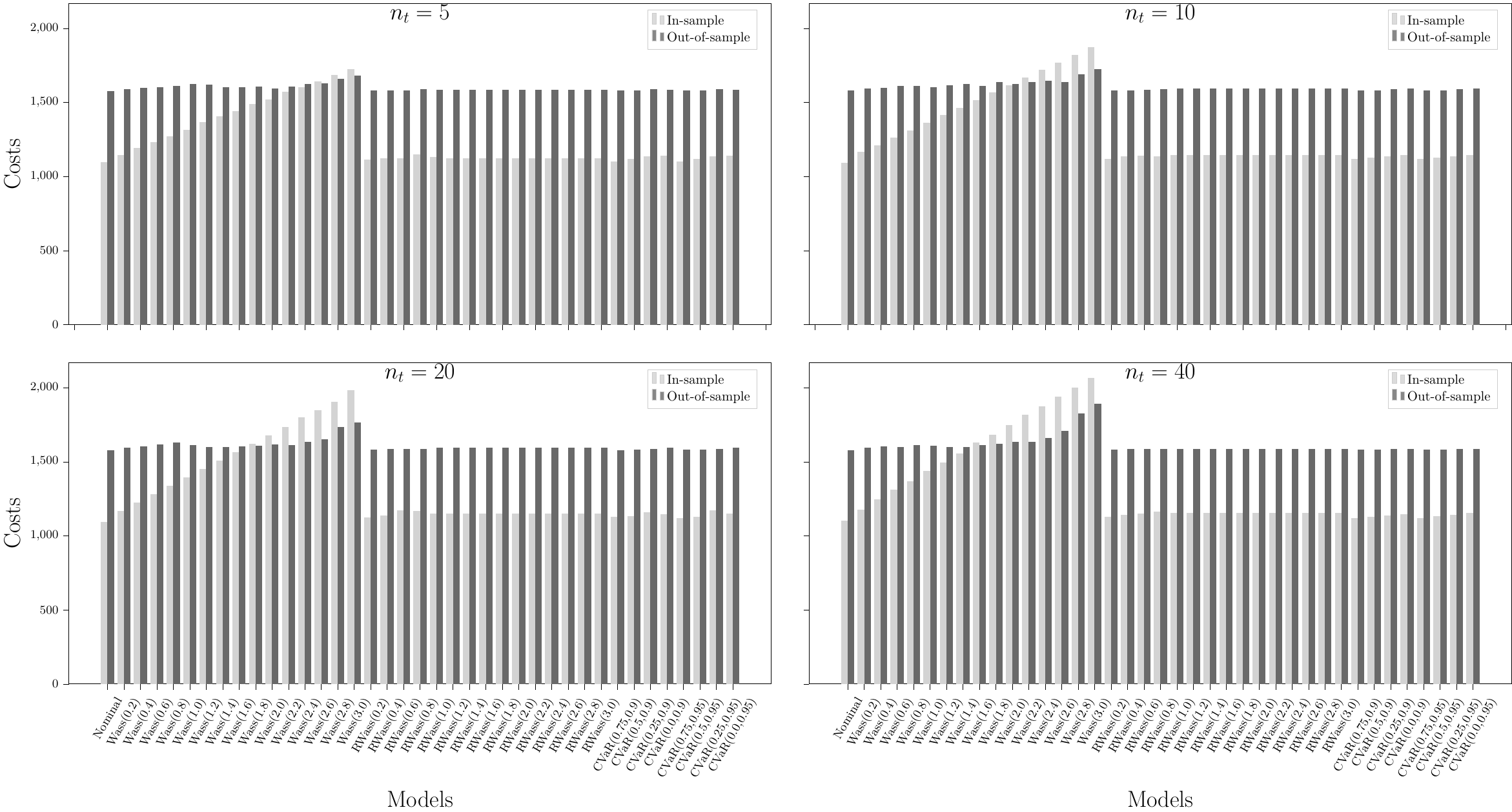}
    \caption{In-sample and Out-of-sample Mean Costs on Multi-commodity Inventory with Uncertain Prices}
    \label{fig:SupplyPlot}
\end{figure}

We plot the in-sample objective costs and out-of-sample mean evaluation costs in Figure~\ref{fig:SupplyPlot} (and Figures~\ref{fig:SupplyPlot1} and~\ref{fig:SupplyPlot2} in Section~\ref{sec:SupplementalResults}).
The label \(\mathrm{Nominal}\) refers to the nominal risk-neutral MSCO model using the empirical probability measures \(\hat{\nu}_t\);
\(\mathrm{Wass}(\gamma)\) and \(\mathrm{RWass}(\gamma)\) refer to the DR-MCO models with Wasserstein and restricted Wasserstein ambiguity sets, in which the radius \(\rho_{t,0}=\gamma\cdot\hat{d}_t\) in each stage \(t\ge 2\), respectively; 
and \(\mathrm{CVaR}(\alpha,\beta)\) refers to the CVaR risk-averse MSCO model with parameters \(\alpha\) and \(\beta\).
As the uncertainty vectors now have an unbounded support, the robust model is no longer applicable.
We see that in all cases, the in-sample objective cost grows linearly with respect to the Wasserstein distance, as predicted by Theorem~\ref{thm:InSampleConservatism}.
As the nominal stochastic model inevitably under-estimates the mean evaluation costs, using Wasserstein DR-MCO models with a relative radius \(\gamma\in[1.6,2.4]\) depending on the sample size \(n_t\) could achieve the out-of-sample performance guarantee in almost all test cases.
Moreover, none of the CVaR MSCO models, or the DR-MCO models with restricted Wasserstein ambiguity sets in the experiments could achieve similar effect.
In fact, once \(\gamma\ge1\) (see~\eqref{eq:RelativeWassersteinRadiusEstimation}), the DR-MCO models with restricted Wasserstein ambiguity sets consider only the worst-case outcome of the existing \(n_t\) samples in each stage, and thus the in-sample values do not further increase with the radius.
Thus we believe that the DR-MCO models with Wasserstein ambiguity sets~\eqref{eq:WassersteinAmbiguitySet} are particularly more favorable in the context of supply contracts.
It is however worth mentioning that we do not observe any improvement of the mean or the variance of evaluation costs from the Wasserstein DR-MCO model over the baseline models.

\subsection{Hydro-thermal Power Planning Problem}
\label{sec:Numerical:Hydrothermal}

We next consider the Brazilian interconnected power system planning problem described in~\cite{ding2019python}.
Let \(\calJ=\{1,\dots,J\}\) denote the indices of four regions in the system, and \(\calL=\cup_{j\in\calJ}\calL_j\) the indices of thermal power plants, where each of the disjoint subsets \(\calL_j\) is associated with the region \(j\in\calJ\).
We first describe the decision variables in each stage \(t\in\calT\). 
We use \(x^l_{t,j}\) to denote the stored energy level, \(y^h_{t,j}\) to denote the hydro power generation of some region \(j\in\calJ\); 
and \(y^g_{t,l}\) to denote the thermal power generation for some thermal power plant \(l\in\calL\).
For two different regions \(j\neq j'\in\calJ\), we use \(y^e_{t,j,j'}\) to denote the energy exchange from region \(j\) to region \(j'\), and \(y^a_{t,j,j'}\) to denote the deficit account for region \(j\) in region \(j'\).
Let \(x_t:=(x^l_{t,1},\dots,x^l_{t,J})\) be the state vector of energy levels, \(y_t\) the internal decision vector consisting of \(y^h_{t,j}\) for \(j\in\calJ\), 
\(y^g_{t,l}\) for all \(l\in\calL\),  
\(y^e_{t,j,j'}\) and \(y^a_{t,j,j'}\) for any \(j\neq j'\in\calJ\);
and \((\xi_{t,1},\dots,\xi_{t,J})\) the uncertain vector energy inflows in stage \(t\) that is supported on the set $\Xi_t:=\bbR_{\ge0}^J$.
We define the DR-MCO by specifying each \(f_t\) as
\begin{align}\label{eq:PowerSystem}
    f_t(x_{t-1},x_t;\xi_t)\coloneqq
    \min_{y_t}\quad&\sum_{j\in\calJ}\bigg(\sum_{l\in\calL_j}C^g_l y^g_{t,l} + \sum_{j'\neq j}\big(C^e_{j,j'} y^e_{t,j,j'} + C^a_{j,j'} y^a_{t,j,j'}\big)\\
                   &\qquad + C^s(x_{t-1,j}^l+\xi_{t,j}-x_{t,j}^l-y_{t,j}^h) \bigg)\notag\\
    \mathrm{s.t.}\quad& x^l_{t,j}+y^h_{t,j}\le x^l_{t-1,j}+\xi_{t,j},\ \forall\,j\in\calJ,\notag\\
    & y^h_{t,j}+\sum_{l\in\calL_j}y^g_{t,l}+\sum_{j'\neq j}(y^a_{t,j,j'}-y^e_{t,j,j'}+y^e_{t,j',j})=D_{t,j}, \ \forall\,j\in\calJ,\notag\\
    & y^g_{t,l}\in[B^{g,-}_l,B^{g,+}_l],\ \forall\,l\in\calL,\notag\\
    & x^l_{t,j}\in[0,B^l_j],\ y^h_{t,j}\in[0,B^h_j],\  \forall\,j\in\calJ,\notag\\
    & y^a_{t,j,j'}\in[0,B^a_{j,j'}],\ y^e_{t,j,j'}\in[0,B^e_{j,j'}],\  \forall\,j,j'\in\calJ.\notag
\end{align}
Here in the formulation, \(C^s\) denotes the unit penalty on energy spillage $x_{t-1,j}^l+\xi_{t,j}-x_{t,j}^l-y_{t,j}^h$,
\(C^g_l\) the unit cost of thermal power generation of plant \(l\),
\(C^e_{j,j'}\) the unit cost of power exchange from region \(j\) to region \(j'\),
\(C^a_{j,j'}\) the unit cost on the energy deficit account for region \(j\) in region \(k'\),
\(D_{t,j}\) the deterministic power demand in stage \(t\) and region \(k\),
\(B^l_j\) the bound on the storage level in region \(j\),
\(B^h_j\) the bound on hydro power generation in region \(j\),
\(B^{g,-}_l,B^{g,+}_l\) the lower and upper bounds of thermal power generation in plant \(l\),
\(B^a_{j,j'}\) the bound on the deficit account for region \(j\) in region \(j'\) such that \(\sum_{j'\neq j}B_{j,j'}^a=D_{t,j}\),
and \(B^e_{j,j'}\) the bound on the energy exchange from region \(j\) to region \(j'\).
The first constraint in~\eqref{eq:PowerSystem} characterizes the change of energy storage levels in each region, the second constraint imposes the power generation-demand balance for each region, and the rest are bounds on the decision variables.
The initial state \(x_0\) and uncertainty vector \(\xi_1\) are given by data.
In our experiment, we consider \(T=13\) and all other parameters used in this problem can be found in~\cite{ding2019python}.

For the problem~\eqref{eq:PowerSystem}, we have the following remarks.
First, we always have the relatively complete recourse as we allow spillage for extra energy inflows and deficit for energy demands in each region.
Then it is straightforward to check that the Lipschitz constant of \(f_t\) in either \(x_{t-1}\) or \(\xi_t\) can be bounded by the maximum of the deficit cost \(C_{j,j'}^a\) and the spillage penalty \(C^s\).
Second, as now the uncertainty has an unbounded support \(\bbR^J_{\ge0}\), we need to estimate the growth rate of the value function \(r_t(x_{t-1})\).
Note that if for any region the inflow is sufficiently large such that all demands and energy exchanges have met their upper bounds, then the only cost incurred by further increasing the inflow is simply the spillage penalty.
Thus we conclude that \(r_t(x_{t-1})=C^s\), which is a constant function for all \(x_{t-1}\in\calX_{t-1}\).
Last, it is easy to see that Assumption~\ref{assum:ConvexUncertaintyCost} holds for the problem~\eqref{eq:PowerSystem} and that the state variables \(\calX_t\) is compact, so our SSSO implementation (Algorithm~\ref{alg:convexSSSO}) would guarantee the convergence of our DDP algorithm (Algorithm~\ref{alg:ConsecutiveDualDP}).

Now we assume that the true uncertainty can be described by the following logarithmic autoregressive time series:
\begin{equation}\label{eq:InflowTimeSeries}
    \ln{\xi_t}-\mu_t=\phi_t(\ln{\xi_{t-1}}-\mu_{t-1})+\epsilon_t,\quad
    \epsilon_t\sim\mathrm{Normal}(0,\Sigma_t),
\end{equation}
where the logarithm and the product are taken componentwise, and the parameters \(\mu_t,\phi_t\in\bbR^J\) and \(\Sigma_t\in\calS_{\succeq0}^J\) are fit from historical data (see modeling details in~\cite{shapiro2012final} and coding details in~\cite{ding2019python}).
Note that~\eqref{eq:InflowTimeSeries} is not linear with respect to the uncertainty vectors \(\xi_t\), and consequently it cannot directly be reformulated into a stagewise independent MSCO (or a DR-MCO)~\cite{lohndorf2019modeling}.
While there are approaches based on Markov chain DDP or linearized version of the model~\eqref{eq:InflowTimeSeries}, they would require alteration of the DDP algorithm or an increase in the state space dimension.
Alternatively, we would like to study the effects of the stagewise independence assumption~\cite{duque2020distributionally} and the Wasserstein ambiguity sets in our DR-MCO model~\eqref{eq:MDRO-RecursiveForm}.

We can see that under the assumption on the true uncertainty~\eqref{eq:InflowTimeSeries}, each uncertainty \(\xi_t\) follows a multivariate lognormal distribution.
Thus instead of directly using the empirical probability measures \(\hat{\nu}_t\) in the MSCO models, we can fit lognormal distributions based on the empirical outcomes \(\hat{\xi}_{t,1},\dots,\hat{\xi}_{t,n_t}\), from which we further construct the SAA probability measures \(\tilde{\nu}_t\) for the MSCO models.
Moreover, we can also use this SAA probability measure \(\tilde{\nu}_t\) to estimate the Wasserstein distance bound \(\rho_{t,0}\) using
\begin{equation}\label{eq:SAAWassersteinRadiusEstimation}
    \begin{aligned}
        \tilde{d}_t:=W_t(\hat{\nu}_t,\tilde{\nu}_t)=
        \min_{\pi_{k,k'}\ge0}\quad & \sum_{k=1}^{n_t}\sum_{k'=1}^{n'_t}\pi_{k,k'}\cdot d_t(\hat{\xi}_{t,k},\tilde{\xi}_{t,k'})\\
        \mathrm{s.t.}\quad& \sum_{k'=1}^{n'_t}\pi_{k,k'}=\frac{1}{n_t},\quad k=1,\dots,n_t,\\
                          & \sum_{k=1}^{n_t}\pi_{k,k'}=\frac{1}{n'_t},\quad k=1,\dots,n'_t.
    \end{aligned}
\end{equation}
We then set the Wasserstein radius to be \(\rho_{t,0}=\gamma\cdot\tilde{d}_t\) for the relative factors \(\gamma=10^{-2.0},10^{-1.8},\dots,10^{-0.2},1.0\).
For the baseline risk-averse MSCO models, we use CVaR parameters \(\alpha=0.1\) and \(\beta\in\{0.0,0.1,\dots,0.9,1.0\}\).

Note that as the SAA resampling step is random, the performance of our DR-MCO models and MSCO models would also be random.
In addition, as it is often very challenging to solve the problem~\eqref{eq:PowerSystem} to certain optimality gap, we choose to terminate it with a maximum of 1000 iterations, and allow random sampling in the Algorithm~\ref{alg:ConsecutiveDualDP}, in which the noninitial stage step does not strictly follow Definition~\ref{def:NoninitialStageOracle} and guarantees only the validness.
The benefit of such random sampling is that empirically the under-approximations (hence the policies) often improve faster especially in the beginning stage of the algorithms.
We refer any interested readers to stochastic DDP literature (e.g.,~\cite{baucke2017deterministic,zhang2022stochastic,lan2020complexity}) for more details.

\begin{figure}[htbp]
    \centering
    \includegraphics[width=1.0\textwidth]{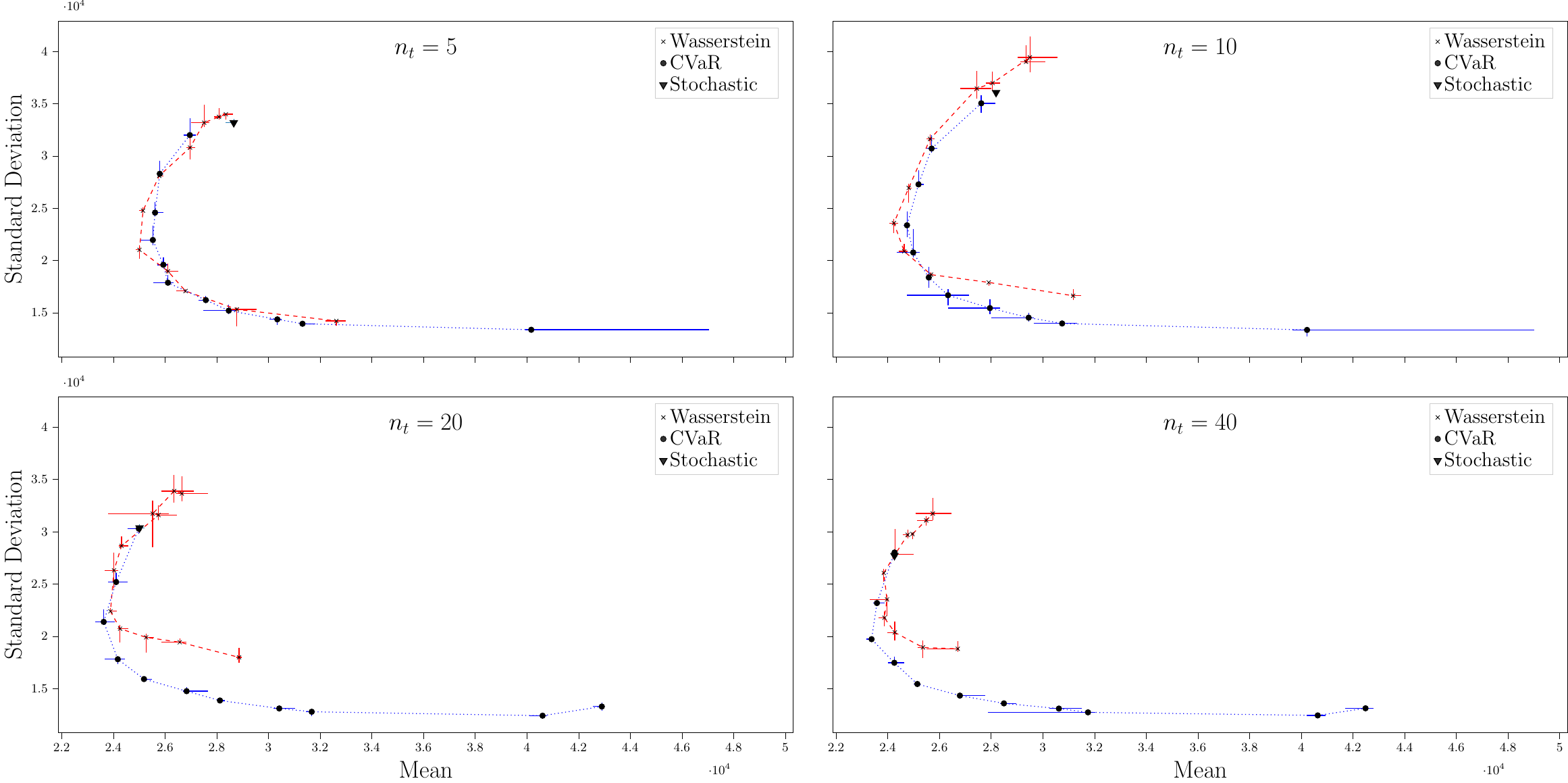}
    \caption{Comparison against Baseline Models on Hydro-thermal Power Planning}
    \label{fig:HydroPlot}
\end{figure}

Due to the randomness of the models and the algorithm, we repeat each test 3 times with the same empirical probability measures \(\hat{\nu}_t\) in the experiment procedure.
Given the significant computational requirements of the hydro-thermal problems, we focus the comparison of our Wasserstein DR-MCO~\eqref{eq:WassersteinAmbiguitySet} against the most critical benchmarks: the nominal and the CVaR MSCO models. These are widely recognized as the most popular models for this problem class ~\cite{shapiro2012final,shapiro2013risk}.
Then we plot the median values with error bars indicating the maximum and minimum values in Figure~\ref{fig:HydroPlot} (and Figures~\ref{fig:HydroPlot1} and~\ref{fig:HydroPlot2} in Section~\ref{sec:SupplementalResults}) with increasing values of \(\gamma\) in the Wasserstein DR-MCO models and decreasing values of \(\beta\) in the CVaR risk-averse MSCO models.
First, we see that most of the Wasserstein DR-MCO models and the CVaR risk-averse MSCO models achieve better performances in either the mean or the standard deviation of evaluation costs, compared with the risk-neutral MSCO model.
This observation justifies the usage of ambiguity sets or risk aversion when we approximate a stagewise dependent problem with stagewise independent models. 
Second, when the sample sizes are small (\(n_t=5\) or \(10\)), we see that the Wasserstein DR-MCO could outperform the CVaR risk-averse MSCO models (e.g., for \(\gamma\in[10^{-1.0},10^{-0.6}]\)) in the out-of-sample mean cost.
Third, as the sample size grows to \(n_t=20\) or \(40\), while the policies obtained from Wasserstein DR-MCO models are always dominated by those from the CVaR risk-averse MSCO models, the latter could achieve both lower mean cost and standard deviation.
Our conjecture is that for larger sample sizes, the cost-to-go functions of the Wasserstein DR-MCO models have more complicated shapes, thus making it hard to approximate in the limited 1000 iterations.
This also suggests that the CVaR risk-averse MSCO models could lead to good out-of-sample performances when an probability distribution fitting is possible and the computation budget is limited.

\section{Concluding Remarks}
\label{sec:Conclusion}

In this work, we study the data-driven DR-MCO models with Wasserstein ambiguity sets.
We show that with a sufficiently large Wasserstein radius, such DR-MCO model satisfies out-of-sample performance guarantee with high probability even with limited data sizes.
Using convex dual reformulation, we show that the in-sample conservatism is linearly bounded by the radius when the value functions are Lipschitz continuous in the uncertainties.
To numerically solve the data-driven DR-MCO models, we design exact SSSO subproblems for the DDP algorithms by exploiting either the concavity of the convexity of the cost functions in terms of the uncertainties.
We conduct extensive numerical experiments to compare the performance of our DR-MCO models against MRCO models, risk-neutral, risk-averse MSCO models  and DR-MCO models with restricted Wasserstein ambiguity sets.
On the multi-commodity inventory problems with uncertain demands, we observe that the DR-MCO models are able to provide policies that dominate those baseline model policies in the out-of-sample evaluations when the in-sample data size is small.
Moreover, on the inventory problems with uncertain prices, the DR-MCO models could achieve out-of-sample performance guarantee with little compromise of the objective value, which has not been achieved by the baseline models.
On the hydro-thermal power planning problems with uncertain energy inflows, we see that with limited number of iterations, while the policies from the DR-MCO models could achieve better out-of-sample performances than the risk-averse MSCO models for small data sizes, they are dominated by the risk-averse MSCO models for larger data sizes.
We hope these numerical experiments could serve as benchmarks for future studies on DR-MCO problems.

\bibliographystyle{plain}
\bibliography{ref}

\clearpage
\appendix
\section{Finite-dimensional Dual Recursion for DR-MCO}
\label{sec:DualRecursion}

In this section, we briefly review some general strong Lagrangian duality theory and then apply it to derive our finite-dimensional dual recursion for our DR-MCO problems~\eqref{eq:MDRO-RecursiveForm}.

\subsection{Generalized Slater Condition and Lagrangian Duality}
\label{sec:StrongDuality}

Given an \(\bbR\)-vector space \(\calM\), we consider the following optimization problem.
\begin{align}\label{eq:GeneralPrimalProblem}
    v^{\rmP}:=\inf_{\mu\in C}\quad& \varphi_0(\mu)\\
    \mathrm{s.t.}\quad&\varphi_j(\mu)\le 0,\quad j=1,\dots,l,\notag\\
    &\varphi_j(\mu)= 0,\quad j=l+1,\dots,m.\notag
\end{align}
Here, \(C\subset\calM\) is a convex subset, the functions \(\phi_j:\calM\to\bbR\cup\{+\infty\}\) are convex for each \(j=0,1,\dots,l\) and \(\phi_j:\calM\to\bbR\) are affine for each \(j=l+1,\dots,m\).
Using a vector of multipliers \(\lambda\in\bbR^m\), the Lagrangian dual problem of~\eqref{eq:GeneralPrimalProblem} can be written as
\begin{equation}\label{eq:GeneralDualProblem}
    v^{\rmD}:=\sup_{\lambda\in\Lambda}\ \inf_{\mu\in C}\left\{\varphi_0(\mu)+\sum_{j=1}^m \lambda_j\varphi_j(\mu)\right\},
\end{equation}
where the admissible set for the multipliers is defined as \(\Lambda:=\{\lambda\in\bbR^m:\lambda_j\ge0,\,\forall\,j=1,\dots,l\}\).
We want to show the strong duality between~\eqref{eq:GeneralPrimalProblem} and~\eqref{eq:GeneralDualProblem}, given the following condition.
\begin{definition}\label{def:SlaterCondition}
    We say that the problem~\eqref{eq:GeneralPrimalProblem} satisfies the (generalized) Slater condition if the point \(\eta=0\) is in the relative interior of the effective domain of the convex value function associated with the primal problem~\eqref{eq:GeneralPrimalProblem}
    \[v(\eta):=\inf_{\mu\in C}\left\{\varphi_0(\mu):\varphi_j(\mu)=\eta_j,\,j=1,\dots,l,\text{ and }\varphi_j(\mu)\le \eta_j,\,j=l+1,\dots,m\right\},\quad \eta\in\bbR^m.\]
\end{definition}

Recall that the effective domain of a convex function \(v:\bbR^m\to\bbR\cup\{\pm\infty\}\) is defined as \(\dom{v}:=\{\eta\in\bbR^m:v(\eta)<+\infty\}\), which is clearly a convex set.
The affine hull of a convex set \(K\subset\bbR^m\) is defined to be the smallest affine space containing \(K\), and the relative interior of \(K\) is the interior of \(K\) viewed as a subset of its affine hull (equipped with the subspace topology).
By convention, we have \(v(\eta)=+\infty\) if there is no \(\mu\in C\) such that \(\varphi_j(\mu)\le \eta_j\) for all \(j=1,\dots,m\).

\begin{proposition}\label{prop:StrongDuality}
    Assuming the Slater condition, the strong duality holds \(v^{\rmP}=v^{\rmD}\) with an optimal dual solution \(\lambda^*\ge0\) (i.e., the supremum in the dual problem~\eqref{eq:GeneralDualProblem} is attained).
\end{proposition}
\begin{proof}
    The weak duality \(v^{\rmP}\ge v^{\rmD}\) holds with a standard argument of exchanging the \(\inf\) and \(\sup\) operators, so it suffices to show that \(v^{\rmP}\le v^{\rmD}\).
    If \(v^{\rmP}=-\infty\) then the inequality holds trivially, so we assume that \(v^{\rmP}>-\infty\).
    Given the Slater condition, the value function \(v(\eta)\) of the primal problem~\eqref{eq:GeneralPrimalProblem} must be proper \(v(\eta)>-\infty\) for all \(\eta\in\bbR^m\) (ref.\ Theorem 7.2 in~\cite{rockafellar1970convex}) because \(\eta=0\) is in the relative interior of the effective domain of \(v\) and \(v(0)>-\infty\).
    Thus it is also subdifferentiable at the point \(\eta=0\) (ref.\ Theorem 23.4 in~\cite{rockafellar1970convex}), i.e., there exists a subgradient vector \(\lambda^*\in\bbR^m\) such that \(v(\eta)\ge v(0)-(\lambda^*)^\transpose \eta\) for any \(\eta\in\bbR^m\).
    Here, for each \(j=1,\dots,l\), the multiplier \(\lambda_j^*\) must be nonnegative since the function \(v(\eta)\) is not increasing in the \(j\)-th component, so we have \(\lambda\in\Lambda\).
    Since the inequality \(v(\eta)+(\lambda^*)^\transpose \eta\ge v(0)=v^{\rmP}\) holds for any \(\eta\in\bbR^m\), we have
    \begin{align*}
    v^{\rmP}&\le \inf_{\eta\in\bbR^m}\{v(\eta)+(\lambda^*)^\transpose \eta\}\\
    &=\inf_{\mu\in C}\inf_{\eta\in\bbR^m}\left\{\varphi_0(\mu)+\sum_{j=1}^m\lambda_j^*\eta_j:\varphi_j(\mu)\le \eta_j,\,j=1,\dots,l,\;\varphi_j(\mu)= \eta_j,\,j=l+1,\dots,m\right\}\\
    &=\inf_{\mu\in C}\left\{\varphi_0(\mu)+\sum_{j=1}^m\lambda_j^*\varphi_j(\mu)\right\}
    \le v^{\rmD}.
    \end{align*}
    The first equality here results from exchanging two infimum operators, while the second one follows by taking \(\eta_j=\varphi_j(\mu)\), due to the nonnegativity of \(\lambda^*_j\) for each \(j=1,\dots,l\), and replacing \(\eta_j\) with \(\varphi_j(\mu)\) for each \(j=l+1,\dots,m\).
\end{proof}

The strong Lagrangian duality guaranteed by the Slater condition is useful for many applications because we do not have to specify the topology on the vector space \(\calM\).
The corollary below summarizes a special case where there is no equality constraint, and all the inequality constraints can be strictly satisfied.
\begin{corollary}\label{cor:StrongDualityStrictInequality}
    For problems~\eqref{eq:GeneralPrimalProblem} and~\eqref{eq:GeneralDualProblem} with \(l=m\) (no equality constraints), the strong duality holds if there exists a point \(\bar{\mu}\in C\) such that \(\varphi_j(\bar{\mu})<0\), for each \(j=1,\dots,m\).
\end{corollary}
\begin{proof}
    Let \(\epsilon_j:=-\varphi_j(\mu)>0\) for \(j=1,\dots,m\), and \(U:=\prod_{j=1}^m(-\frac{\epsilon_j}{2},\frac{\epsilon_j}{2})\subset\bbR^m\) be an open hyperrectangle.
    Then for any \(\eta\in U\), we have \(v(\eta)\le\varphi_0(\bar{\mu})<+\infty\).
    Therefore, we know that the Slater condition holds because \(0\in U\subset\dom{v}\), and the result follows from Proposition~\ref{prop:StrongDuality}.
\end{proof}

Assuming that \(\calM\) is normed and complete, we could have another generalized Slater condition for problems with equality constraints.
However, it is in general much harder to check so we do not base our discussion in Section~\ref{sec:Model} on it.

\subsection{Finite-dimensional Dual Recursion}
Now we derive the finite-dimensional dual recursion for the DR-MCO~\eqref{eq:MDRO-RecursiveForm}.
Recall that by the definition of Wasserstein distance~\eqref{eq:WassersteinDistance}, the constraint \(W_t(p,\hat{\nu}_t)\le \rho_{t,0}\) is ensured if there exists a probability measure on the product space \(\pi_t\in\calM^{\Prob}(\Xi_t\times\Xi_t)\) with marginal probability measures \(P^1_*(\pi_t)=p\) and \(P^2_*(\pi_t)=\hat{\nu}_t\), such that
\begin{align}
    \rho_{t,0}
    \ge\int_{\Xi_t\times\Xi_t}d_t\dif\pi_t
    =\int_{\Xi_t}\int_{\Xi_t}d_{t}(\xi,\xi')\dif\pi_t(\xi\vert\xi')\dif\hat{\nu}_t(\xi')
    =\frac{1}{n_t}\sum_{k=1}^{n_t}\int_{\Xi_t}d_{t,k}(\xi)\dif p_{t,k}(\xi)
\end{align}
where we define \(p_{t,k}:=\pi_t(\cdot\vert\hat{\xi}_{t,k})\in\calW_t\) to be the probability measure conditioned on \(\{\xi'=\hat{\xi}_{t,k}\}\) and \(d_{t,k}(\xi):=d(\xi,\hat{\xi}_{t,k})\) for any \(\xi\in\Xi_t\). 
Then we have \(p=\frac{1}{n_t}\sum_{k=1}^{n_t}p_{t,k}\) by the law of total probability.
Due to this condition, we define the following parametrized optimization problem for any \(\rho>0\) and any state \(x_{t-1}\in\calX_{t-1}\)
\begin{align}\label{eq:ParametrizedRecursion}
    q_t(x_{t-1};\rho):=\sup_{p_{t,k}\in\calW_t}\quad
    &\frac{1}{n_t}\sum_{k=1}^{n_t}\int_{\Xi_t}Q_t(x_{t-1};\xi_t)\dif p_{t,k}(\xi_t)\\
    \mathrm{s.t.}\quad
    &\frac{1}{n_t}\sum_{k=1}^{n_t}\int_{\Xi_t}d_{t,k}\dif{p_{t,k}}\le\rho,\notag\\
    &\frac{1}{n_t}\sum_{k=1}^{n_t}\int_{\Xi_t}g_{t,j}\dif{p_{t,k}}\le\rho_{t,j}, \quad j=1,\dots,m_t.\notag
\end{align}
From the discussion above, we see that \(\calQ_{t-1}(x_{t-1})\ge q_t(x_{t-1};\rho_{t,0})\).
We next show that the equality holds assuming the strict feasibility of the empirical measure \(\hat{\nu}_t\).

Let \(\lambda\in\bbR^{m_t+1}_{\ge0}\) denote a multiplier vector.
Assumption~\ref{assum:NominalStrictFeasibility} guarantees by Corollary~\ref{cor:StrongDualityStrictInequality} the strong duality of the following Lagrangian dual problem
\begin{align}
    q_t(x_{t-1};\rho)
    &=\min_{\lambda\ge0}\left\{\rho\lambda_0+\sum_{j=1}^{m_t}\rho_{t,j}\lambda_j+\frac{1}{n_t}\sum_{k=1}^{n_t}\sup_{p_{t,k}\in\calW_t}\int_{\Xi_t}\bigg[Q_{t}(x_{t-1};\cdot)-\lambda_0 d_{t,k}-\sum_{j=1}^{m_t}\lambda_j g_{t,j}\bigg]\dif{p_{t,k}}\right\}\notag\\
    &=\min_{\lambda\ge0}\left\{\rho\lambda_0+\sum_{j=1}^{m_t}\rho_{t,j}\lambda_j+\frac{1}{n_t}\sum_{k=1}^{n_t}\sup_{\xi_{k}\in\Xi_t}\Bigg\{Q_{t}(x_{t-1};\xi_k)-\lambda_0 d_{t,k}(\xi_k)-\sum_{j=1}^{m_t}\lambda_j g_{t,j}(\xi_k)\Bigg\}\right\}.\label{eq:ParametrizedRecursionDual}
\end{align}
where the second equality holds because each Dirac measure centered at \(\xi\in\Xi_t\) satisfies \(\delta_{\xi}\in\calW_t\) and each \(p_{t,k}\in\calW_t\) is a probability measure.
We are now ready to show that \(\calQ_{t-1}(x_{t-1})=q_t(x_{t-1};\rho_{t,0})\), which implies Theorem~\ref{thm:FiniteDimensionalDual}.

\begin{theorem}\label{thm:FiniteDimensionalDual-Extended}
Under Assumption~\ref{assum:NominalStrictFeasibility}, in any stage \(t\ge 2\), we have \(q_t(x_{t-1};\rho)\) is a concave function in \(\rho\) for any fixed \(x_{t-1}\).
Consequently, the expected cost-to-go function~\eqref{eq:MDRO-RecursiveForm} satisfies \(\calQ_{t-1}(x_{t-1})=q_t(x_{t-1};\rho_{t,0})\) and thus can be equivalently rewritten as
\begin{equation*}
    \calQ_{t-1}(x_{t-1})=\min_{\lambda\ge0}\left\{\sum_{j=0}^{m_t}\rho_{t,j}\lambda_j+\frac{1}{n_t}\sum_{k=1}^{n_t}\sup_{\xi_k\in\Xi_t}\left\{Q_{t}(x_{t-1};\xi_k)-\lambda_0 d_{t,k}(\xi_k)-\sum_{j=1}^{m_t}\lambda_j g_{t,j}(\xi_k)\right\}\right\}.
\end{equation*}
\end{theorem}
\begin{proof}
    Let us fix any \(x_{t-1}\in\calX_{t-1}\).
    The first assertion on the concavity of \(q_t(x_{t-1};\rho)\) follows directly from~\eqref{eq:ParametrizedRecursionDual} since \(q_t(x_{t-1};\cdot)\) is a minimum of affine functions.
    Moreover, from the definition~\eqref{eq:ParametrizedRecursion} we see that \(q_t(x_{t-1};\rho)\ge 0\) for any \(\rho>0\) as the measures \(p_{t,k}=\delta_{\hat{\xi}_{t,k}}\) satisfy the constraints and \(Q_t(x_{t-1};\xi_t)\ge 0\) by the nonnegativity of the cost functions \(f_t\).
    If \(q_t(x_{t-1};\rho_{t,0})=+\infty\), then the equality holds trivially as we already showed that \(\calQ_{t-1}(x_{t-1})\ge q_t(x_{t-1};\rho_{t,0})\).
    Otherwise, we must have \(q_t(x_{t-1};\rho)<+\infty\) for any \(\rho>0\) due to the concavity.
    Thus \(q_t(x_{t-1};\cdot)\) is a continuous function on \((0,+\infty)\).

    To prove the inequality \(\calQ_{t-1}(x_{t-1})\le q_t(x_{t-1};\rho_{t,0})\), take any \(\epsilon>0\).
    From the definition~\eqref{eq:WassersteinDistance}, the constraint \(W_t(p,\hat{\nu}_t)\le\rho_{t,0}\) implies that there exists a probability measure \(\pi_t\in\calM^{\Prob}(\Xi_t\times\Xi_t)\) with marginal probability measures \(P^1_*(\pi_t)=p\) and \(P^2_*(\pi_t)=\hat{\nu}_t\) such that \(\int d_t\dif{\pi_t}\le\rho_{t,0}+\epsilon\).
    In other words, we have \(\calQ_{t-1}(x_{t-1})\le q_t(x_{t-1};\rho_{t,0}+\epsilon)\).
    Now by the continuity of \(q_t(x_{t-1};\cdot)\), we conclude that \(\calQ_{t-1}(x_{t-1})\le\lim_{\epsilon\to 0+}q_t(x_{t-1};\rho_{t,0}+\epsilon)=q_t(x_{t-1};\rho_{t,0})\).
\end{proof}

\section{Supplemental Numerical Results}
\label{sec:SupplementalResults}

In this section, we display supplemental results from our numerical experiments in Section~\ref{sec:Numerical}.

\begin{figure}[htbp]
    \centering
    \includegraphics[width=1.0\textwidth]{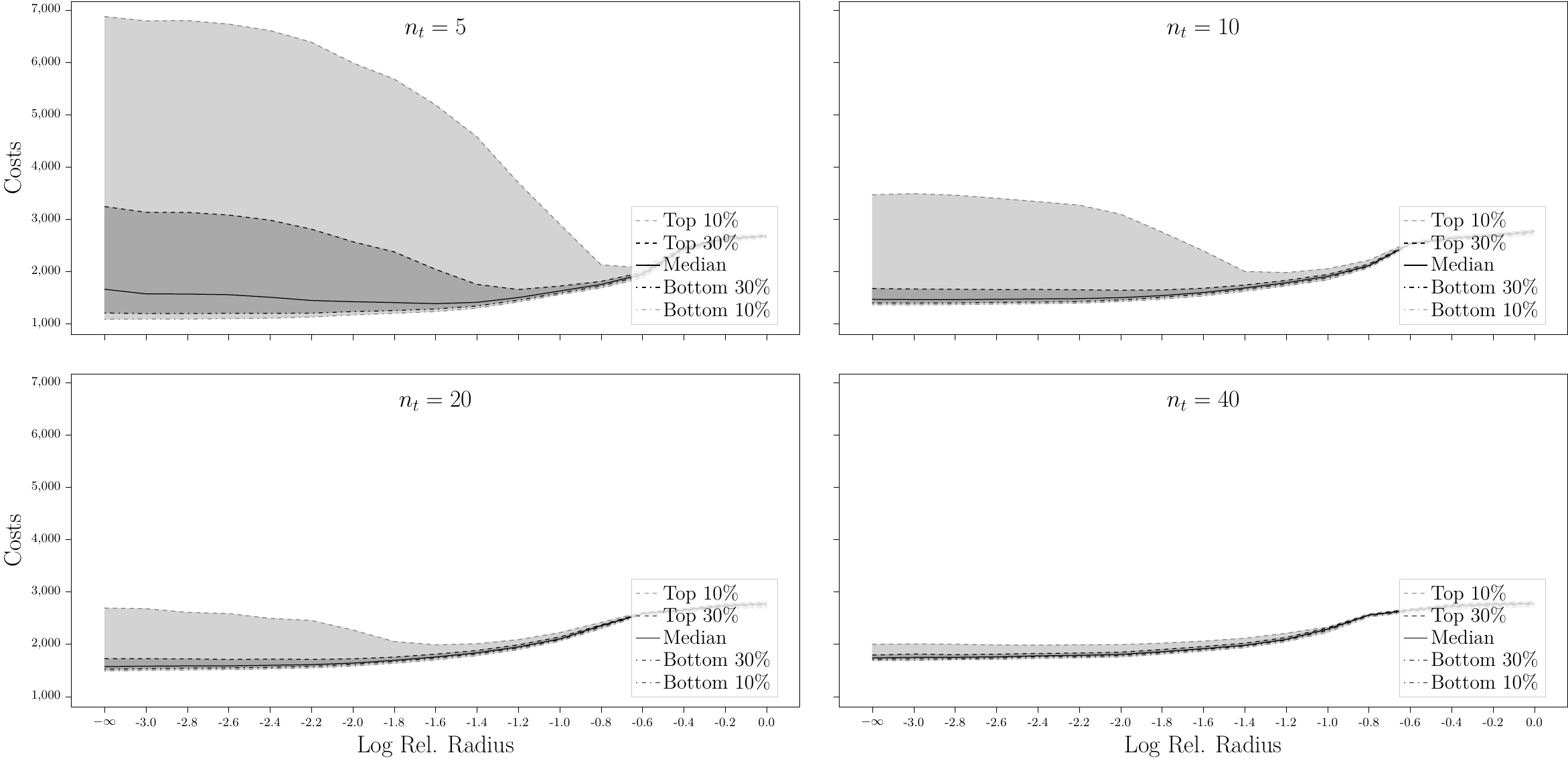}
    \caption{Out-of-sample Cost Quantiles for Different Radii on Multi-commodity Inventory with Uncertain Demands, Additional Run 1}
    \label{fig:DemandLinePlot1}
\end{figure}
\begin{figure}[htbp]
    \centering
    \includegraphics[width=1.0\textwidth]{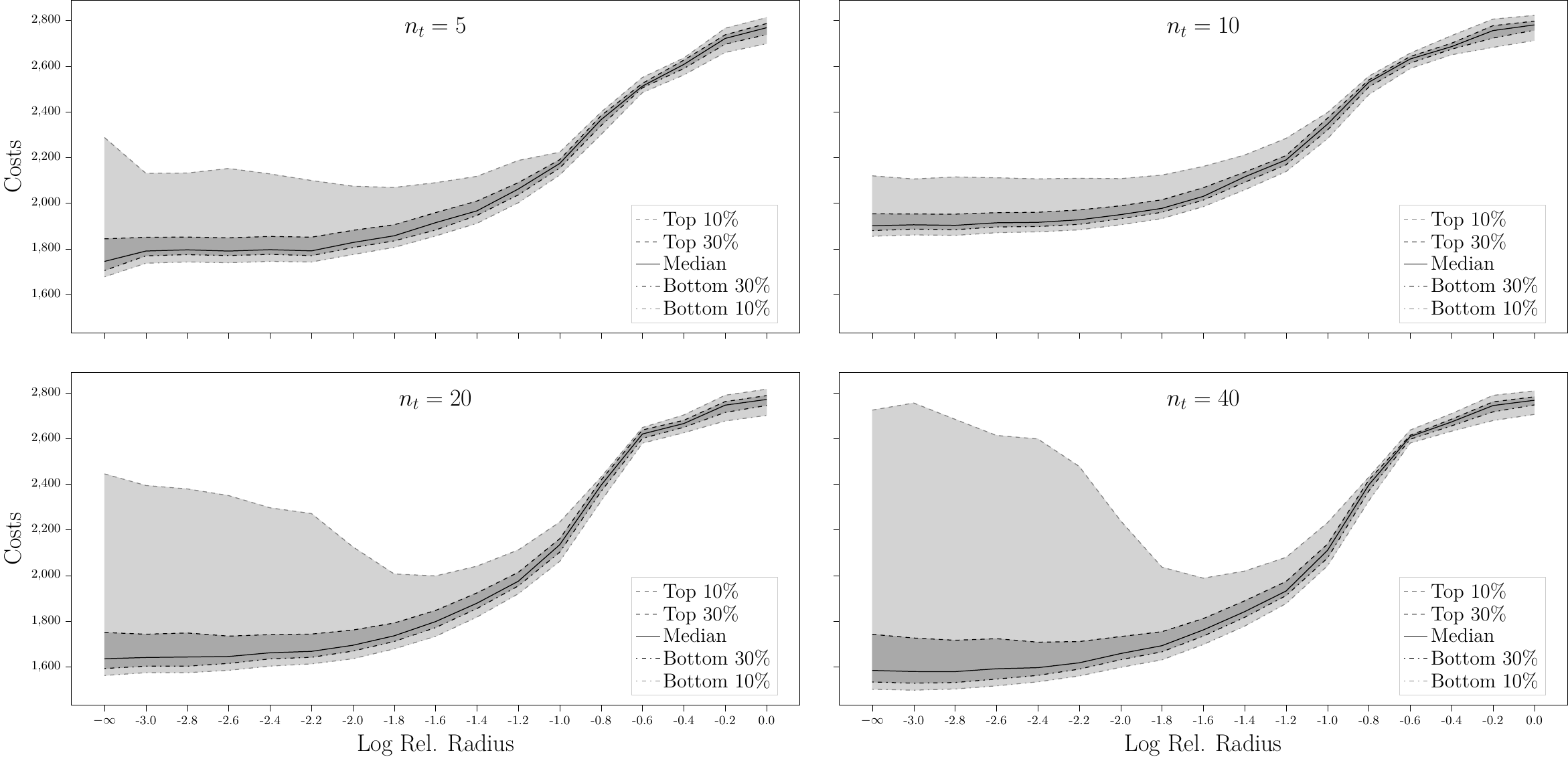}
    \caption{Out-of-sample Cost Quantiles for Different Radii on Multi-commodity Inventory with Uncertain Demands, Additional Run 2}
    \label{fig:DemandLinePlot2}
\end{figure}

\begin{figure}[htbp]
    \centering
    \includegraphics[width=1.0\textwidth]{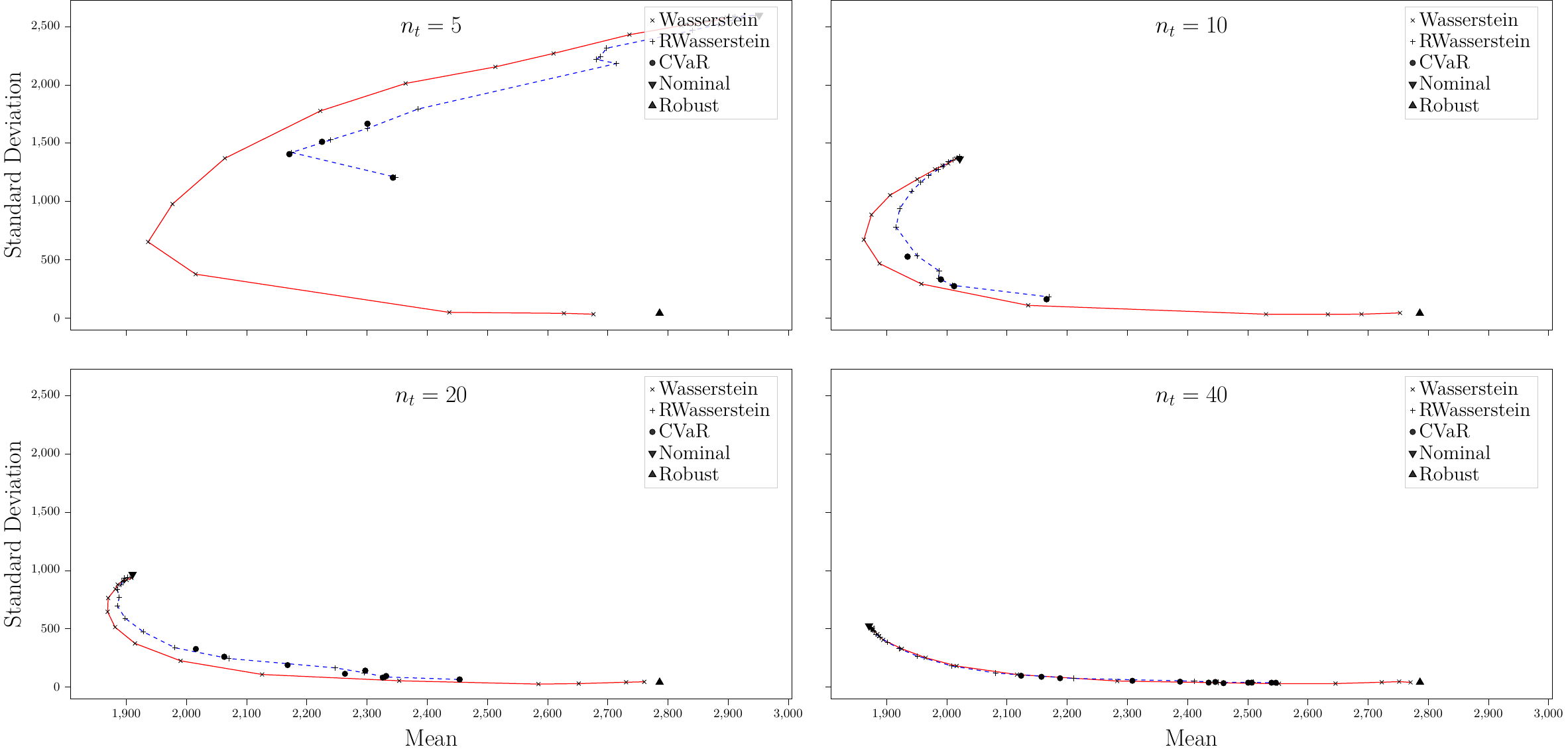}
    \caption{Comparison against Baseline Models on Multi-commodity Inventory with Uncertain Demands, Additional Run 1}
    \label{fig:DemandScatterPlot1}
\end{figure}
\begin{figure}[htbp]
    \centering
    \includegraphics[width=1.0\textwidth]{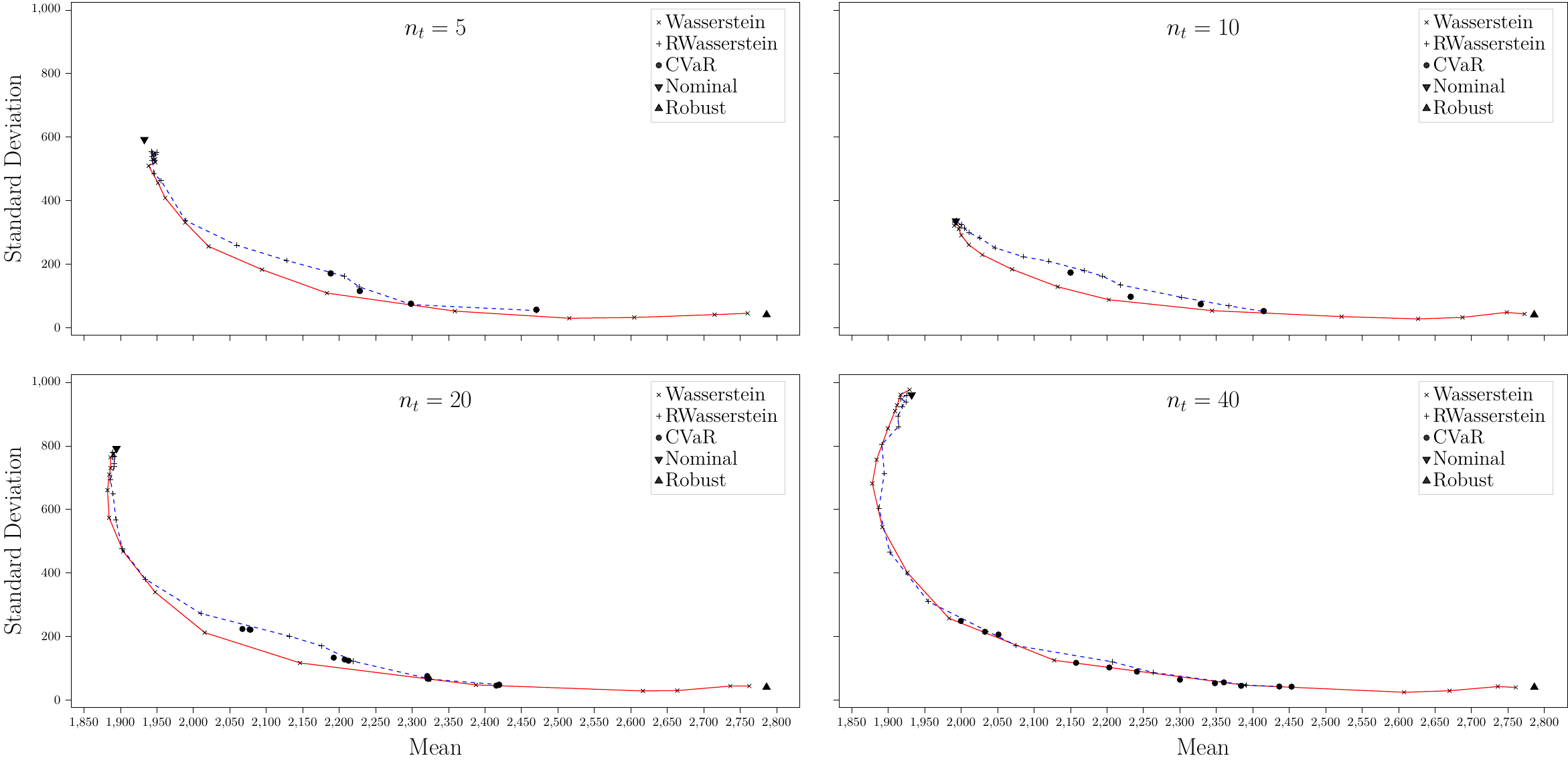}
    \caption{Comparison against Baseline Models on Multi-commodity Inventory with Uncertain Demands, Additional Run 2}
    \label{fig:DemandScatterPlot2}
\end{figure}

\begin{figure}[htbp]
    \centering
    \includegraphics[width=1.0\textwidth]{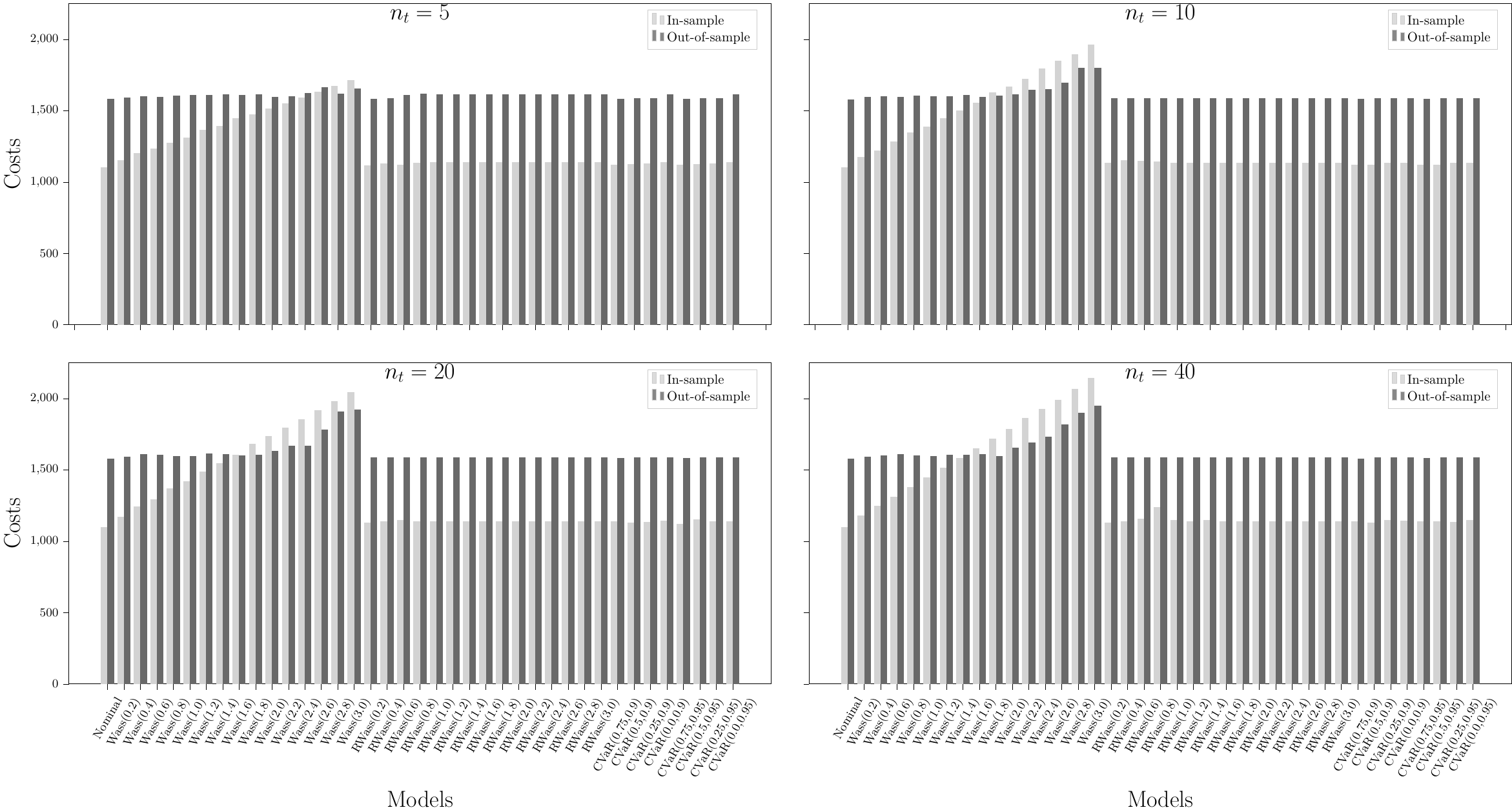}
    \caption{In-sample and Out-of-sample Mean Costs on Multi-commodity Inventory with Uncertain Prices, Additional Run 1}
    \label{fig:SupplyPlot1}
\end{figure}
\begin{figure}[htbp]
    \centering
    \includegraphics[width=1.0\textwidth]{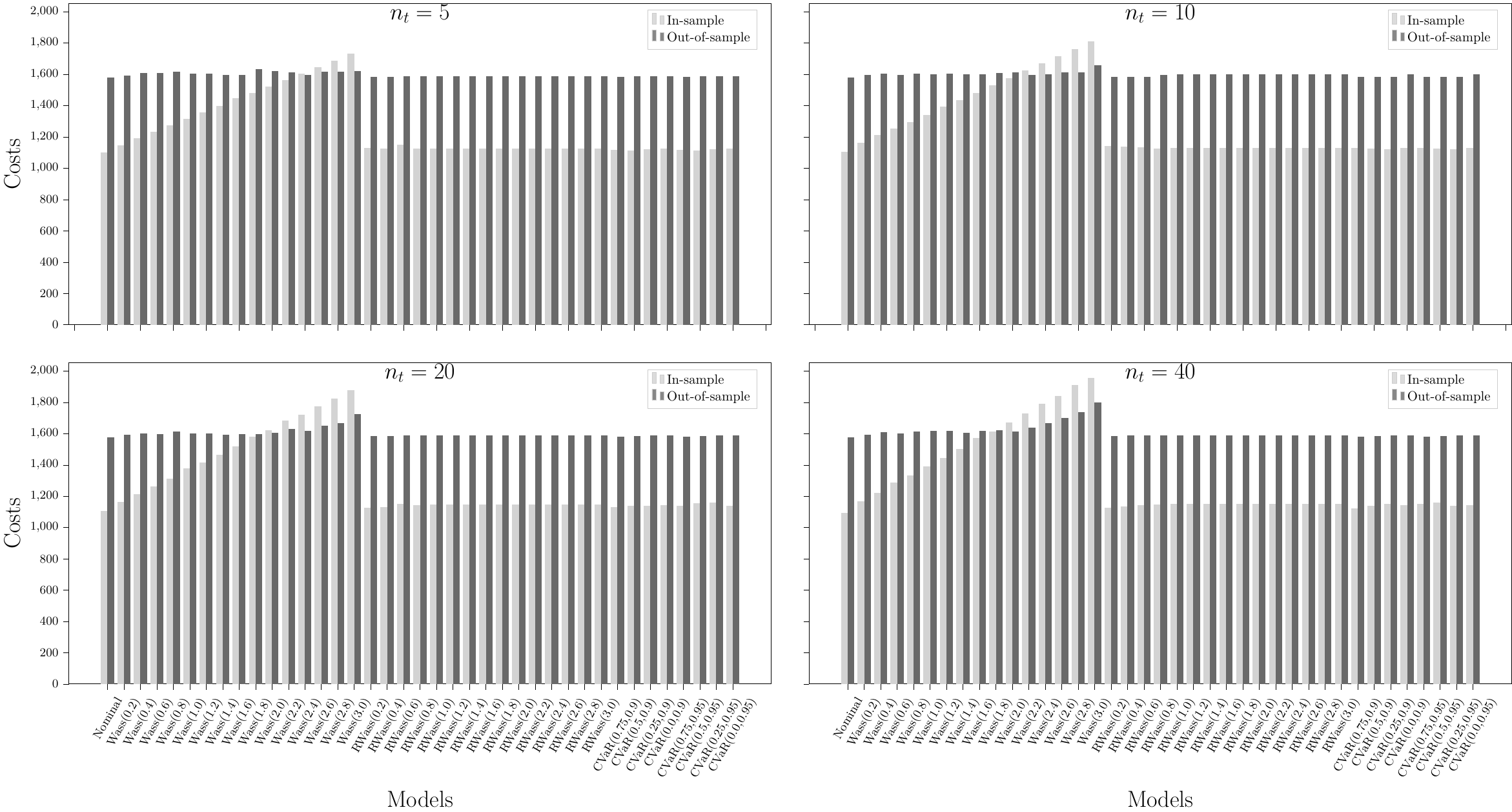}
    \caption{In-sample and Out-of-sample Mean Costs on Multi-commodity Inventory with Uncertain Prices, Additional Run 2}
    \label{fig:SupplyPlot2}
\end{figure}

\begin{figure}[htbp]
    \centering
    \includegraphics[width=1.0\textwidth]{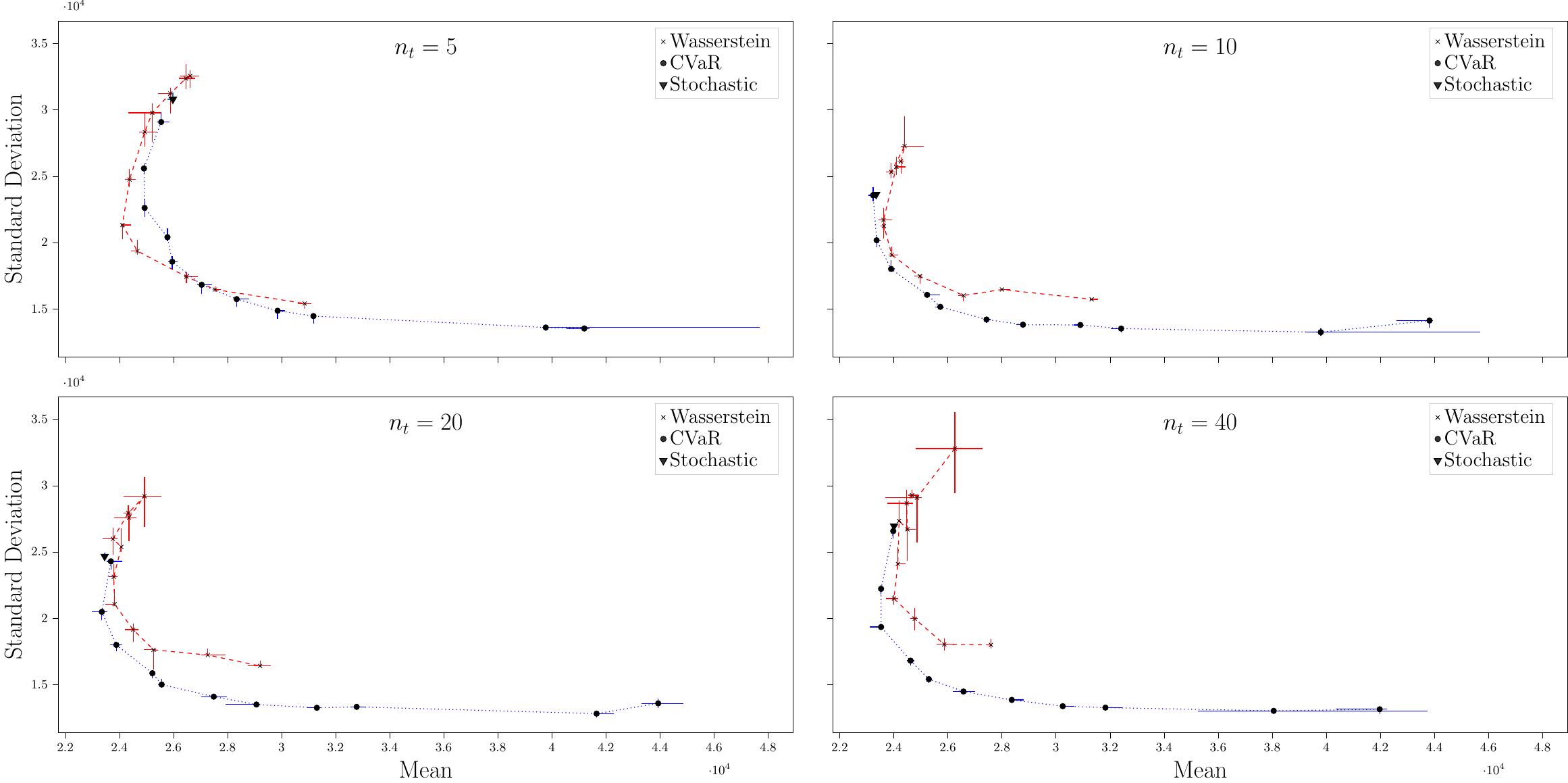}
    \caption{Comparison against Baseline Models on Hydro-thermal Power Planning, Additional Run 1}
    \label{fig:HydroPlot1}
\end{figure}
\begin{figure}[htbp]
    \centering
    \includegraphics[width=1.0\textwidth]{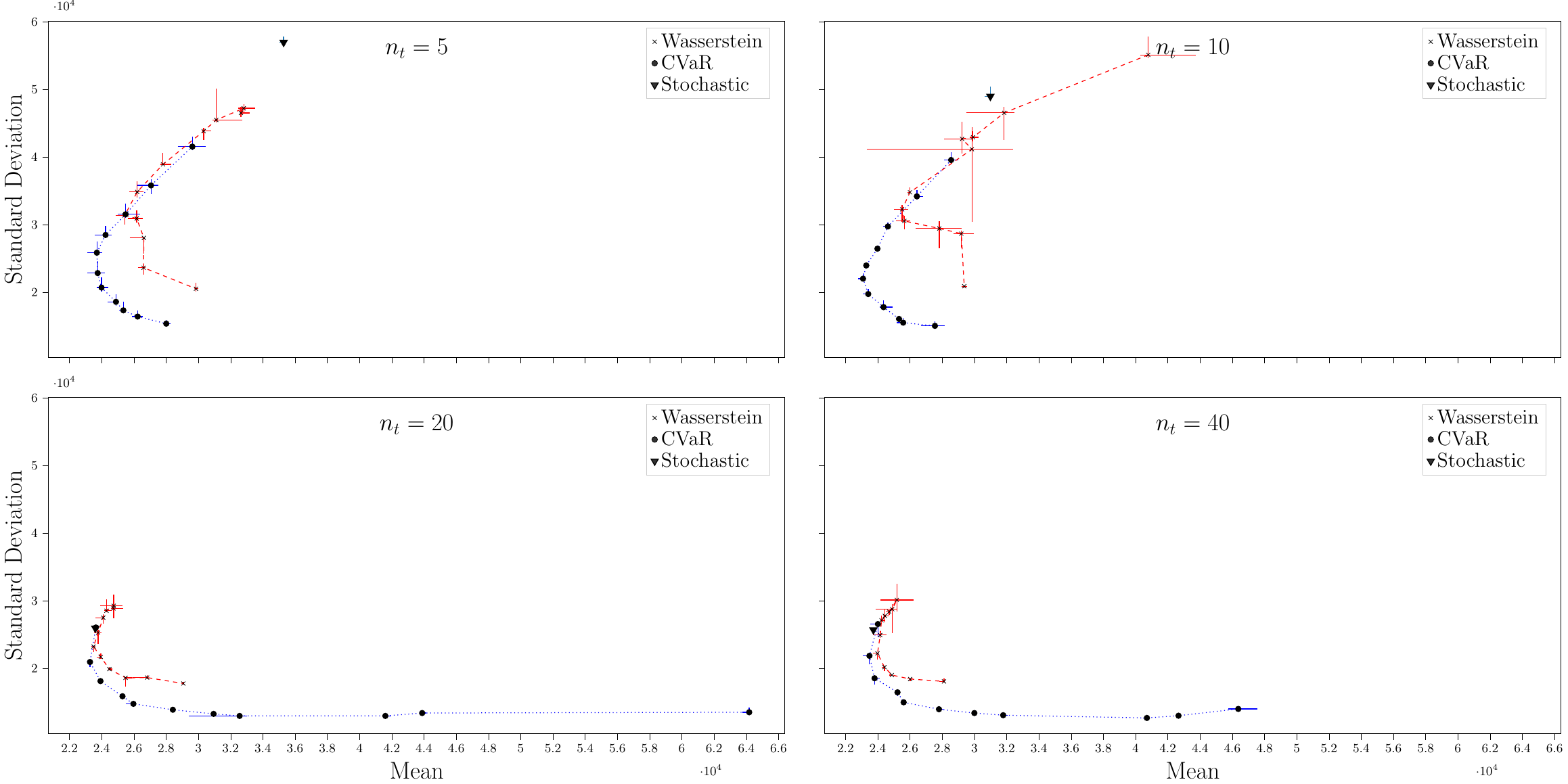}
    \caption{Comparison against Baseline Models on Hydro-thermal Power Planning, Additional Run 2}
    \label{fig:HydroPlot2}
\end{figure}

\end{document}